\documentclass[a4paper, 12pt]{amsart} %(needed for arXiv)

\usepackage[pdftex]{graphicx} %for arXiv

\usepackage{amssymb, amsmath, amsthm}
\usepackage{amscd}

\usepackage{comment}
\usepackage[all]{xy}
\usepackage{color}

\usepackage[utf8]{inputenc}
\usepackage{tikz}
\usetikzlibrary{positioning}
\usetikzlibrary{intersections}
\usetikzlibrary{calc, quotes, angles}

\begin{document}

\title
[Collapsing Alexandrov three-spaces with boundary]
{Collapsing three-dimensional Alexandrov spaces with boundary} %%%and with a lower curvature bound}
 
\author[A. Mitsuishi]{Ayato Mitsuishi}
\author[T.Yamaguchi]{Takao Yamaguchi}

\address{Ayato Mitsuishi, Department of Applied Mathematics, Fukuoka University, Jyonan-ku, Fukuoka, Fukuoka 814--0180, JAPAN}
\email{mitsuishi@fukuoka-u.ac.jp}

\address{Takao Yamaguchi, Department of mathematics, Kyoto University, Kitashirakawa, Kyoto 606--8502, JAPAN}
\curraddr{Institute of Mathematics, University of Tsukuba, Tsukuba 305-8571, Japan}
 \email{takao@math.tsukuba.ac.jp}


\thanks{This work was supported by JSPS KAKENHI Grant Numbers 17H01091,
18H01118, 15H05739, 26287010}

\date{\today}

\theoremstyle{plain}
\newtheorem{thm}{Theorem}[section]
\newtheorem{lem}[thm]{Lemma}
\newtheorem{cor}[thm]{Corollary}
\newtheorem{prop}[thm]{Proposition}
\newtheorem{defn}[thm]{Definition}
\newtheorem{rem}[thm]{Remark}
\newtheorem{ex}[thm]{Example}
\newtheorem{fact}[thm]{Fact}
\newtheorem{claim}[thm]{Claim}

\newcommand{\diam}[0]{\mathrm{diam}}
\newcommand{\gh}[0]{\xrightarrow{\mathrm{GH}}}
\newcommand{\ba}[0]{\partial^{\bf A}} %%%boundary of Alex sp
\newcommand{\pa}[0]{\partial} %%%topological boundary
\newcommand{\bm}[0]{\partial^{\mathrm{met}}} %%%metric boundary
\newcommand{\reg}[0]{\mathrm{reg}}
\newcommand{\Mo}[0]{\mathrm{M\ddot{o}}}
\newcommand{\e}[0]{\epsilon}

\newcommand{\R}[0]{\mathbb R}
\newcommand{\C}[0]{\mathbb C}
\newcommand{\Q}[0]{\mathbb Q}
\newcommand{\Z}[0]{\mathbb Z}
\newcommand{\N}[0]{\mathbb N}

\newcommand{\ca}[0]{\mathcal}
\newcommand{\pmed}[0]{\par\medskip}
\newcommand{\pbig}[0]{\par\bigskip}
\newcommand{\psmall}[0]{\par\smallskip}
\newcommand{\n}[0]{\noindent}

\newcommand{\cbl}{\color{blue}}
\newcommand{\cred}{\color{red}}

\newcommand{\beq}[0]{\begin{equation}}
\newcommand{\eeq}[0]{\end{equation}}

\newcommand{\beqq}[0]{\begin{equation*}}
\newcommand{\eeqq}[0]{\end{equation*}}

\newcommand{\bali}[0]{\begin{align}}
\newcommand{\eali}[0]{\end{align}}

\newcommand{\benum}[0]{\begin{enumerate}}
\newcommand{\eenum}[0]{\end{enumerate}}

\makeatletter
\renewcommand{\theequation}{%%%
 \thesection.\arabic{equation}}
 \@addtoreset{equation}{section}
\makeatother

%\begin{document}

\begin{abstract}
As a continuation of \cite{MY}, 
we determine the topologies of collapsing three-dimensional compact Alexandrov spaces with nonempty boundary.
\end{abstract}
\maketitle

\setcounter{tocdepth}{1}

\tableofcontents

\section{Introduction} \label{sec:intro}
This paper is a continuation of the paper \cite{MY}.
In \cite{MY}, we classified the topologies of collapsing three-dimensional Alexandrov spaces with no boundary and with 
bounded diameters.
See also \cite{GGZ} for related results.
%the classification of 
%collapsing three-dimensional Alexandrov spaces with no boundary and with unbounded diameters.
In the present paper, we determine the topologies of collapsing three-dimensional compact Alexandrov spaces with non-empty boundary.

We consider the precompact family $\ca M(3, D)$ consisting of all three-dimensional compact Alexandrov spaces $M$ with boundary satisfying
\[
\text{curvature $\ge -1$, \quad $\diam(M)\le D$.}
\]
Let $M_i$ be a sequence in $\ca M(3, D)$ converging to a compact metric space $X$
with respect to the Gromov-Hausdorff distance, 
where $X$ is an Alexandrov space with curvature $\ge -1$ and dimension $\le 3$.
Our main problem is to find topological relations between $M_i$ and $X$ 
for sufficiently large $i$.
In view of the Perelman stability theorem (
\cite{PerAlex2}, \cite{Per:elem} cf. \cite{Kap:stab}), we may assume $0\le \dim X\le 2$.
We always assume that $i$ is large enough.
Our result can be summarized as 

\begin{thm} \label{thm:dim012}
If $0\le \dim X\le 2$, then
there is a singular fibration $f_i : M_i \to X$.
\end{thm}

From the construction of $f_i$ in Theorem \ref{thm:dim012}, it will become clear that $f_i : M_i \to X$ also provides a 
$\theta_i$-approximation with 
$\lim_{i\to\infty} \theta_i=0$, and has the property of
almost Lipschitz submersion over a regular part of 
$X$.
Therefore for simplicity we omit those statements
in our main results.

From now, we restate our result more explicitly in each case
depending on $\dim X$ and the presence of the boundary of $X$.
First we consider the most basic case that $X$ is two-dimensional and having 
no boundary. 
We denote by $\ba M$ the boundary of an 
Alexandrov space $M$.

\begin{thm} \label{thm:dim2nobdry}
If $X$ is two-dimensional and having no boundary, then
there is a generalized $I$-bundle $f_i : M_i \to X$.

In particular, we have
\begin{enumerate}
 \item $f_i$ may have singular $I$-fibers over some essential singular point set $\ca S=\{ x_\alpha\}$ of $X\,;$
 \item each singular fiber $f_i^{-1}(x_\alpha)$ is an arc joining a topological singular point in $M_i\setminus \ba M_i$ 
and a point of $\ba M_i\,;$
 \item any regular $I$-fiber of $f_i$ is an arc joining two points of $\ba M_i\;$
\item the restriction $f_i|_{\ba M_i}:\ba M_i\to X$ is 
a branched $\Z_2$-covering branched at $\ca S$.
\end{enumerate}
\end{thm}

For the definition of a generalized $I$-bundle, 
see Definition \ref{def:gen I-bdl}.

\begin{rem} \label{rem:diff-closed1} \upshape
In the case when a three-dimensional closed Alexandrov space $M_i$ collapses to a two-dimensional space $X$ {\it without boundary},
we have a generalized Seifert bundle $M_i \to X$, where
a singular $I$-fiber occurs as an arc joining two topological singular 
points of $M_i$ (\cite[Theorem 1.3]{MY}). In addition, a singular $S^1$-fiber may occurs 
over an essential singular point of $X$. In the situation of Theorem \ref
{thm:dim2nobdry}, if $M_i$ has no topological singular point, then 
it is an $I$-bundle over $X$.
\end{rem}

For an Alexandrov space $M$ with boundary, let ${\rm inrad}(M)$ denote the 
{\it inradius} of $M$, which is defined as 
$$
 \displaystyle{{\rm inrad}(M) :=\sup_{x\in M}\, |x,\ba M|}.
$$
Theorem \ref{thm:dim2nobdry} shows that if $X$ is two-dimensional and having no boundary, then $\lim_{i\to\infty} {\rm inrad}(M_i) =0$. That is,
$M_i$ {\it inradius collapses} to $X$.
% in the sense of 
%\cite{YZ}.
See \cite{YZ} for the general discussion on inradius collapsed manifolds whose second fundamental forms of 
boundaries are uniformly
bounded.

Next we consider the case that the limit space $X$ is two-dimensional and having 
nonempty boundary.
%%%%%%
Passing to a subsequence, we have one of the following two cases. 
\begin{itemize}
\item[(1)] $\mathrm{inrad}(M_i) \to 0$. That is, 
$M_i$ inradius collapses to $X\,;$
\item[(2)] $\mathrm{inrad}(M_i)$ is uniformly bounded 
away from $0$. That is, 
$M_i$ non-inradius collapses to $X$.
\end{itemize}

\begin{thm} \label{thm:dim2wbdy}
If $X$ is two-dimensional and having 
nonempty boundary, then the topology of $M_i$ is determined as follows.
\begin{enumerate}
\item 
If $M_i$ inradius collapses to $X$, then 
there is a generalized $I$-bundle 
$f_i : M_i \to X\,;$
\item If $M_i$ non-inradius collapses to $X$, then 
there is a generalized 
Seifert fibration $f_i : M_i \to X$. In particular, we have 
 \begin{enumerate}
 \item $f_i$ sends each component $\ba_\alpha M_i$ of $\ba M_i$ to 
 an extremal subset $E_\alpha$ of $X$ contained in $\ba X\,;$
 \item $\ba_\alpha M_i$ is homeomorphic to one of $S^2$, $P^2$ $T^2$ and $K^2$. 
\end{enumerate}
\end{enumerate}
\end{thm}

For the definition of a generalized 
Seifert fibration, see Definition \ref{def:gen S1-bdl}.

\begin{rem} \label{rem:hosoku-Thm2} \upshape
It should be noted that in Theorem \ref{thm:dim2wbdy}(1), $f_i$ has no singular $I$-fibers over $\ba X$.
\end{rem} 

Next we assume $\dim X=1$, that is, $X$ is isometric to a circle, a closed interval.

\begin{thm} \label{thm:main-circle}
If $X$ is a circle, then $M_i$ is homeomorphic to the total space of a fiber bundle over the circle whose fiber is one of $D^2$, $S^1 \times I$ and 
a M\"obius band $\Mo$.
\end{thm}

In the case that $X$ is either a closed interval or a point, we have 
explicit topological relations between $M_i$ and $X$ in detail 
(see Theorems \ref{thm:interval} and \ref{thm:dim0}).
This can be summarized as

\begin{thm} \label{thm:dim01}
If $\dim X\le 1$, then $M_i$ is homeomorphic to a three-dimensional 
Alexandrov space of nonnegative curvature with boundary.
\end{thm}

\begin{rem} \upshape
\begin{enumerate}
\item When $X$ is a circle, Theorem \ref{thm:dim01} follows from 
Theorem \ref{thm:main-circle} $;$
\item Conversely, if each $M_i$ has nonnegative curvature, then
it can collapse to a point by rescaling of metrics.
\end{enumerate}
\end{rem}

As a result of Theorem \ref{thm:dim01}, we immediately have the following
gap phenomenon.

\begin{cor} \label{cor:gap}
There is a uniform positive constant $\e$ such that 
if a compact three-dimensional Alexandrov space $M$ with boundary satisfies
\[
\text{{\rm curvature} $\ge -\e$, \quad $\diam(M)\le 1$,}
\]
then it is homeomorphic to an Alexandrov space with nonnegative curvature.
\end{cor}

\begin{rem} \label{rem:diff-closed2} \upshape
Because of three-dimensional nilmanifolds, Corollary \ref{cor:gap} does not hold in the case when 
$M$ is closed. 
\end{rem}

By the present paper together with 
\cite{MY}, we have completed 
the detailed topological classification of three-dimensional Alexandrov spaces with curvature uniformly bounded below.

\psmall
\n
{\bf Organization of the paper}

In Sections \ref{sec:prelim}--\ref{sec:sing_fib}, we prepare some preliminary materials. 
 In Sections \ref{sec:prelim}, we prepare some basics of Alexandrov spaces.
In Section \ref{sec:prelim2}, we recall some basic results which are effectively used in the study of convergence or collapsing Alexandrov spaces.
In particular, in dimension three, we recall the classification results of 
three-dimensional Alexandrov spaces with nonnegative curvature.
In Section \ref{sec:sing_fib}, we provide the notions of generalized $I$-bundles and the generalized Seifert fibrations.
These are used in the later sections to describe the collapsing Alexandrov three-spaces with boundary.

In Sections \ref{section: 2-dim no boundary}--\ref{sec:point}, we prove the collapsing theorems 
stated above.
In Section \ref{section: 2-dim no boundary}, we prove Theorem \ref{thm:dim2nobdry}
by making use of the collapsing of the doubles.

In Section \ref{sec:2-dim with bdry}, we prove Theorem \ref{thm:dim2wbdy}
by considering the two cases, one is the inradius collapse and the other is the non-inradius collapse. 

Theorems \ref{thm:main-circle}, \ref{thm:dim01} are proved in Sections 
\ref{sec:circle}, \ref{sec:interval} and \ref{sec:point}.

In Appendix \ref{sec:flow}, we provide an 
equivariant flow theorem, which is needed when we decompose collapsed Alexandrov spaces into several closed domains.
\pmed\n 
{\bf Acknowledgements}. The authors would like to thank Makoto Sakuma and Ken'ichi Ohshika for the discussion on singular fibrations
from three-dimensional orbifolds to surfaces.

\section{Notations and conventions}

We fix some notations and terminologies used in the paper.

\begin{itemize}
\item $\mathbb R^n_+$ denotes the $n$-dimensional Euclidean half space,
 \[
 \R^n_+=\{ (x_1,\ldots, x_n)\in\R^n\,|\, x_n\ge 0\} \,;
\] 
\item $I$ denotes a closed interval;
\item $D^n=\{ x\in \R^n\,|\,||x||\le 1\}$ and $D^n_+=D^n\cap \R^n_+$\,;
\item $S^n$, $P^n$ and $D^n$ denote $n$-dimensional sphere, real projective space,
and closed disk respectively;
\item $\Mo$ and $K^2$ denote a M\"obius band and a Klein bottle;
\item $I_\ell$ denotes a closed interval of length $\ell$, 
and $S^1_{\ell}$ denotes the circle of length $\ell$ in $\C$ around the origin;
\item For a subset $K$ of a metric space $X$, let $U(K,r)$ (resp. $B(K,r)$) denote
the open (resp. the closed) metric ball around $K$ of radius $r$.
We also use $S(K,r)$ to denote the metric $r$-sphere around $A$;
$S(A,r)=\{ x\in X| |x, A|=r\}$. Sometimes, we write as $B^X(A,r)$ to 
emphasize that it is a ball in $X$;
\item
Let $A(K,r,R)$ denote the metric annulus around $K$ defined as $A(K,r,R)=B(K,R)\setminus U(K,r)$;
\item For a topological space $\Sigma$, let $K(\Sigma)$ denotes an open cone over $\Sigma$, that is, $K(\Sigma) = \Sigma \times [0,\infty) /\Sigma \times \{0\}$.
$K_1(\Sigma)$ denotes the unit cone:
$K_1(\Sigma) = \Sigma \times [0,1] /\Sigma \times \{0\}\,;$ 
\item For a subset $A$ of a topological space, $\mathring A$ denotes the interior of $A\,;$
\item For topological spaces $A, B$, $A \approx B$ means that $A$ and $B$ are homeomorphic. Furthermore, for $p \in A' \subset A$ and $q \in B' \subset B$, $(A,A',p) \approx (B,B',q)$ means that there exists a homeomorphism $f : A \to B$ such that $f(A') = B'$ and $f(p) = q$. 
\end{itemize}

Throughout the paper, we always consider finite-dimensional Alexandrov spaces with a lower curvature bound.

\section{Preliminaries I --- Basics of Alexandrov spaces} %%%
\label{sec:prelim}
Let us recall the definition of Alexandrov spaces and their properties. 
We refer \cite{BGP} and \cite{BBI} for the details. 

%A {\it geodesic} is an isometric embedding from an interval to a metric space. 
%A {\it ray} (resp. a {\it line}) is a geodesic with domain $[0, \infty)$ (resp. $\mathbb R$). 
%A metric space is called a {\it geodesic space} if every two points can be connected by a geodesic. 

For any $\kappa \in \mathbb R$, the simply connected complete surface of constant curvature $\kappa$ is called the {\it $\kappa$-plane}.
For three points $x,y,z$ in a metric space with $y,z$ distinct from $x$, the {\it comparison angle} $\tilde \angle yxz$ at $x$ (with respect to $(-1)$-plane) is defined by 
\[
\cosh |y,z| = \cosh |x,y| \cosh |x,z| - \sinh |x,y| \sinh |x,z| \cos \tilde \angle yxz.
\]
A complete metric space $M$ is called an {\it Alexandrov space of curvature $\ge -1$} if 
it is a geodesic space and 

\begin{equation} \label{eq:qudruple}
\tilde \angle x_1xx_2 + \tilde \angle x_2xx_3 + \tilde \angle x_3xx_1 \le 2\pi
\end{equation}
holds for every $x,x_1,x_2,x_3 \in M$ with
$x\neq x_i$\, $(1\le i\le 3)$. 
For any $\kappa \in \mathbb R$, the notion of Alexandrov spaces of curvature $\ge \kappa$ is also defined, by using the comparison angle with
$\kappa$-plane. Here, when $\kappa > 0$, we assume 
$|x,x_i|+|x_i,x_j|+|x_j,x|<\pi/\sqrt{\kappa}$ for all $1\le i\neq j\le 3$.
The diameter of an Alexandrov space of curvature $\ge \kappa$ dose not exceed $\pi / \sqrt {\kappa}$.

From now on, we simply write as an Alexandrov space
to mean an Alexandrov space of curvature bounded from below by a fixed constant. 
In this paper, we consider only finite dimensional Alexandrov spaces. 

Let $M$ be an Alexandrov space. 
If $\dim M \le 2$, then $M$ is a topological manifold.
For $p \in M$, 
let $\Sigma_p = \Sigma_p M$ be the space of directions at $p$
equipped with the angle distance denoted by $\angle$ (or $|\,\cdot\,,\,\cdot\,|$). 
Then, 
the diameter of $\Sigma_p$ does not exceed $\pi$. 
If $\dim M \ge2$, then $\Sigma_p$ is a compact Alexandrov space of 
curvature $\ge 1$. 
We denote by $T_p M$ the tangent cone of $M$ at $p$.

Let us use the following convention. 
For $x \ne p \in M$, we denote by $\Uparrow_p^x$ the set of all directions of geodesics from $p$ to $x$. 
Furthermore, an element of $\Uparrow_p^x$ is denoted by $\uparrow_p^x$.

We denote the {\it boundary} of $M$ by $\ba M$, which is inductively defined as follows:
If $\dim M = 1$, then 
$\ba M$ is the usual boundary as a manifold; when $\dim M \ge 2$, $\ba M$ is the subset of $M$ consisting of all $x \in M$ such that $\ba \Sigma_x$ is not empty.

The {\it inradius} of an Alexandrov space $M$ with boundary is defined by 
\[
\mathrm{inrad}(M) := \sup_{x \in M} |\ba M,x|.
\]

For a subset $M' \subset M$, we set 
\begin{equation} \label{eq:bdry-notaion}
\ba_M M' = \ba M \cap M'.
\end{equation}
If there is no misunderstanding, we simply write $\ba M' = \ba_M M'$. 
By $\mathrm{int}^{\mathbf A} M$, we denote the interior of $M$ which is the complement of $\ba M$ in $M$.

For a topological space $Y$ and its subset $Y'$, we use the notation $\partial Y' = \pa_Y Y'$ meaning the topological boundary, $\partial_Y Y' = \mathrm{cl}_Y Y' \setminus \mathrm{int}_Y Y'$, where $\mathrm{cl}_Y (\cdot)$ and $\mathrm{int}_Y (\cdot)$ denote the closure and interior in $Y$, respectively.

For instance, let us consider 
the upper half plane $\mathbb R_+^2 = \{(x,y) \in \mathbb R^2 \mid y \ge 0\}$ with a standard Euclidean metric and the domain 
$D_+^2 = \{ v \in \mathbb R_+^2 \mid |v| \le 1\}$, which itself is 
an Alexandrov space. 
Then, $\ba_{D_+^2} D_+^2$ is a circle, and $\ba_{\mathbb R_+^2} D_+^2$ and $\pa_{\mathbb R_+^2} D_+^2$ are arcs.

The following result plays a crucial role 
in the study of Alexandrov spaces with 
nonnegative curvature.

\begin{thm}[\cite{PerAlex2}] \label{thm:dist-to-bdry}
If $X$ is an Alexandrov space of nonnegative curvature, then the distance function $d_{\ba X}$ from the boundary is concave on $X$.
That is, for every $x, y \in X$ and $z \in X$ with $|x,z| + |z,y| = |x,y|$, we have 
\[
d_{\ba X} (z) |x,y| \ge d_{\ba X}(x) |z,y| + d_{\ba X}(y) |z,x|.
\]
\end{thm}

For an Alexandrov space $M$ with nonempty boundary $\ba M$, let $M_1$ and $M_2$ denote two copies of $M$.
Then the {\it double} $D(M)$ of $M$ is defined as gluing 
of $M_1$ and $M_2$ along their boundary via the canonical identification 
map $\ba M_1 \to \ba M_2$.
We always consider the length metric on $D(M)$ induced by those of $M_1$ 
and $M_2$.
Then $M_1$ and $M_2$ are considered as subsets of $D(M)$ and $D(M) = M_1 \cup M_2$ and $M_1 \cap M_2 = \ba_{M_i} M_i = \partial_{D(M)} M_i$ for $i=1,2$.
The double $D(M)$ is known to be an Alexandrov space with the same lower curvature bound as that of $M$ (\cite{PerAlex2}).
Furthermore, $M_i$ is convex in $D(M)$ in the sense that any two points in $M_i$ are connected by a minimal geodesic contained in $M_i$ for $i=1,2$. 
We always identify $M$ with the subset $M_1$ of $D(M)$.
We denote by $\phi = \phi_M$ the reflection of $D(M)$ with respect to $\ba M$ which is defined in an obvious way. 
It is an isometric involution of $D(M)$ and $D(M) / \phi = M$.

For a closed subset $N$ of $M=M^1$, the double of $N$ is 
defined as 
\begin{align}\label{eq:D(N)}
  D(N)=N\cup\phi(N).
\end{align}

The notion of extremal subsets of an Alexandrov space was introduced by Perelman and Petrunin (\cite{PP ext}). 
A closed subset $E$ of an Alexandrov space $M$ is called an {\it extremal subset} of $M$ if it satisfies the following condition: 
for every $p \in M \setminus E$, if $q \in E$ is a local minimizer of $|p, \,\cdot\,|$ on $E$, then we have 
\[
\limsup_{x \to q} \frac{|p,x|-|p,q|}{|x,q|} \le 0. 
\]

\begin{ex} \upshape 
The whole space $M$ and the boundary $\ba M$ are examples of extremal subsets. 
Note that for $p \in M$, the singleton $\{p\}$ is extremal if and only if $\mathrm{diam}\, \Sigma_p \le \pi/2$. 
For a rectangle $[0,a] \times [0,b]$,
any union of boundary edges is an extremal subset. 
Moreover, in general, for an Alexandrov ``surface'' $X$ with boundary, if $E \subset \ba X$ is an arc, then $E$ is extremal if and only if for each end point $p \in E$, we have $\diam\, \Sigma_p X \le \pi/2$. 
\end{ex}

Let $M$ be an $n$-dimensional Alexandrov space of curvature $\ge -1$. 
A point $p \in M$ is said to be $\delta$-{\it strained} if 
there exists a collection of pairs of points $\{(a_i, b_i)\}_{i=1,2,\dots,n}$ such that 
\[
\begin{aligned}
&\tilde \angle a_ipb_i > \pi- \delta; &&\tilde \angle a_i p a_j > \pi/2-\delta; \\
&\tilde \angle b_i p b_j > \pi/2-\delta; && \tilde \angle a_i p b_j > \pi/2-\delta
\end{aligned}
\]
for $1 \le i < j \le n$. 
We say that such a $\{(a_i,b_i)\}_{i=1,2,\dots,n}$ is a $(n,\delta)$-{\it strainer} at $p$, and that $\min_{1\le i\le n} \min \{|p,a_i|, |p,b_i|\}$ is the {\it length} of the strainer.
The set of all $\delta$-strained points is denoted by $\mathcal R_{\delta}(M)$, which is open from the definition. 
Moreover, it is known: 
\begin{thm}[\cite{BGP}, cf.\cite{OS}] 
\label{thm:regular-measure}
The Hausdorff dimension of $M \setminus (\mathcal R_{\delta}(M) \cup \ba M)$ is less than or equal to $n-2$. 
\end{thm}
Let $\mathcal S_\delta(M) := M \setminus \mathcal R_\delta(M)$. 

\begin{rem} \label{rem:sing} \upshape
If $M$ is two-dimensional, then $\mathcal S_\delta(M)\setminus \ba M$ is discrete. 
\end{rem}

%\subsection{Topological singular points}
Let $M$ be an $n$-dimensional Alexandrov space possibly with boundary. 
We say that $p \in M$ is a {\it topological singular point} if $p$ has no neighborhood which is homeomorphic to $\mathbb R^n$. 
If $p$ is not a topological singular point, it is called a {\it topological regular point} or a {\it manifold point}. 
 A point $p\in M$ is called an {\it essential singular point} if the radius 
$$
  {\rm rad}(\Sigma_p):=\inf_{\xi \in\Sigma_p} \sup_{\eta \in\Sigma_p} |\xi, \eta|
$$ 
of $\Sigma_p$ is less than or 
equal to $\pi/2$. If $p$ is not an essential singular point, then it is topologically regular 
(\cite{GP}, \cite{PP ext}).

\subsection{Gromov-Hausdorff convergence}

Let us recall the notion of
Gromov-Hausdorff convergence.

For metric spaces $X$ and $Y$,
a not necessarily continuous map $f:X\to Y$ is called
$\e$-approximation if 
\begin{itemize}
\item $||f(x), f(x')|-|x,x'|| < \epsilon$ for all $x, x' \in X\,;$ 
\item $f(X)$ is $\e$-dense in $Y$.
\end{itemize} 
The Gromov-Hausdorff distance $d_{GH}(X,Y)<\e$
iff there are $\e$-approximations $\varphi:X\to Y$ and 
$\psi:Y\to X$.
 
For compact subsets $A_1,\ldots,A_k\subset X$
and $B_1,\ldots,B_k\subset Y$,
the Gromov-Hausdorff distance 
$d_{GH}((X,A_1,\ldots,A_k),(Y,B_1,\ldots,B_k))<\e$ 
iff there are $\e$-approximations $\varphi:X\to Y$ and 
$\psi:Y\to X$ such that the restrictions
$\varphi|_{A_i}$ and $\psi|_{B_i}$ give
$\e$-approximations between $A_i$ and $B_i$ for any
$1\le i\le k$.

For pointed metric spaces $(X,a)$ and $(Y,b)$, 
the pointed Gromov-Hausdorff distance 
$d_{pGH}((X,a), (Y,b))<\e$ iff
there are $\e$-approximations 
$\varphi:B(a,1/\e)\to B(b,1/\e)$ and 
$\psi:B(b,1/\e)\to B(a,1/\e)$ such that
 $|\varphi(a),b|<\epsilon$ and $|\psi(b),a|<\epsilon$.

Let us recall the notion of pointed equivariant 
Gromov-Hausdorff distance
introduced by Fukaya (\cite{F}).

Let $G$ and $H$ be groups of isometries of $X$ and $Y$ respectively. 
For $r > 0$, we set 
\[
G(a,r) := \{ g \in G \mid |ga, a| < \epsilon\}.
\]
We say that a triple $(f,\xi,\eta)$ an {\it $\epsilon$-approximation from $(X,a,G)$ to $(Y,b,H)$} if 
\begin{enumerate}
\item $f$ is an $\epsilon$-approximation from $(X,a)$ to $(Y,b)$; 
\item $\xi$ is a map $\xi : G \to H$ satisfying that 
\begin{itemize}
\item $\xi(g) \in H(b,\epsilon^{-1}+\epsilon)$ for every $g \in G(a,\epsilon^{-1})\,;$
\item for $g \in G(a,\epsilon^{-1})$ and $x \in B(a,\epsilon^{-1}-\epsilon)$, we have
\[
|f(g x), \xi(g)(f(x))| < \epsilon; 
\]
\end{itemize}
\item $\eta$ is a map $\eta : H \to G$ satisfying that 
\begin{itemize}
\item $\eta(h) \in G(a,\epsilon^{-1}+\epsilon)$ for every $h \in H(b,\epsilon^{-1})\,;$
\item for $h \in H(b,\epsilon^{-1})$ and $x \in B(a,\epsilon^{-1}-\epsilon)$, 
\[
|f(\eta(h) x), h(f(x))| < \epsilon.
\]
\end{itemize}
\end{enumerate}
%Then 
The {\it equivariant pointed Gromov-Hausdorff distance}
between $(X,a,G)$ and $(Y,b,H)$, denoted by
$d_{epGH}((X,a,G),(Y,b,H))$, is defined as the infimum 
of those $\e$ for which there are $\e$-approximations 
from $(X,a,G)$ to $(Y,b,H)$ and from $(Y,b,H)$ to $(X,a,G)$.
We say that a sequence $(X_i,a_i,G_i)$ 
converges to $(X,a,G)$ if $d_{epGH}((X_i,a_i,G_i),(X,a,G))\to 0$ as
$i\to\infty$.

\section{Preliminaries II --- Basics of collapsing Alexandrov spaces}
\label{sec:prelim2}
In this section, we first recall some fundamental results in the collapsing theory, Stability Theorem, Fibration Theorem and the rescaling argument, which are used to determine the local topology of collapsing spaces. 
We also recall the classification of nonnegatively curved Alexandrov three-spaces with boundary, which appear
as the rescaling limit spaces.
Main reference here are \cite{PerAlex2}, \cite{Y conv}, \cite{Y 4-dim} and \cite{Y ess}.

\subsection{Regularities of distance maps}
Let $M$ be an $n$-dimensional Alexandrov space of curvature $\ge -1$. 
For a compact subset $A$ of $M$, $x \in M \setminus A$ and $\xi \in \Sigma_x$, the {\it directional derivative} of $d_A = |A, \,\cdot\,|$ in the direction $\xi$ is defined as 
\[
d_A'(\xi) := \lim_{t \to 0} \frac{d_A(\gamma(t))-d_A(x)}{t}, 
\]
where $\gamma$ is a rectifiable curve in $M$ such that $\gamma(0) = x$ and $\gamma'(0)=\xi$. 
Furthermore, the first variation formula (see \cite{BGP})
of the distance functions says
\[
d_A'(\xi) = -\cos \angle (\Uparrow_x^A, \xi).
\]
Here, $\Uparrow_x^A$ is the union of $\Uparrow_x^a$ for all $a\in A$ 
with $|a,x| = |A,x|$. 
We say that $d_A$ is $c$-{\it regular} at $x$ for $c > 0$, if $d_A'(\xi) > c$ for some $\xi \in \Sigma_x$. 
$d_A$ is {\it critical} at $x$ if $d_A$ is not $c$-regular at $x$ for every $c > 0$.

Let $A_1, \dots, A_k$ be compact subsets of $M$. 
\begin{defn}[\cite{PerAlex2}] \upshape \label{def:regular}
For $c, \epsilon > 0$, the $k$-tuple $(d_{A_1}, d_{A_2}, \dots, d_{A_k})$ is called {\it $(c,\epsilon)$-regular} at $p \in M \setminus \bigcup_{i=1}^k A_i$ if there exists $x \in M$ such that 
\begin{align}
&\angle (\Uparrow_p^{A_i}, \Uparrow_p^{A_j}) > \pi/2 - \epsilon; \\
&\angle (\uparrow_p^x, \Uparrow_p^{A_i}) > \pi/2+c
\end{align}
for $1 \le i \ne j \le k$. 
\end{defn}
If $\epsilon$ is small compared with $c$, and if $(d_{A_i})_{i=1}^k$ is $(c,\epsilon)$-regular at some point in an Alexandrov space $M$, then we have $k \le \dim M$
(see \cite{PerAlex2}).

\subsection{Stability and Fibration Theorems}
Perelman proved the following significant result. 

\begin{thm}[Stability Theorem {\cite{PerAlex2}}] \label{thm:stability}
For a compact Alexandrov space $X$, there exists $\e=\e_X>0$
satisfying the following. Le $Y$ be a compact Alexandrov space of curvature $\ge -1$ with the same dimension as $X$.
If $d_{GH}(X,Y)<\e$, then $X$ and $Y$ are homeomorphic to each other.
\end{thm}
There is a respectful version of this theorem.

\begin{thm}[\cite{PerAlex2}] \label{thm:stability respectful}
For a compact Alexandrov space $X$
and compact subsets $A_1,\ldots,A_k\subset X$, suppose that 
$(d_{A_1}, \dots, d_{A_k})$ is $(c,\delta)$-regular on an open set $U \subset X$, where $\delta$ is small enough compared with 
$c$ and $U$.
Then there exists $\e=\e_{U,A_1,\ldots,A_k}>0$ satisfying the following:
For a compact Alexandrov space $Y$ of curvature $\ge -1$ 
with the same dimension as $X$, and for 
$B_1,\ldots,B_k\subset Y$ ($1\le i\le k$), suppose that the Gromov-Hausdorff distance between $(X,A_1,\dots, A_k)$ and $(Y,B_1, \dots, B_k)$ is less than $\e$. Then 
there exists a homeomorphism $f : X \to Y$ such that 
$d_{B_i} \circ f = d_{A_i}$ on $K$ for all $1 \le i \le k$, where $K = \bigcap_{1 \le i \le k} d_{A_i}^{-1}[a_i,b_i]$ and $[a_i,b_i] \subset d_{A_i}(U)$. 
\end{thm}

In Theorem \ref{thm:stability respectful}, distance functions from compact sets can be replaced by more general functions (see \cite{PerAlex2}).

As a direct consequence of Theorem \ref{thm:stability respectful}, we have: 
\begin{cor} \label{cor:cone}
For each point $p$ in every Alexandrov space, there exists $r > 0$ such that $(U(p,r), p)$ is homeomorphic to $(K(\Sigma_p), o)$. 
\end{cor}

By Corollary \ref{cor:cone} and a classification of the topologies of space of directions, which are positively curved Alexandrov spaces, in low dimension, when $n \le 3$, we immediately have 

\begin{cor} \label{cor:low_dim} \upshape
Any three-dimensional Alexandrov space is topologically an orbifold possibly with boundary.
\end{cor}

Here, for an $n$-dimensional orbifold $O$ possibly with boundary, its boundary is defined as the set of $x \in O$ such that $H^n(O, O \setminus \{x\}; \mathbb Z_2) = 0$.
This concept coincides with the boundary of an Alexandrov space.

To state Fibration Theorem, we prepare some definition. 

\begin{defn}[\cite{Y conv}] \upshape \label{def:Lipschitz submersion}
Let $X$ and $Y$ be domains in Alexandrov space, and $\epsilon > 0$. 
A continuous surjective map $f : Y \to X$ is called an {\it $\epsilon$-almost Lipschitz submersion} if it is an $\epsilon$-Gromov-Hausdorff approximation and for any $y, y' \in Y$, we have 
\[
\left| \sin \theta_{y,y'} - \frac{|f(y),f(y')|}{|y,y'|}\right| < \epsilon, 
\]
where $\theta_{y,y'}$ is the infimum of the angles $\angle y' y y''$ when $y''$ runs over the fiber $f^{-1}(f(y))$.
\end{defn}

Let $X$ be a $k$-dimensional Alexandrov space of curvature $\ge -1$ with nonempty boundary.
Let $D(X)$ be the double of $X$ which is $X_1 \cup X_2$ as a set, where $X_1 = X$ and $X_2$ is an isometric copy of $X$. 
A $(k,\delta)$-strainer $\{(a_i, b_i)\}_{1 \le i \le k}$ in $D(X)$ at $p \in X$ is {\it admissible} if $a_i \in X$, $b_j \in X$ for $1 \le i \le k$ and $1 \le j \le k-1$.
Let $\mathcal R_\delta^D(X)$ denote the set of all points of $X$ which have an admissible $(k,\delta)$-strainer.

For a compact domain $Y$ of $\mathcal R_\delta^D(X)$, the $\delta_D$-{\it strained radius of} $Y$, denoted by $\delta_D\text{-str.rad}(Y)$, is defined as the infimum of $\ell > 0$ such that every point $p \in Y$ has an admissible $(k,\delta)$-strainer of length $\ge \ell$. 

For small $\nu > 0$, we set 
\[
\begin{aligned}
Y_\nu &:= \{x \in Y \mid |\ba X, x| \le \nu \}, \\
\partial_0 Y_\nu &:= \{x \in Y \mid |\ba X, x| = \nu \},\\
\mathrm{int}_0 Y_\nu &:= Y_\nu \setminus \partial_0 Y_\nu.
\end{aligned}
\]

\begin{thm}[Fibration Theorem \cite{Y conv}, \cite{Y 4-dim}] \label{thm:fibration}
Let $X$ be a $k$-dimensional Alexandrov space and $G$ a finite group acting on $X$ by isometries. 
For $\mu > 0$, there exists positive numbers $\e_{X,G}(\mu)$ and $\nu = \nu_{X,G}(\mu)$ satisfying the following:
Let $Y \subset \mathcal R_{\delta}^D(X))$ be a $G$-invariant compact domain with ${\delta}_D\text{-}str.rad(Y) > \mu$, where $\delta < \delta_k$. 
Let $M$ be an $n$-dimensional Alexandrov space of curvature $\ge -1$.
Suppose $M= \mathcal R_{\delta_n}(M)$ and $d_{\mathrm{eGH}}((M,G), (X,G)) < \epsilon \le \e_{X,G}(\mu)$. 
Then there exists a $G$-invariant compact domain $N \subset M$ and a $G$-invariant decomposition
\[
N = N_\mathrm{int} \cup N_\mathrm{cap}
\]
of $N$ into two closed domains glued along their boundaries and a $G$-equivariant Lipschitz map $f : N \to Y_\nu$ such that 
\begin{enumerate}
\item $N_\mathrm{int}$ is the closure of $f^{-1}(\mathrm{int}_0 Y_\nu)$ and $N_\mathrm{cap} = f^{-1}(\partial_0 Y_\nu)$; 
\item the restriction $f|_{N_\mathrm{int}} : N_\mathrm{int} \to Y$ and $f |_{N_\mathrm{cap}} : N_\mathrm{cap} \to \partial_0 Y_\nu$ are 
\begin{enumerate}
\item locally trivial fiber bundles; 
\item $\tau(\delta, \nu, \epsilon/\nu)$-Lipschitz submersion.
\end{enumerate}
\end{enumerate}
\end{thm}

\subsection{Rescaling argument} \label{ssub:rescaling}

To study the topology of neighborhoods of singular points, the following result is useful. 

\begin{thm}[\cite{Y ess}] \label{thm:rescaling}
Let $M_i$ be a sequence of Alexandrov spaces of curvature $\ge -1$ of the same dimension having the collapsing limit $X$. 
Let $p_i \in M$ and $R > 0$. 
Suppose that one of the following holds. 
\begin{enumerate}
\item For any $p_i' \in M_i$ with $|p_i,p_i'| \to 0$, $d_{p_i'}$ has a critical point in $B(p_i',R) \setminus \{p_i'\}$ for large $i$.
\item For any $p_i' \in M_i$ with $|p_i,p_i'| \to 0$, $\pa B(p_i',R)$ is not homeomorphic to any space of curvature $\ge 1$ for large $i$.
\end{enumerate}
Then, there exist points $\hat p_i \in M_i$ with $|p_i,\hat p_i| \to 0$ and positive numbers $\delta_i$ with $\delta_i \to 0$ such that any limit $(Y,y_0)$ of a rescaled sequence $\{((1/\delta_i)M_i, \hat p_i)\}$ is a noncompact Alexandrov space of nonnegative curvature with 
\begin{itemize}
\item $\dim Y \ge \dim X + 1$, 
\item $\dim S \le \dim Y - \dim X$, 
\item $\dim Y(\infty) \ge \dim X-1$.
\item Moreover, there exists an expanding map $\Sigma_pY \to Y(\infty)$, where $p\in Y$ is the limit of $\hat p_i$. 
\end{itemize}
Here, $S$ is a soul of $Y$ and $Y(\infty)$ denotes the ideal boundary of $Y$. 

Furthermore, if every $M_i$ is a double of some Alexandrov space $N_i$ and $p_i \in \ba N_i \subset D(N_i) = M_i$, then $\hat p_i$ can be taken as $\hat p_i \in \ba N_i$.
\end{thm}

\begin{defn} \label{def:rescaling assumption} \upshape
Let $\{M_i\}$ be as in Theorem \ref{thm:rescaling}.
We say that a sequence of domains $D_i \subset M_i$ {\it can be rescaled} if $D_i$ contains $B(p_i,R)$ for fixed $R > 0$ such that $\{p_i\}$ and $R$ satisfies one of the assumption $(1)$ and $(2)$ in Theorem \ref{thm:rescaling}.
\end{defn}

Remark that the condition $(2)$ in Theorem \ref{thm:rescaling} implies $(1)$, due to Perelman's Morse theory (\cite{Per:elem}, \cite{PerAlex2}).

\subsection{Classification of nonnegatively curved three-spaces with boundaries} \label{ssec:classif=nonnegative}
We recall the classification result of three-dimensional Alexandrov spaces of nonnegative curvature with boundaries, due to the second author (\cite{Y 4-dim}).

First, we recall the construction of a soul of a nonnegatively curved Alexandrov space possibly with boundary (\cite{CG}). 
Let $C$ be a compact Alexandrov space of nonnegative curvature having nonempty boundary $\ba C$. 
Due to Theorem \ref{thm:dist-to-bdry}, the distance function 
\[
d_{\ba C} : C \to \mathbb R
\]
is concave on $C$. 
It follows that $d_{\ba C}$ is regular on $d_{\ba C}^{-1}((0, \max d_{\ba C}))$.
Then, the set $C^1 := d_{\ba C}^{-1} (\max d_{\ba C})$ of all maximum points of $d_{\ba C}$ is totally convex in $C$ and $\dim C^1 < \dim C$.
Hence, $C^1$ itself is a nonnegatively curved Alexandrov space possibly with boundary. 
If $C^1$ has boundary, 
repeating the same procedure, 
we obtain a totally convex subset $C^2$ of $C^1$ with $\dim C^2 < \dim C^1$. 
% with $\dim C^2 < \dim C^1$. 
Since $\dim C < \infty$, 
we finally obtain a sequence 
\[
C = C^0 \supset C^1 \supset C^2 \supset \cdots \supset C^k 
\]
of totally convex subsets $C^i$ of $C$ with $\dim C^i < \dim C^{i-1}$ and $\ba C^k = \emptyset$.
Then, we call $C^k$ a {\it soul} of $C$.

Let $X$ be a noncompact Alexandrov space of nonnegative curvature possibly with boundary. 
We fix an arbitrary point $p \in X$ and consider a ray 
\[
\gamma : [0, \infty) \to X
\]
starting from $p = \gamma(0)$. 
Note that $X$ has at least one ray starting from $p$, because $X$ is noncompact. 
Let us consider the {\it Busemann function} $b_\gamma$ associated with $\gamma$ defined by 
\[
b_\gamma (x) = \lim_{t \to \infty} d(\gamma(t),x) - t.
\]
Note that, by the triangle inequality, the above limit always exists and $b_\gamma$ is $1$-Lipschitz.
By the assumption that $X$ has nonnegative curvature, $b_\gamma$ is concave.
Let us consider the Busemann function $b_p$ with respect to $p$ by 
\[
b_p (x) = \inf_{\gamma} b_\gamma (x),
\]
where $\gamma$ runs over all rays starting from $p$, which is also a concave $1$-Lipschitz function on $X$.
Note that 
$b_p$ has a maximum value. 
For any $t \ge 0$, the subset 
\[
X^t := b_p^{-1} ([\max b_p - t, \max b_p])
\]
is compact and totally convex in $X$.
Fixing an arbitrary positive number $t$, we set 
\[
C := X^t
\]
which is a compact Alexandrov space of nonnegative curvature with boundary 
\[
\ba C = \pa X^t \cup \ba_X X^t = b_p^{-1} (\max b_p - t) \cup \ba_X X^t.
\]
By above argument, we obtain a soul of $C$, which is also called a soul of $X$.
Note that if $\ba X = \emptyset$, then $\ba C = \pa X^t$ and $C^1 = X^0$, 
and hence, this construction is the same as the usual construction of soul in an open Alexandrov space of nonnegative curvature.

A relation between a soul and the whole space is known as follows.
\begin{thm}[\cite{PerAlex2}] \label{thm:Sharaftodinov}
Let $X$ be an Alexandrov space of nonnegative curvature with a soul $S$.
We assume that either $X$ is compact and has boundary, or is noncompact.
Then, there exists a continuous map 
\[
\varphi : X \times [0,1] \to X 
\] 
such that 
\begin{itemize}
\item $\varphi_0 = \mathrm{id}_X$; 
\item $\varphi_1 (X)=S$; 
\item $\varphi_t$ is $1$-Lipschitz and $\varphi_t|_S = \mathrm{id}_S$ for every $t \in [0,1]$.
\end{itemize}
Here, $\varphi_t := \varphi(\,\cdot\,, t)$. 
In particular, $X$ has the same homotopy type as $S$.
\end{thm}

When $X$ is three-dimensional, it is known that the geometry of a soul determines the topology and geometry of $X$ (\cite{SY}, \cite{MY}, \cite{Y 4-dim}).
We recall how the topology and geometry of noncompact 3-dimensional Alexandrov spaces with boundary are determined by 
their souls, where are needed in later sections.

\par\medskip

We consider the following three-dimensional Alexandrov
spaces with nonnegative curvature, which play
important roles.
\par\n (1)\,
Let $M_{\mathrm{pt}}$ denote the quotient space of $\mathbb R^2 \times S^1$ by the relation $(v,z) \mapsto (-v,\bar z)$.
Here, $S^1$ is regarded as a unit circle in $\mathbb C$ centered at $0$ and $\bar z$ is the complex conjugation of $z$. 
\par\n (2)\,
Let $M_{\mathrm{pt}}^+$ denote the quotient space of $\mathbb R^2 \times [-1,1]$ by the relation $(v,t) \sim (-v,-t)$.

Note that $M_{\mathrm{pt}}$ is homeomorphic to the 
double of $M_{\mathrm{pt}}^+$.

\begin{thm}[Soul Theorem {\cite[Theorem 7.13 Case B]{Y 4-dim}}] \label{soul theorem}
Let $X^3$ be a three-dimensional noncompact Alexandrov space of nonnegative curvature with boundary.
Suppose that $\ba X^3$ has only one component.
Let $S$ be a soul of $X^3$.
\begin{itemize}
\item[(1)] If $\dim S = 2$, then $X^3$ is isometric to $S \times [0, \infty)$.
\item[(2)] If $\dim S = 1$, then $X^3$ is isometric to $(\mathbb R \times N^2) / \Lambda$, where $N^2$ is either homeomorphic to $\mathbb R^2_+$ or isometric to $\mathbb R \times I$,
and $\Lambda \simeq \mathbb Z$. 
In particular, $X^3$ is homeomorphic to an $N^2$-bundle over $S^1$.
\item[(3)] Suppose $\dim S = 0$. 
If $X^3$ has two ends, then it is isometric to a product $\mathbb R \times X_0$ with $X_0 \approx D^2$. 

Suppose that $X^3$ has exactly one end, and let $C$ be the maximum set (possibly empty) of $d_{\ba X^3}$.
Then $C$ is either empty or of dimension $\ge 1$.
\begin{itemize}
\item[(a)] If $C$ is empty, then $X^3$ is homeomorphic to $\mathbb R^3_+$; 
\item[(b)] If $\dim C = 1$, then $C$ is a geodesic ray, and $X^3$ is homeomorphic to ether $\mathbb R^3_+$
or the identification space $D^2 \times \mathbb R / (x, y) \sim (-x, -y)$; 
\item[(c)] If $\dim C = 2$, then $C$ is homeomorphic to 
 $\mathbb R^2$ or $\mathbb R^2_+$.
\begin{itemize}
\item[(i)]  
  If $C \approx \mathbb R^2$, then $X$ is either homeomorphic to $\mathbb R^3_+$ or isometric to one of the quotient spaces 
\begin{align*}
 \hspace{1cm} & \hspace{1cm}\hat C \times [-t,t] / (x,y) \sim (\sigma(x), -y), \\
& \hspace{1cm} S_{\ell}^1 \times \mathbb R \times [-t,t] /(z,x,y) \sim (\bar z,-x,-y),
\end{align*}
for some $t, \ell> 0$,
where $\hat C$ is a branched double covering of $C$ with a single branching point that is a unique topological singular point in $C$, and $\sigma$ is the deck transformation of the branched covering $\hat C \to C$. 
\item[(ii)] 
  If $C \approx \mathbb R^2_+$, then $X^3$ is either homeomorphic to $\mathbb R^3_+$ or isometric to the quotient 
$$
  D\times\R/(x,y)\sim (\sigma(x),-y),
$$
where $D$ is a compact Alexandrov surface of nonnegative curvature with boundary such that 
\begin{itemize}
\item the farthest points from $\pa D$ forms a segment or a point, say $L\,;$
\item $D$ is symmetric with respect to the midpoint $O$ of $L$,
and $\sigma$ is the symmetry of $D$ about $O$.
\end{itemize}
\end{itemize}
\end{itemize}
\end{itemize}
\end{thm}

As a corollary to Theorem \ref{soul theorem}, we obtain the geometry and topology of a large ball around the soul $S$ of 
$X$, which is used to classify the local topologies of collapsing Alexandrov three-spaces with boundaries.

We set
\begin{align*}
B_\mathrm{pt} &:= D^2 \times S^1 / (x,z) \sim (-x,\bar z)
       \subset M_\mathrm{pt} \\
B_\mathrm{pt}^+ &:= D^2 \times [-1,1] / (x,t) \sim (-x,-t)
       \subset M_\mathrm{pt}^+.
\end{align*}

\begin{cor} \label{cor:soul thm ball}
Let $X^3$ and $S$ be as in Theorem \ref{soul theorem}. 
Let $B$ be a closed ball around $S$ with a large radius.
Then, the following holds.
\begin{itemize}
\item[(1)] If $\dim S = 2$, $B$ is isometric to $S \times [0, r]$ for some $r > 0$ and $\ba B = S \times \{0\}$ and $\pa B = S \times \{r\}$.
\item[(2)] If $\dim S = 1$, $B$ is isometric to $(\mathbb R \times B_0) / \mathbb Z$, that is $B_0$-bundle over $S^1$,
where $B_0$ is homeomorphic to $D^2$. 
Furthermore, $\pa B$ is homeomorphic to an $I$-bundle or an $(I \sqcup I)$-bundle over a circle.

\item[(3)] Suppose $\dim S = 0$. 
If $X^3$ has two ends, then $B$ is isometric to a product $[-r,r] \times X_0$ with $X_0 \approx D^2$ for some $r > 0$ and $\partial B = \{\pm r\} \times X_0$.

Suppose that $X^3$ has exactly one end, and let $C$ be the maximum set of $d_{\ba X^3}$.
\begin{itemize}
\item[(a)] If $C = \emptyset$, then $B \approx D^3$ and $\partial B \approx D^2$; 
\item[(b)] If $\dim C = 1$, then $B$ is homeomorphic to $D^3$ with $\partial B \approx D^2$ or $K_1(P^2)$ 
with $\partial B \approx D^2$;
\item[(c)] If $\dim C = 2$, then $C \approx \mathbb R^2$ or $\mathbb R^2_+$.
\begin{itemize}
\item[(i)] If $C \approx \mathbb R^2$, then one of the following holds:
\begin{itemize}
\item $B \approx D^3$ with $\partial B \approx D^2$; 
\item $B \approx B_{\mathrm{pt}}^+$ with $\partial B \approx \Mo$; 
\item $B \approx S^1 \times [-1,1] \times [-1,1] / (x,s,t) \sim (-x,-s,-t)\approx B_{\mathrm{pt}}$ with $\pa B \approx (S^1 \times \{\pm 1\} \times [-1,1] /\!\!\sim) \approx S^1 \times I$.
\end{itemize}
\item[(ii)] If $C \approx \mathbb R^2_+$, $B \approx D^3$ with $\partial B \approx D^2$ or $B \approx K_1(P^2)$ with $\partial B \approx D^2$.
\end{itemize}
\end{itemize}
\end{itemize}
\end{cor}
In what follows, for convenience, 
we summarize the above result as follows.
Let $S$ and $B$ be as in the Corollary \ref{cor:soul thm ball}. 

\renewcommand{\arraystretch}{1.2}
\begin{table}[htb] 
	\begin{tabular}{|c|c|c|c|} \hline
		$\dim S$ & $B \equiv$ & $\partial B$ & supplement \\ \hline 
		2 & $S \times [0,r]$ & $S \times \{r\}$ & $r>0$ \\ \hline 
		$\begin{array}{c}
			\\[-0.7em] 1 \\[-0.7em]
		\end{array}$ & $
		\begin{array}{c}
			\\[-0.7em]
			(\mathbb R \times B_0) /\mathbb Z
			\\[-0.7em]
		\end{array}$ & 
		$I\text{-bundle over }S^1$ & $
		\begin{array}{c}
			\\[-0.5em] 
			B_0 \approx D^2 \\[-0.5em]
		\end{array}
		$ \\ \cline{3-3} 
		&&$(I \sqcup I)\text{-bundle over }S^1$ & \\ \hline 
	\end{tabular}
	\pmed
	\caption{Topology of $(B,\pa B)$}
	\label{table:B-paB}
\end{table}
\renewcommand{\arraystretch}{1}

When $\dim S=0$ and $X$ has exactly one end, 
we have the following. Let $C$ be the set of maximum points of $d_{\ba X}$ as before.

\renewcommand{\arraystretch}{1.2}
\begin{table}[htb] 
	\begin{tabular}{|c|c|c|c|} \hline
		$\dim C$ & $C$ & $B \approx$ & $\partial B \approx$ \\ \hline 
		$-1$ & $=\emptyset$ & $D^3$ & $D^2$ \\ \hline 
		1 & $
		\begin{array}{c}
			\\[-0.7em]
			=\text{a ray }
			\\[-0.7em]
		\end{array}
		$ & $D^3$ & $
		\begin{array}{c}
			\\[-0.7em]
			D^2 \\[-0.7em]
		\end{array}
		$ \\ \cline{3-3}
		&& $K_1(P^2)$ & \\ \hline 
		& & $D^3$ & $D^2$ \\ \cline{3-4}
		& $\approx \mathbb R^2$ & $B_\mathrm{pt}^+$ & $\Mo$ \\ \cline{3-4}
		2 & & $B_{\mathrm{pt}}$ & $S^1\times I$ \\ \cline{2-4} 
		& $\begin{array}{c}
			\\[-0.7em]
			\approx \mathbb R^2_+
			\\[-0.7em]
		\end{array}$ & $D^3$ & $
		\begin{array}{c}
			\\[-0.7em]
			D^2
			\\[-0.7em]
		\end{array}$ \\ \cline{3-3} 
		& & $K_1(P^2)$ & \\ \hline 
	\end{tabular}
	\pmed
	\caption{Topology of $(B,\pa B)$}
	\label{table:BpaB2}
\end{table}
\renewcommand{\arraystretch}{1}

Next we discuss the compact case.
First, let $X$ be a noncompact Alexandrov three-space without boundary and with codimension one soul $S$.
Then $X$ is isometric to a singular line bundle, say $L(S)$, over $S$ in the sense of Definition \ref{def:gen I-bdl}
(\cite{SY}, \cite{Y 4-dim}).
When the number of singular fibers in 
 $L(S)$ is $k$ and $S$ is homeomorphic to 
 a closed surface $S_0$, we also write as 
 $L(S)=L(S_0;k)$. Since the singular fibers in $L(S)$ occur
 at essential singular points of $S$, we have 
$k\le 4$. 

If $L(S)$ is not a topological manifold, then it is 
isometric to one of the following classes
 $$
 L(S^2;2),\quad L(S^2;4), \quad L(P^2;2),
 $$
(see \cite[Section 17]{Y 4-dim}, \cite[Section 2]{MY}).
By Lemma \cite[Corollary 2.56]{MY}, we have
\[
L(S)=\hat S \times \R / (x,t) \sim (\sigma (x), -t)
\]
where $\hat S$ is a branched covering of $S$
branched at the topological singular points of $L(S)$,
and $\sigma$ is an isometric involution of $\hat S$. 
Here
\beqq
\hat S=
\begin{cases}
&S^2  \quad \text{for $L(S^2;2)$} \\
&T^2  \quad \text{for $L(S^2;4)$} \\
&K^2  \quad \text{for $L(P^2;2)$}
\end{cases}
\eeqq

Let $B(S)$ denote a large metric ball
around $S$ in $L(S)$.

\begin{thm}[{\cite[Theorem 17.13. Case A]{Y 4-dim}}] \label{soul theorem compact}
Let $X$ be a compact Alexandrov three-space of nonnegative curvature with boundary. 
Let $S$ denote a soul of $X$.
Then, the topology of $X$ is determined as follows.
If $\ba X$ is disconnected, then $X$ is isometric to $X_0 \times I$, where $X_0$ is a component of $\ba X$.
Suppose that $\ba X$ is connected.
\begin{itemize}
\item[(1)]
If $\dim S = 0$, then $X$ is homeomorphic to $D^3$, $K_1(P^2)$ or $B_\mathrm{pt};$
\item[(2)]
If $\dim S = 1$, then $X$ is homeomorphic to a $D^2$-bundle over $S^1\,;$
\item[(3)]
If $\dim S = 2$, then $X$ is isometric to one of 
a flat $I$-bundle over $S$, 
 $B(S^2;2)$, $B(S^2;4)$ and $B(P^2;2)$.
\end{itemize}
\end{thm}

\section{Generalized $I$-bundles and generalized Seifert fibrations}%%%M...
\label{sec:sing_fib}
In this section, we define the notions of singular 
fiber bundles used in the present paper.

We say that continuous maps $p : \tilde A \to A$ and 
$p' :\tilde A' \to A'$ between topological spaces are {\it isomorphic} 
if there are homeomorphism $\tilde q :\tilde A \to\tilde A'$ and $q : A \to A'$ such that $q \circ p = p' \circ\tilde q$.
Such a pair $(\tilde q,q)$ is called an isomorphism from $p$ to $p'$. 
Moreover, for subsets $E \subset A$, $E' \subset A'$ and homeomorphism $e : E \to E'$, $p$ and $p'$ are isomorphic {\it respecting} $e$ if there exists an isomorphism $(\tilde q, q)$ from $p$ to $p'$ satisfying 
$q = e$ on $E$.

\subsection{Generalized $I$-bundles}

The next example plays an important role in this paper.

\begin{ex} \upshape
Let $M_{\mathrm{pt}}^+$ denote the quotient space of $\mathbb R^2 \times [-1,1]$ by the relation $(v,t) \sim (-v,-t)$.
The natural projection defined as
\begin{equation} \label{eq:Mpt+}
\pi_{M_\mathrm{pt}^+} : M_{\mathrm{pt}}^+ \ni [v,t] \mapsto [v] \in \mathbb R^2/(v \sim -v)=K(S^1_\pi)
\end{equation}
is an $I$-bundle except over $[0]$. 
The singular fiber $\pi_{M_\mathrm{pt}^+}^{-1}([0])$ is also an arc. 
The preimage of the unit cone $K_1(S^1_\pi)$ by $\pi_{M_{\mathrm{pt}}^+}$ is denoted by $B_{\mathrm{pt}}^+$, which is homeomorphic to $K_1(P^2)$.
We denote by $\pi_{B_\mathrm{pt}^+} : B_{\mathrm{pt}}^+\to K_1(S^1_\pi)$ the restriction of $\pi_{M_\mathrm{pt}^+}$ to 
$B_{\mathrm{pt}}^+$.
 \end{ex}

The following notion of generalized $I$-bundles
is needed to describe the collapsing 
to Alexandrov surfaces in Sections \ref{sec:2-dim no bdry} and \ref{sec:2-dim with bdry}.

\begin{defn} \label{def:gen I-bdl} \upshape
Let $M$ be a three-dimensional topological orbifold with boundary and $X$ be a topological surface possibly with boundary. 
Let $f : M \to X$ be a continuous surjection.
We say that $f$ is a {\it generalized $I$-bundle} if it satisfies the following conditions: 
\begin{enumerate}
\item there exists a (possibly empty) discrete set $\mathcal S \subset X \setminus \partial X$ such that $f$ is a usual $I$-bundle over $X \setminus \mathcal S$; 
\item for each $x \in \mathcal S$, there exists an open neighborhood $U$ of $x$ in $X \setminus \partial X$ such that the restriction $f|_U : f^{-1}(U) \to U$ is isomorphic to \eqref{eq:Mpt+}, respecting $\{x\} \to \{[0]\}$.
\end{enumerate}

The set $\mathcal S$ is called the {\it singular locus} of $f$ and the preimage of a point of $\mathcal S$ is a {\it singular fiber}. 
\end{defn}

\begin{rem}\label{rem:branched-cover}\upshape
For a generalized $I$-bundle $f:M\to X$, 
let 
$N_0:=f^{-1}(\pa X)$, and $M_0$ denote the closure 
of $\pa M\setminus N_0$.
Then 
$\pa M$ is the gluing of $M_0$ and
$N_0$, and $f|_{M_0}:M_0\to X$ is a branched $\Z_2$-covering
branched at $\ca S$.
\end{rem}

The following is the basic example of local 
model of collapsing, which appears in 
Section \ref{sec:2-dim no bdry}.

\begin{ex}[Collapsing of $M_{\mathrm{pt}}^+$] \upshape \label{example: 2-dim no bdry}
For $\epsilon > 0$, let $X(\epsilon)$ be the Cartesian product $\mathbb R^2 \times [-\epsilon, \epsilon]$
equipped with the involution $\sigma$ defined by $\sigma (v, t) = (-v, -t)$ for $v \in \mathbb R^2$ and $t \in [-\e,\e]$. 
The quotient space $M_{\rm pt}^+(\epsilon)$ of $X(\epsilon)$ by 
$\sigma$ is an Alexandrov space of nonnegative curvature with boundary $\ba M_{\rm pt}^+(\epsilon) = \ba X(\epsilon) / \sigma$ whose intrinsic metric is isometric to the flat Euclidean metric on $\mathbb R^2$.
Clearly $M_{\rm pt}^+(\epsilon)$ has a unique topological singular point, which is the origin of $X(\e)$, and it converges to $K(S^1_\pi)$ as $\e\to 0$.
\end{ex}

\subsection{Generalized Seifert fibrations}

In this subsection, we define the notion of 
generalized Seifert fibrations extending 
the one given in \cite[Definition 2.48]{MY}
(see Definition \ref{def:gen S1-bdl}).

We first recall the notion of Seifert fibered solid torus.
Let $\mu, \nu$ be coprime positive integers with 
$\mu \ge \nu\ge 1$.
The quotient space of $\mathbb C \times S^1$ by the identification $(v,z) \sim (e^{2 \pi i/ \mu}v, e^{2\pi i \nu / \mu} z)$, is denoted by $V_{\mu, \nu}$. Let
\begin{equation} \label{eq:V-mu-nu}
\pi_{V_{\mu,\nu}} : V_{\mu,\nu} \ni [v,z] \mapsto [v] \in \mathbb C / v \sim e^{2\pi i/\mu}v
\end{equation}
be the projection.
The preimage of $D^2 / v \sim e^{2\pi i/ \mu} v$ by $\pi_{V_{\mu, \nu}}$ is called the Seifert fibered solid torus of type $(\mu,\nu)$.

The next example was already discussed in \cite[Example 1.2]{MY}, and played an important role there.

\begin{ex}[\cite{MY}]\upshape
Let $M_{\mathrm{pt}}$ denote the quotient space of $\mathbb R^2 \times S^1$ by the relation $(v,z) \mapsto (-v,\bar z)$, where $S^1$ is regarded as a unit circle in $\mathbb C$ centered at $0$ and $\bar z$ is the complex conjugation of $z$. 
We have the natural projection 
\begin{equation} \label{eq:Mpt}
\pi_{M_{\mathrm{pt}}} : M_{\mathrm{pt}} \ni [v,z] \mapsto [v] \in \mathbb R^2 / v\sim -v=K(S^1_\ell),
\end{equation}
which is a circle-bundle over $K(S^1_\ell) \setminus \{[0]\}$, while $\pi_{M_{\mathrm{pt}}}^{-1}([0])$ is an arc.
We denote by $B_{\mathrm{pt}}$ the preimage of 
$K_1(S^1_\pi)$ by $\pi_{M_{\mathrm{pt}}}$.
\end{ex}

Note that $M_{\rm pt}^+$ is a ``half of $M_\mathrm{pt}$''. That is,
the double of $M_\mathrm{pt}^+$ is
homeomorphic to $M_\mathrm{pt}$.

\begin{ex} \label{ex:boundary}\upshape
Let 
\[
 L^2:=D(\{ y\ge 0, -1\le z\le 1\}),
\]
and consider the projection 
\begin{align} \label{eq:boundary}
\pi:\R\times L^2\to \R^2_+, \quad 
\pi(x, [y,z])=(x,y).
\end{align}
This is an $S^1$-bundle over $\mathring{\R}^2_+$,
while $\pi^{-1}(x,0)$ is an arc for any $x\in\R$.
\eqref{eq:boundary} is nothing but
\eqref{eq:L2xId} below.
\end{ex}

We recall the notion of generalized Seifert
bundle defined in \cite[Definition 2.48]{MY}
to describe the collapsing of three-dimensional
{\it closed} Alexandrov space to Alexandrov surfaces. 

\begin{defn}[cf.~\cite{MY}] \upshape \label{def:gen Seif}
Let $M$ be a three-dimensional topological orbifold without boundary, and let $X$ be a topological surface possibly with boundary. 
A surjective continuous map $f : M \to X$ is called a {\it generalized Seifert bundle}if
there exists a (possibly empty) discrete set $\mathcal S$ of $X\setminus \pa X$ such that 
\begin{enumerate}
\item $f$ is an $S^1$-bundle over 
$X \setminus(\mathcal S\cup \pa X)$; 
\item for each $x \in \mathcal S$, there exists an open neighborhood $U$ of $x$ in $X$ such that the restriction of $f$ to $f^{-1}(U)$ is isomorphic to \eqref{eq:Mpt} or \eqref{eq:V-mu-nu} for some $(\mu,\nu)$, respecting the map $\{x\} \to \{[0]\}$ between singletons$\,;$
\item for each $x\in\pa X$, there exists an open neighborhood $U$ of $x$ in $X$ such that the restriction of $f$ to $f^{-1}(U)$ is isomorphic \eqref{eq:boundary}.
\end{enumerate}
The set $\mathcal S$ and the preimage of a point of $\mathcal S$ 
are called the singular locus and a singular fiber as before.
If a singular fiber is an arc (resp. a circle), then it is called a singular $I$-fiber (resp. a singular $S^1$-fiber) of $f$.
Singular fibers are often called exceptional fibers. 
\end{defn}

\begin{rem} \label{rem:bundle-ifference}\upshape
Definition \ref{def:gen Seif} (3) is slightly different from 
the corresponding one given in \cite[Definition 2.48]{MY}. In the present case, the fiber over a point of $\pa X$ is $I$ while it is a point in \cite[Definition 2.48]{MY}. Although both are essentially the same,
we employ Definition \ref{def:gen Seif} (3)
 because of establishing the equivalence
of the notions of generalized $I$ bundles and reflection invariant generalized Seifert bundles,
as stated in Lemma \ref{lem:MtoD(M)}.

When $X$ has no boundary, Definition \ref{def:gen Seif} is 
the same as \cite[Definition 2.48]{MY}.
\end{rem}

In Definition \ref{def:gen S1-bdl},
we shall extend the notion of generalized Seifert bundle substantially to the case that both $M$ and $X$ have 
nonempty boundaries. 

As the following lemma shows, the notion of 
a generalized $I$-bundle $M \to X$ is equivalent 
to that of a reflection-invariant generalized Seifert bundle
$D(M) \to X$.

\begin{lem}\label{lem:MtoD(M)}
Let $M$ and $X$ be as in Definition \ref{def:gen I-bdl}. 
If $f : M \to X$ is a generalized $I$-bundle in the sense of Definition \ref{def:gen I-bdl}, then the double $D(f) : D(M) \to X$ of $f$ that is
defined as $D(f)(\phi(x)) = D(f)(x) = f(x)$ for $x \in M$
is a generalized Seifert bundle in the sense of 
Definition \ref{def:gen Seif}, with no singular circle fibers. 

Conversely, if $F:D(M)\to X$ is a $\phi$-invariant
generalized Seifert bundle in the sense of 
Definition \ref{def:gen Seif}, then it induces 
a generalized $I$-bundle $f:M\to X$ in the sense of Definition \ref{def:gen I-bdl}.
\end{lem}
\begin{proof}
The first half is clear, and hence omitted.

For a $\phi$-invariant
generalized Seifert bundle $F:D(M)\to X$,
define $f:M\to X$ by 
$f(p):=F(p)$ for $p\in M=M_1\subset D(M)$. Let $\ca S$ be the singular locus
of $F$.
For any $x\in X\setminus(\ca S\cup \pa X)$,
since the regular $S^1$-fiber $F^{-1}(x)$ is $\phi$-invariant, it is easy to see that $F^{-1}(x)$ contains exactly two points
of $\ba M$, say $p_1,p_2$, and that 
$\phi$ restricted to $F^{-1}(x)$ is the reflection around the fixed points
$p_1,p_2$. Thus there are exactly two points
of $\ba M$ on $F^{-1}(x)$. 
It follows from a limit argument that for any $x\in \pa X$,
$F^{-1}(x)$ is an arc fixed by $\phi$.

These imply that $f$ is an $I$-bundle
over $X\setminus\ca S$.
For any $x\in\ca S$,
since the singular $I$-fiber $F^{-1}(x)$ is $\phi$-invariant, it is easy to see that $F^{-1}(x)$ contains a unique point
of $\ba M$, say $q$, and that 
$\phi$ restricted to $F^{-1}(x)$ is the reflection around the fixed point $q$. It follows that 
if $F^{-1}(B)\to B$ is isomorphic to
\eqref{eq:Mpt}, then 
$f:f^{-1}(B)\to B$ is isomorphic to
\eqref{eq:Mpt+}. 
\end{proof} 

\begin{ex} \upshape \label{example: 2-dimDM}
Let $M_{\rm pt}^+(\epsilon) \xrightarrow[\epsilon \to 0]{G H} K(S^1_\pi)$
be as in Example \ref{example: 2-dim no bdry}, and let $F : M_{\rm pt}^+(\epsilon) \to K(S^1_\pi)$ be the natural projection defined by $F([(v,t)])=[v]$.
Let $M_{\rm pt}(\epsilon)$denote the double of $M_{\rm pt}^+(\epsilon)$, which is isometric to $\mathbb R^2 \times S^1_{2 \epsilon} / (v, z) \sim (-v, \bar z)$.
Clearly $M_{\rm pt}(\epsilon)$ also converges to $K(S^1_\pi)$ as $\epsilon \to 0$.
Letting $D(F)$ be the double of $F$, we have the following commutative diagram:
\[
\begin{CD}
 & M_{\rm pt}(\epsilon) & @ > D(F) >> \, & K(S^1_\pi) & \\
&@V / \phi V V& & @ V V \mathrm{id} V & \\ 
 & M_{\rm pt}^+(\epsilon) & \,\,\, @> F >> & K(S^1_\pi), &
\end{CD}
\]
where $\phi$ is the canonical reflection on the double $D(M_{\rm pt}^+(\epsilon))=M_{\rm pt}(\epsilon)$ (see 
Section \ref{sec:prelim}). 
\end{ex}

Toward Definition \ref{def:gen S1-bdl}, we prepare several singular $S^1$-bundles
over surfaces (Examples \ref{ex:toy01},
\ref{eq:K(P2)}, \ref{ex:toy02}), which are needed in the formulation 
in Definition \ref{def:gen S1-bdl}.

\begin{ex}[Toy models] \upshape \label{ex:toy01}
 (1)\,
Let $L^2$ be as in Example \ref{ex:boundary}.
This has a singular $S^1$-bundle defined as
\[
\pi_{L^2}:L^2\ni [y,z] \mapsto y \in \mathbb R_+,
\]
where $\pi_{L^2}^{-1}(0)$ is the unique singular $I$-fiber $[-1,1]$.

(2)\,Let us consider 
the infinite M\"obius strip $\Mo_\infty$ defined as 
\[
\Mo_\infty := S^1 \times \mathbb R / (x,t) \sim (-x,-t). 
\]
This also has a singular $S^1$-bundle defined by
\begin{align*}
\pi_{\Mo_\infty} : \Mo_\infty \ni [x,t] \mapsto |t| \in \mathbb R_+,
\end{align*}
where $\pi_{\Mo_\infty}^{-1}(0)$ is the unique singular 
$S^1$-fiber denoted by $S^1/2$. 
\end{ex}

The following model of a singular fibration on $K(P^2)$ is a variant 
of the one described in \cite{MY}.

\begin{ex}[cf.~\cite{MY}] \upshape \label{ex:KP2}
Let $O$ denote the center of the segment
$J:=\{ 0\}\times [-1,1]\subset L^2$,
and let $\sigma$ be the symmetry of $L^2$ 
about $O$.
%defined by $\sigma=([y,z])=\phi([y,-z])$. 
We realize the cone $K(P^2)$ over the real projective plane $P^2$ as the quotient 
$\R\times L^2/(x, [y,z])\sim (-x,\sigma([y,z]))$.
We consider the singular $S^1$-bundle on $K(P^2)$
defined as 
\begin{align} \label{eq:K(P2)}
\pi_{K(P^2)}:K(P^2) \ni [(x,[y,z])] \mapsto (|x|,y)\in \mathbb R_+ \times \mathbb R_+.
\end{align}
The fiber of $\pi_{K(P^2)}$ over any point of 
$(\mathbb R_{> 0})^2$ is a regular circle.
For any $y> 0$, the fiber at 
$(0,y)$ is a singular $S^1$-fiber $S^1/2$, and 
 there exists an open rectangular neighborhood $B$
 of $(0,y)$ in $\mathbb R_+\times \R_+$ such that 
 the restriction of $\pi_{K(P^2)}$ to $B$ is isomorphic to 
\begin{align}\label{eq:MoxId}
\pi_{\Mo_\infty} \times \mathrm{id}_{\mathbb R} : \Mo_\infty \times \mathbb R \to \mathbb R_+ \times \mathbb R. 
\end{align}
For any $x > 0$,
the fiber over $(x,0)$ is an arc $I$, and there 
exists an open rectangular neighborhood $B$ of
$(x,0)$ in $\mathbb R_+\times \mathbb R_+$ such that the restriction of $\pi_{K(P^2)}$ to $B$ is isomorphic to 
\begin{align} \label{eq:L2xId}
\pi_{L^2} \times \mathrm{id}_{\mathbb R} : L^2 \times \mathbb R \to \mathbb R_+ \times \mathbb R.
\end{align}
\end{ex}

\begin{ex}[Toy models] \label{ex:toy02} \upshape
By restricting the maps $\pi_{L^2}$ and $\pi_{\Mo_\infty}$ defined in Example \ref{ex:toy01}, we obtain the maps
\begin{align*}
&\pi_{D^2} := \pi_{L^2}|_{D^2}: D^2 \to [0,1], \\
&\pi_{\Mo} := \pi_{\Mo_\infty} |_{\Mo} : \Mo \to [0,1],
\end{align*}
where $D^2 :=\{ [y,z] \in L^2 \mid |y| \le 1 \} $.
These are singular $S^1$-bundles. 

Let us consider three closed surfaces $S^2$, $P^2$ and $K^2$ and divide them into two pieces: 
\begin{align*}
S^2 &\approx D^2 \cup_{\pa} D^2, \hspace{1em}
P^2 \approx D^2 \cup_{\pa} \Mo, \hspace{1em}
K^2 \approx \Mo \cup_{\pa} \Mo. 
\end{align*}
Here, $A \cup_\pa B$ denote the space obtained by gluing $\pa A$ and $\pa B$ for some homeomorphism $\varphi : \pa A \to \pa B$, when $A$ and $B$ are surfaces with boundary. 
Gluing the two maps $\pi_{D^2}$ and $\pi_{\Mo}$
in an obvious way, 
we obtain singular $S^1$-bundles: 
\begin{align*}
\pi_{S^2} &: S^2 \to I, \hspace{1em}
\pi_{P^2} : P^2 \to I, \hspace{1em}
\pi_{K^2} : K^2 \to I,
\end{align*}
where $I$ is a closed interval defined as the gluing 
$I = [0,1] \cup_{1=1} [0,1]$.
Notice that the isomorphism classes of the maps 
$\pi_{S^2}$, $\pi_{P^2}$ and $\pi_{K^2}$ 
are independent of the gluing maps
and uniquely determined. 
\end{ex}

The following concept is needed in Section \ref{sec:2-dim with bdry}.

\begin{defn} \upshape \label{def:gen S1-bdl}
Let $M$ be a topological three-dimensional orbifold (possibly) with boundary, and $X$ a compact topological surface with boundary. 
We call a continuous surjection $f : M \to X$ a 
{\it generalized Seifert fibration}
if it satisfies the following: 
\begin{enumerate}
\item The restriction of $f$ to $X \setminus \pa X$ is a generalized Seifert bundle in the sense of Definition \ref{def:gen Seif}\,;
\item $f(\pa M)\subset \pa X$. Moreover, letting $E:=f(\pa M)$, we have the following:
\begin{enumerate}
\item $E$ is a disjoint union of finitely many closed arcs $E_\alpha$\,$(\alpha\in\Gamma)$ and 
circles $S_\beta$\,$(\beta\in\Lambda)$ in $\pa X\,;$
\item For each $E_\alpha$, 
there exists a collar neighborhood $B_\alpha$ of $E_\alpha$ in $X$ such that the restriction 
$f|_{B_\alpha}$ is isomorphic to one of 
the following:
\begin{align*}
\pi_{S^2} \times \mathrm{id}_{[0,1)} &: S^2 \times [0,1) \to I \times [0,1); \\
\pi_{P^2} \times \mathrm{id}_{[0,1)} &: P^2 \times [0,1) \to I \times [0,1);\\
\pi_{K^2} \times \mathrm{id}_{[0,1)} &: K^2 \times [0,1) \to I \times [0,1),
\end{align*}
where the maps $\pi_{S^2}, \pi_{P^2}$ and $\pi_{K^2}$ are defined in Example \ref{ex:toy02}, and
$E_\alpha$ corresponds to $I \times \{0\}$. 
\item  $f$ is an $S^1$-bundle near each $S_\beta$. 
\item There exists a finite set $\Omega\subset \pa X\setminus E$ (possibly empty) satisfying the following: 
For every $x\in \Omega$
there exists an open neighborhood $B$ of $x$ in $X$ such that the restriction of $f$ to $f^{-1}(B)$ is isomorphic to \eqref{eq:K(P2)}.
\item For every $x \in \pa X \setminus (E\cup\Omega)$, there exists an open neighborhood $B$ of $x$ in $X$ such that the restriction of $f$ to $f^{-1}(B)$ is isomorphic to one of \eqref{eq:MoxId} and \eqref{eq:L2xId}.
\end{enumerate}
\end{enumerate}
\end{defn}

\begin{rem} \upshape
For the notion of a generalized Seifert fibration, we give several remarks. 
Let $f : M \to X$ be a generalized Seifert fibrations. 
Let $E=(\bigcup_{\alpha\in\Gamma} E_\alpha)\cup
(\bigcup_{\beta\in\Lambda} S_\beta) \subset \pa X$ be the closed set as in Definition \ref{def:gen S1-bdl}. 
\begin{enumerate}
\item $\pa M$ is the disjoint union of $f^{-1}(E_\alpha)$
\,$(\alpha\in\Gamma)$ and $f^{-1}(S_\beta)$\,$(\beta\in\Lambda)$, where $f^{-1}(E_\alpha)$ (resp. $f^{-1}(S_\beta)$) is homeomorphic to one of
$\{S^2,P^2,K^2\}$ (resp. one of $\{T^2,K^2\}$)\,;
\item If $X$ has no boundary, then $M$ has no boundary and $f$ is a generalized Seifert bundle in the sense of Definition \ref{def:gen Seif}.
\item If $M$ has no boundary, then $E$ is empty and $f$ is a map considered in \cite[Theorem 1.5]{MY}.
\item If there exists a component $X_0$ of $\pa X$ not meeting $E$, then for a small closed tubular neighborhood $B$ of $X_0$, we called $f^{-1}(B)$ a {\it generalized solid torus/generalized solid Klein bottle} in \cite[Definition 1.4]{MY}. 
\end{enumerate}
\end{rem}

Summarizing the fiber data of a generalized Seifert 
fibration $f:M\to X$, we have the following:
\beqq
 f^{-1}(x)= \left\{
 \begin{array} {cll}
 S^1 &\quad &x\in (\mathring{X}\setminus\ca S) 
    \cup\mathring{E} \\
 I &\quad &x\in\ca S\\
 \text{$I$ or $S^1/2$} &\quad &x\in \pa X\setminus 
 \mathring{E}.
 \end{array}
 \right.
\eeqq

\section{The case that $X$ is a closed surface}
\label{section: 2-dim no boundary} \label{sec:2-dim no bdry}
In the rest of the paper, we assume that a sequence $M_i$ in $\ca M(3,D)$ 
collapses to an Alexandrov space $X$. 
In this section, we first consider the most basic case that $X$ is two-dimensional and having no boundary.
We denote by $\theta_i$ a positive number
satisfying $\lim_{i\to\infty}\theta_i=0$.

\begin{proof}[Proof of Theorem \ref{thm:dim2nobdry}]
We denote by $D(M_i)=M_i^1\cup M_i^2$ the double of $M_i$ as in 
Section \ref{sec:prelim}, where $M_i^1$ and $M_i^2$ are two copies of $M_i$.
For simplicity, we make an identification 
$$
 M_i=M_i^1,\quad \ba M_i = \ba M_i^1=\ba M_i^2.
$$
Let $\phi_i:D(M_i)\to D(M_i)$ denote the reflection along $\ba M_i$. 

\begin{lem} \label{lem:double-conv}
$D(M_i)$ also converges to $X$.
\end{lem}
\begin{proof}
Let $Y$ be any limit of $D(M_i)$ and let $X^1\subset Y$ and $X^2\subset Y$ be limits of $M_i^1$ and $M_i^2$ respectively 
under the convergence $D(M_i)\to Y$.
Note that $Y$, $X^1$ and $X^2$ are topological two-manifolds with $Y = X^1 \cup X^2$, where 
both $X^1$ and $X^2$ are isometric to $X$ and hence having no boundaries.
The invariance of domains then implies that $X^1 = X^2 = Y$.
\end{proof}

\begin{lem} \label{lem:inrad}
\[
 {\rm inrad}(M_i)<\theta_i.
\]
\end{lem}
\begin{proof}
Let $\varphi_i:D(M_i)\to X$ be a $\theta_i$-approximation.
By the proof of Lemma \ref{lem:double-conv}, the restrictions 
$\varphi_i^{\alpha}=\varphi_i|_{M_i^\alpha}:M_i^{\alpha} \to X$ \,$(\alpha=1,2)$ also give 
$\theta_i$-approximations. 
Let $\psi_i^{\alpha}:X\to\ M_i^{\alpha}$ be $\theta_i$-approximations such that 
$|\psi_i^{\alpha}\circ \varphi_i^{\alpha}(x), x|<\theta_i$ for all $x\in M_i^{\alpha}$.
Then for all $x\in X$, we have
\[
 |\psi_i^{1}(x), \psi_i^{2}(x)| \le
 |\varphi_i\circ\psi_i^{1}(x), \varphi_i\circ\psi_i^{2}(x)|+\theta_i
 <2\theta_i.
\]
It follows that for all $x\in M_i^{1}$, we have 
\begin{align*}
 |x, \ba &M^{1}| \le |x, \psi_i^{2}\circ \varphi_i^{1}(x)| \\
 &\le |x, \psi_i^{1}\circ \varphi_i^{1}(x)|+ 
  |\psi_i^{1}\circ \varphi_i^{1}(x), \psi_i^{2}\circ \varphi_i^{1}(x)| <3\theta_i,
\end{align*}
from which the conclusion is immediate. 
\end{proof}

We fix a small $\delta>0$.
By Remark \ref{rem:sing}, 
$\mathcal S_\delta(X)$ is finite. Set $\{ x_n \}_{n=1}^N:=\mathcal S_\delta(X)$.
For a small $\e>0$, choose small metric balls 
$B_n$ around $x_n$ such that 
\begin{itemize}
 \item $B_{n}\cap B_{n'}=\emptyset$ for all $1\le n\neq n'\le N\,;$
 \item $d_{x_n}$ is $\e$-regular on $B_{n}\setminus \{ x_n\}$.
\end{itemize}
We consider the closed domain
$X_0$ defined as 
\beq \label{eq:X0X1} 
 X_0:=X\setminus \bigcup_{n=1}^N \mathring{B}_{n}.
\eeq

%%%%%%%%
Let $ \left< \phi_i \right>\simeq\Z_2$ be the group of isometries of $D(M_i)$ generated by $\phi_i$.
Observe that $(D(M_i), \left< \phi_i \right>)$ converges to $(X, \{1_X\})$ in the equivariant Gromov-Hausdorff topology, where $1_X$ is the identity of $X$.
Due to the equivariant fibration theorem 
(\cite[Theorem 18.4]{Y 4-dim}), for a $\phi_i$-invariant closed domain $D(M_i)_0$ of $D(M_i)$, we have a $\phi_i$-invariant $S^1$-bundle 
$$
f^D_{i,0}:D(M_i)_0 \to X_0,
$$ 
which induces an $I$-bundle
\beq \label{eq:fi0}
 f_{i,0}:M_{i,0} \to X_0,
\eeq
where $M_{i,0}=D(M_i)_0/\phi_i$.
Note that $D(M_i)_0\subset R_{2\delta}(D(M_i))$ for large enough $i$.

Applying \cite[Theorem 1.3]{MY} to the convergence
$D(M_i)\to X$, we obtain an extension of 
$f^D_{i,0}$
\begin{equation} \label{gen Seif fib}
f_i^D: D(M_i) \to X
\end{equation} 
that is a generalized Seifert bundle in the sense of 
Definition \ref{def:gen Seif}, which is also 
a $\theta_i$-approximation.
It should be emphasized that 
from the construction in \cite[Theorem 1.3]{MY}, 
 $f_i^D$ might not be $\phi_i$-invariant near the singular points of $X$.

We set 
\[
 f_i:=f_i^D|_{M_i}:M_i\to X.
\]
As remarked above, since $f_i^D$ might not be $\phi_i$-invariant near the singular points of $X$,
$f_i$ might not have a fiber structure there.
In what follows, we deform $f_i$ to one having a fiber structure over $X$.
%%%%%%%%%%%%%%

We fix any $1\le n\le N$, and set $x:=x_{n}$ and $B:=B_{n}$ for simplicity.
Our main purpose in this section is to find a fiber structure of $f_i^D: (f_i^D)^{-1}(B) \to B$ 
compatible with that of $f_{i,0}$.

Let $r$ denote the radius of $B$. Choose $p_i\in\pa M_i$
converging to $x$, and set 
$$
B_i^D:=(f_i^D)^{-1}(B), \quad B_i^\alpha:= B_i^D\cap M_i^\alpha\,\,(\alpha=1,2).
$$
Note that $B_i^D$ is $\phi_i$-invariant.
By Theorem \ref{thm:smooth approximation}, 
$B_i^D\setminus \mathring{B}^{D(M_i)}(p_i, r/2)$
is homeomorphic to 
$\pa B_i^D\times I$ via $\phi_i$-invariant flow curves.
This implies 
$$
B_i^\alpha \approx B(p_i,r/2) \approx B(p_i,r)\, (\alpha=1,2).
$$
Since $\pa B_i^D$ is homeomorphic to either a torus or a Klein bottle, 
it can be rescaled in the sense of Definition \ref{def:rescaling assumption}. Namely, 
we can apply Theorem \ref{thm:rescaling}
to obtain points 
$\hat p_i\in D(M_i)\cap \pa M_i$ and $\delta_i>0$ 
with $\hat p_i\to x$ and $\delta_i\to 0$ such that 
\begin{itemize}
\item $\left(\frac{1}{\delta_i} B_i^D, \hat p_i\right)$ 
converges to a three-dimensional complete noncompact Alexandrov space $(Y^D,y_0)$ of nonnegative curvature;
\item $\mathring{B}_i^D\simeq Y^D$.
\end{itemize}
 Note that the reference point $\hat p_i$ can be taken as 
a unique local maximum point of an averaged distance functions from 
$\phi_i$-invariant set of points (see \cite{Y ess}). 
Therefore $\hat p_i$ is $\phi_i$-invariant, that is, we have 
$$
  \hat p_i\in \pa M_i.
$$
Passing to a subsequence, we may assume 
that $\left(\frac{1}{\delta_i} B_i^\alpha, \hat p_i\right)$ converges to some space $(Y^\alpha, y_0)$\,($\alpha=1,2$) under this convergence. Then we have 
\begin{align}\label{eq:Y=D(Y)}
 Y^D=D(Y^\alpha)=Y^1\cup Y^2.
\end{align}
Let $S$ be a soul of $Y^\alpha$.
Theorem \ref{thm:rescaling} implies that 
\begin{align}\label{eq:dimDY(infty)}
\dim S \le 1, \quad
 \dim Y^D(\infty)=\dim Y^\alpha(\infty)\ge 1.
\end{align}
From \cite{MY}, we already know that 
the soul of $Y^D$ is a circle and $Y^D$ is isometric to $\mathbb R\times L/\Lambda$,
where $\Lambda\simeq \mathbb Z$ and $L$ is complete noncompact surface
of nonnegative curvature homeomorphic to $\mathbb R^2$.

First we consider the following case:
\pmed
Case (A). $B_i^D$ has no topological singular points. 
\pmed

By \cite[Theorem 1.3]{MY}, the Case (A) occurs if and only if 
$f_i^D$ has no singular $I$-fiber over $B$. 
Thus, $f_i^D: B_i^D \to B$ is isomorphic to a Seifert bundle 
of type $(\mu,\nu)$ for some coprime integers
$\mu, \nu$ with $\mu\ge \nu\ge 1$.
In what follows, we show $\mu=1$.

\begin{lem} \label{lem:no-exceptional}
There exists a trivial $I$-bundle 
$\tilde f_i: B_i^\alpha \to B$ 
compatible with the $I$-bundle structure of $f_{i,0}$
over $\pa B$.
\end{lem}
\begin{proof}

In view of \eqref{eq:Y=D(Y)} and \eqref{eq:dimDY(infty)},
from Theorem \ref{soul theorem} we have the only possibilities that $\ba Y^\alpha$ is disconnected or 
$Y^\alpha$ is isometric to 
$\mathbb R\times L_+/\Lambda$, where 
$L_+\simeq\mathbb R^2_+$.

We show that the latter case does not occur.
Suppose that $Y^\alpha$ is isometric to $\mathbb R\times L_+/\Lambda$.
Choose $\mu_i$ converging to $0$ with $\mu_i \gg \delta_i$ for which we have the convergence $(\frac{1}{\mu_i}B_i^\alpha,\hat p_i)\to (T_x X, o_x)$.
Note that $B_i^D\approx S^1\times D^2$ and 
$(B_i^{\alpha}, \ba _{M_i^{\alpha}} B_i^{\alpha}) \approx (S^1\times D^2_+, S^1\times\ba _{\mathbb R^2_+} D^2_+)$.
Let $\tilde B_i^D$ and $\tilde B_i^\alpha$ be the universal covers of $B_i^D$ and $B_i^\alpha$ respectively.
Let $\pi_i^D:\tilde B_i^D\to B_i^D$ be the 
projection.

Since 
\[
 \Gamma_i^\alpha:=\pi_1(B_i^\alpha)\simeq \Gamma_i^D:=\pi_1(B_i^D)\simeq\mathbb Z,
\]
we may assume that $\tilde B_i^\alpha\subset \tilde B_i^D$ and
$\tilde B_i^D$ is the gluing of $\tilde B_i^1$ and $\tilde B_i^2$ along 
Alexandrov boundaries
$(\pi_i^D)^{-1}(\ba_{M_1}B_i^1)=(\pi_i^D)^{-1}(\ba_{M_2}B_i^2)$.
Let 
$\tilde p_i\in\tilde B_i^\alpha$ be a point over $\hat p_i$.
Passing to a subsequence, we may assume that 
$(\frac{1}{\mu_i}\tilde B_i^D, \tilde p_i, \Gamma_i^D)$
 (resp. $(\frac{1}{\mu_i}\tilde B_i^{\alpha}, \tilde p_i, \Gamma_i)$) converges to 
some $(Z^D, z_0, G^D)$ (resp. $(Z^\alpha, z_0, G^\alpha)$) with respect to the equivariant pointed 
Gromov-Hausdorff topology, where $Z^D$ and $Z^\alpha$ are complete 
noncompact Alexandrov spaces with nonnegative curvature.
Since the action of $\Gamma_i^\alpha$ coincides with the restriction of 
the action of $\Gamma_i^D$ to $\tilde B_i^\alpha$, 
the action of $G^\alpha$ coincides with the restriction of 
the action of $G^D$ to $Z^\alpha$.
From the splitting theorem, we have 
$$
Z^D=\mathbb R\times N,\quad 
G^D=\mathbb R\times A,
$$ 
where $N$ is a Euclidean cone, say $K(S^1_\ell)$, and $A$ is a finite cyclic group
acting on $N$, and $N/A$ is isometric to the 
cone $T_xX$ (see \cite{SY} for the detail). 
Similarly we have 
$$
Z^\alpha/G^\alpha =T_xX,\quad
Z^\alpha=\mathbb R\times N^\alpha.
$$
Therefore $G^\alpha=\mathbb R\times A^\alpha$, where 
$A^\alpha$ is the restriction of $A$ to $N^\alpha$.

However, since $Z^\alpha$ has nonempty boundary, so does $N^\alpha$.
%Note also that $Z^D=D(Z^\alpha)=\mathbb R\times D(N^\alpha)$. 
It turns out that 
$N^\alpha=K(I)$ for a closed interval $I$, and hence
\[
 Z^\alpha/G^\alpha =N^\alpha/A^\alpha=K(I)/A^\alpha \neq T_xX,
\]
which is a contradiction to the fact $Z^\alpha/G^\alpha =T_xX$.

Thus we conclude that $\ba Y^\alpha$ is disconnected,
and therefore $Y$ is isometric to $I\times L$
(see \cite[Theorem 17.3]{Y 4-dim})
with
$L\approx\R^2$
because of \eqref{eq:dimDY(infty)}.

By Stability Theorem \ref{thm:stability}, $B_i^\alpha$ (resp. $B_i^D$) is homeomorphic to 
$B\times I\approx D^2\times I$ (resp. $B\times D(I)\approx D^2\times S^1$). 
Note that the $I$-fiber structure on $\pa B_i^{\alpha}$
defined by $f_{i,0}$ is isotopic to that on $\pa B_i^{\alpha}$ defined by a 
homeomorphism $\pa B_i^{\alpha}\approx \pa D^2\times I$.
Therefore one can define an trivial $I$-bundle $f_i^\alpha: B_i^{\alpha}\to B$
compatible with that on $\pa B_i^{\alpha}$
defined by $f_{i,0}$.
This completes the proof of Lemma \ref{lem:no-exceptional}.
\end{proof}

Next we consider the other case:

\pmed
Case (B). $B_i^D$ has topological singular points.
\pmed

In this case, $(f_i^D)^{-1}(x)$ is a singular $I$-fiber over the point $x=x_{\alpha}\in B$ and $B_i^D\approx B_{\rm pt}$
 (see \cite[Theorem 1.3]{MY}).

\pmed
\begin{lem} \label{lem:local-sing-I}
There exists a map $f_{B,i}': f_i^{-1}(B) \to B$ 
isomorphic to $\pi_{B_\mathrm{pt}^+}$ in \eqref{eq:Mpt+} and 
compatible with the $I$-bundle structure of $f_{i,0}$
over $\pa B$.
\end{lem}
\begin{proof}
%In what follows, we fix $\alpha=1$ and hence omit the
%superscript $\alpha$.
Note that 
$\mathrm{int}^{\mathbf{A}} B_i^\alpha$ contains a unique topological singular point, say $q_i$, where $q_i$ is an endpoint of the segment $(f_i^D)^{-1}(x)$.
We verify that 
\beq \label{eq:Bi-paBi}
B^{\alpha}_i\approx B_\mathrm{pt}^+, \quad 
\pa B^{\alpha}_i \approx \Mo,\quad
\ba_{M^\alpha} B^{\alpha}_i\approx D^2.
\eeq
Let $S$ be a soul of $Y^\alpha$. 
Note that $B^{Y^\alpha}(y_0, 1)$ contains a topological singular point,
which is the limit of $q_i$.
By Stability Theorem \ref{thm:stability}, $B^{\alpha}_i$ is homeomorphic to a large ball around $S$ in $Y^\alpha$.
If $\dim S = 1$, then $Y^\alpha$ is isometric to $\mathbb R\times N/\mathbb Z$, where
$N\approx \mathbb R^2_+$ or $N=\mathbb R\times I$. It turns out that 
$Y^\alpha$ is a manifold with boundary. This is a contradiction.
Thus we have $\dim S = 0$.
By the boundary condition $\partial B^{\alpha}_ i\approx \Mo$,
only the case (C) (i) occurs in Corollary \ref{cor:soul thm ball}.
This implies \eqref{eq:Bi-paBi}.

We identify $(B,x)=(K_1(S^1_\pi),0)$ topologically.
Using the $I$-bundle structure on $\pa B_i^\alpha$ defined by 
$f_{i,0}$ over $\pa B$, construct a map
$\hat h_{i,0}:\pa B_i\to \pa B_{\rm pt}^+$
such that $\pi_{B_{\rm pt}^+}\circ \hat h_{i,0}=f_{i,0}$
on $\pa B_i$. Extend it to a homeomorphism
$\hat h_{i}:\pa B_i\cup \ba B_i\to 
\pa B_{\rm pt}^+\cup\ba B_{\rm pt}^+$.
Finally extend $\hat h_i$ to a homeomorphism
$h_i:B_i\to B_{\rm pt}^+$ using the topological cone structures of 
$B_i$ and $B_{\rm pt}^+$.
Then the required map $f_{B,i}'$ is obtained as
$f_{B,i}'=\pi_{B_{\rm pt}^+}\circ h_i$.
This completes the proof of Lemma \ref{lem:local-sing-I}.
\end{proof}

Combining Lemmas \ref{lem:no-exceptional} and \ref{lem:local-sing-I}, 
we have completed the proof of 
Theorem \ref{thm:dim2nobdry}.
\end{proof}

\section{The case that $X$ is a surface with boundary} \label{sec:2-dim with bdry}

Let $M_i$ be a sequence of three-dimensional Alexandrov spaces with boundary in $\mathcal M(3,D)$ collapsing to $X$. 
In this section, we always assume that 
$X$ is an Alexandrov surface with boundary.
As indicated in Section \ref{sec:intro},
passing to a subsequence, we have one of the following two cases. 
\begin{itemize}
\item[(1)] $M_i$ inradius collapses to $X\,;$
\item[(2)] $M_i$ non-inradius collapses to $X$.
\end{itemize}

As a preliminary for the proof of 
Theorem \ref{thm:dim2wbdy},
following \cite[Section 5]{SY} and 
\cite[Definition 5.1]{MY}, 
we divide $X$ into two closed domains $X'$ and $X''$.

We first construct $X''$.
Let $\e$ be a small positive number, and 
let $C$ be any component of $\ba X$. 
Then we obtain 
an {\it $\e$-regular covering} $\{ B_\alpha, D_\alpha\}_{\alpha=1}^m$ 
of $C$, where $B_\alpha$ and $D_\alpha$ are
closed domains of $X$ satisfying the following.
\begin{itemize}
\item There exists a finite set $\Phi_C=\{ x_\alpha \}_{\alpha=0}^m$ arranged in a cyclic order 
on the circle $C$ with $x_0 := x_m\,;$ 
\item $B_\alpha$ is a small metric ball around $x_\alpha$ of radius,
say $r$, so that $d_{x_\alpha}$ is 
$\e$-regular on $B_{\alpha} \setminus \{ x_\alpha\}\,;$ 
\item Let $\gamma_\alpha$ denote a unique shortest curve joining $x_{\alpha-1}$and $x_{\alpha}$ in $C$, and set
$$
D_\alpha:=B(\gamma_\alpha, \mu) \setminus (\mathring{B}_\alpha \cup \mathring{B}_{\alpha-1}),
$$
with $0< \mu \ll r$.
Then we have 
\begin{enumerate}
\item[(a)]
 $(d_{x_\alpha}, d_{\gamma_\alpha})$ is 
$(\pi/10,\e)$-regular on $D_\alpha \setminus \gamma_\alpha\,;$
\item[(b)]
$\tilde\angle x_{\alpha-1} y x_\alpha >\pi-\e$ holds
for all $y\in D_\alpha$.
\end{enumerate}
\end{itemize} 
Clearly, any corner point in $\ba X$ is contained in 
$\Phi$, where $x \in \ba X$ is called a {\it corner} point if $\diam\, \Sigma_{x} \le \pi /2$.

Let $X_C''$ denote the union of all $B_\alpha$ and 
$D_\alpha$, and set 
\[
  X'':=\bigcup_C X_C'', \quad 
  X':=X\setminus\mathring{X}'',\quad \Phi:=\bigcup_C \Phi_C
\]
where $C$ runs over all the components of $\ba X$.

%Our argument below will simplifies that in \cite[Section 5]{MY}.

Taking smaller $r$ and $\mu$ if necessary, we may assume 
\begin{align*}
B(X'',\mu)\setminus \ba X \subset R_\delta(X), \quad
B(X'',\mu)\setminus \Phi\subset R_\delta^D(X).
\end{align*}
 Take the double $D(M_i)$ of $M_i$ and consider the limit 
\beq\label{eq:def=XD}
  X^D:= \lim_{i\to\infty} D(M_i),
\eeq
which is the union of two isometric copies $X^1$ and $X^2$ of $X$ as a set.
The limit of the reflection $\phi_i$ on $D(M_i)$ along $\ba M_i$ defines an isometric involution $\phi$ on $X^D$ such that $X=X^D/\phi$.
Note that $\phi$ might be trivial. 

Let $\{ y_k\}_{k=1}^K$ denote the finite set $X' \cap \mathcal S_\delta(X^D)$.
For small open balls $U_k$ centered at $y_k$, choose a closed domain $X_0$ of $\mathcal R_\delta(X)$ such that 
\beq\label{eq:choice-X0}
\biggl(X' \setminus \bigcup_{k=1}^K U_k \biggr)
\bigcup (A(\ba X,\mu/2,2\mu)\setminus U(\Phi,r/3)) \subset \mathring{X}_0.
\eeq
Take two copies of $X_0$ as $X_0^{1} \subset X^1$ and $X_0^{2} \subset X^2$, and set 
\beq \label{eq:XD=X1+X2}
X_0^D := X_0^{1} \cup X_0^{2} \subset \mathcal R_\delta(X^D).
\eeq
From Equivariant Fibration Theorem \ref{thm:fibration}, there are a $\phi_i$-invariant closed domain $D(M_i)_0$ of $D(M_i)$ and a $\Z_2$-equivariant $S^1$-bundle 
\begin{equation} \label{eq:equiv fib}
f_{i,0}^D : D(M_i)_0 \to X_0^D,
\end{equation}
which induces a map 
\beq \label{eq:mapfi0}
   f_{i,0}:M_{i,0}\to X_0,
\eeq
where $M_{i,0}=D(M_i)_0/\phi_i$.

\subsection{Case (1) -- Inradius collapse}
\label{ssec:inradius-collapse}

In this subsection, we prove Theorem \ref{thm:dim2wbdy}(1). Let $M_i$ inradius collapse to $X$.

The following lemma holds in the general dimension.

\begin{lem} \label{lem:D(M)toX}
Let $M_i$ be a sequence of $n$-dimensional Alexandrov spaces of curvature $\ge -1$ with boundary and $\diam (M_i)\le D$ 
converging to an Alexandrov space $X$.
Then $M_i$ inradius collapses to $X$ if and only if
$D(M_i)$ converges to $X$, that is,
$X^D=X$ and $\phi$ is the identity. 
\end{lem}
\begin{proof}
Suppose ${\rm inrad}(M_i)\to 0$. This implies that 
$\ba M_i\to X$ under the convergence $M_i\to X$.
For any point $p_i^1\in M^1$, set $p_i^2:=\phi_i(p_i^1)\in M_i^2$.
Then we have $|p_i^1,p_i^2|=2|p_i^1,\ba M_i|\to 0$.
This implies $X^1=X^2$ as required.

If ${\rm inrad}(M_i)>c>0$ for $c$ independent of $i$, take $p_i\in M_i$ with $B(p_i,c)\subset
{\rm int}^A M_i$. Let $p_i^1\in M_i^1$ 
%and $p_i^2=\phi_i(p_i^1)\in M_i^2$ 
be the point corresponding to $p_i$, and $x\in X^1$ the limit of $p_i^1$.
Then $B(x,c/2)$ does not meet $X^2$,
and hence $X\neq X^D$.
\end{proof}

By Lemma \ref{lem:D(M)toX}, each 
$S^1$-fiber of $f_{i,0}^D : D(M_i)_0 \to X_0$ 
is $\phi_i$-invariant, and hence has
exactly two $\phi_i$-fixed points.
This implies that 
\beqq
\text{ 
$f_{i,0}:M_{i,0}\to X_0$ in \eqref{eq:mapfi0}
is an $I$-bundle.}
\eeqq

On the other hand, since $M_i$ inradius collapses to $X$ and 
since $X'$ is away from the boundary $\ba X$,
we can apply Theorem \ref{thm:dim2nobdry}
(see also Lemma \ref{lem:MtoD(M)})
to obtain a closed domain $M_i' \subset M_i$ and a generalized $I$-bundle
\beq \label{eq:fiprime}
f_i' : \mathring{M}_i' \to \mathring{X}'
\eeq
such that 
\begin{itemize}
\item the double $D(f_i'):D(M_i')\to X'$ 
of $f_i'$ is a generalized Seifert bundle$\,;$
\item $D(f_i')$ is compatible with the $S^1$-bundle structure of $f_{i,0}^D$ on $X_0$.
\end{itemize}

Applying \cite[Theorem 1.5]{MY} to the convergence $D(M_i) \to X$, 
we obtain a generalized Seifert fibration 
\[
f_i^D : D(M_i) \to X
\]
that extends both $f_{i,0}^D$ and $D(f_i')$. 
However, from the construction in \cite[Theorem 1.5]{MY},
$f_i^D$ might not be 
$\phi_i$-invariant near the singular point set
$\ba X\cup S_\delta({\rm int}^{\rm\bf A} X)$.
In the rest of this subsection, we shall deform $f_i^D$ to make it $\phi_i$-invariant, and obtain a generalized 
$I$-bundle $M_i\to X$.

Let us set 
\[
M_i'' := M_i \setminus \mathring{M}_i'. 
\]
We are going to show that $M_i''$ is an $I$-bundle
over $X''$.
Take $p_{\alpha,i}\in\ba M_i$ converging to $x_\alpha$,
and choose an arc $\gamma_{\alpha,i}$ in $\ba M_i$ joining $p_{\alpha-1,i}$
and $p_{\alpha,i}$ and converging to $\gamma_\alpha$ as $i\to\infty$.
In what follows, since we fix $1\le \alpha\le m$, we often omit the subscript $\alpha$ and set
\beq \label{eq:omit-alpha}
\begin{aligned}
&x := x_\alpha, \quad p_i:=p_{\alpha,i},\quad p_i':=p_{\alpha+1,i},\quad p_i'':=p_{\alpha-1,i}, \\
&B := B_\alpha,\quad B' := B_{\alpha+1},\quad
 B'':= B_{\alpha-1},\quad
D:= D_{\alpha}, \quad D':=D_{\alpha+1}, \\
&\gamma:=\gamma_\alpha, \quad
\gamma':=\gamma_{\alpha+1},\quad
\gamma_i:=\gamma_{\alpha,i}, \quad
\gamma_i':=\gamma_{\alpha+1,i}.
\end{aligned}
\eeq

\begin{claim} \label{claim:flowB}
There exists a closed domain $B_i$ around 
$p_i$ satisfying  
\begin{enumerate}
\item $B_i\approx B(p_i,r)\,;$
\item $B_i$ converges to $B\,;$
\item $\pa B_i \cap \{d_{\gamma_i\cup\gamma_i'}\ge \mu/2\}\subset f_{i,0}^{-1}(\pa B\cap X_0)\,$
\item $\pa B_i \cap \{d_{\gamma_i\cup\gamma_i'}\le \mu/3\}\subset\pa B(p_i,r)\,$
\end{enumerate}
\end{claim}
\begin{proof}
Using Theorem \ref{thm:smooth approximation},
we construct a $\phi_i$-invariant gradient-like flow $\psi_i^D$ for $d_{p_i}$ on $A^{D(M_i)}(p_i, r/2, 2r)$.
This flow induces a flow $\psi_i$ on $A^{M_i}(p_i, r/2, 2r)$. Making use of $\psi_i$, we connect the two level sets
$\pa B(p_i,r)$ and $f_{i,0}^{-1}(\pa B\cap X_0)$ depending on the 
distance from $\gamma_i\cup\gamma_i'$ to construct $\pa B_i$ satisfying (3), (4).
We omit the detail of the construction of $\pa B_i$
since it is standard.
Let $B_i$ be the domain 
of $M_i$ bounded by $\pa B_i$ and
containing $p_i$. Then (1) and (2) are immediate.
\end{proof}

Considering 
the following slightly longer closed domain 
\beq \label{eq:Dalpha0}
D_{i,0}:=B(\gamma_i,2\mu/3)\cup 
f_{i,0}^{-1}(A(\gamma,\mu/2,\mu)\cap X_0)
\setminus (B(p_i,r/2)\cup B(p_i'',r/2)),
\eeq
we define $D_i$ and $F_i$ by 
\begin{align} \label{eq:Di}
D_i := D_{i,0}\setminus
   (\mathring{B}_i\cup\mathring{B}_i''),\quad
F_i := \pa B_i\cap D_i.
\end{align}
We define 
$D_i'$ using $\gamma',\gamma_i'$ in a similar way, and set
$F_i':=\pa B_i\cap D_i'$.

Consider the arcs
 $F := D\cap \pa B$ and $F' := D' \cap \pa B$
corresponding to $F_i$ and $F_i'$.
Since $f_i'$ is an $I$-bundle over a neighborhood of $\pa X'$, we have
\begin{align}\label{eq:B-FcupF'}
\mathrm{cl}_{M_i} (\pa B_i \setminus (F_i \cup F_i')) = f_i^{-1} (\mathrm{cl}_X(\pa B \setminus (F \cup F')))\approx D^2.
\end{align}

\begin{lem} \label{lem:D_iFi=D3D2} 
Both 
$(D_i,F_i)$ and $(D_i',F_i') $ are homeomorphic to $(D^3,D^2)$.
\end{lem}
\begin{proof}
Let $\{ y\}:=\ba X \cap F$.
From the assumption, we can choose $y_i \in\ba M_i$ converging to $y$.
Note that a neighborhood of $y_i$ in $M_i$ 
(resp. in $D(M_i)$) is homeomorphic to
the form $F_i\times I$ (resp. $D(F_i)\times I)$,
and that $\pa D(F_i)$ is a circle.
From the generalized Margulis lemma (\cite{FY}, 
\cite{Y conv}, \cite{KPT}, \cite{KW}), 
$F_i$ is homeomorphic to one of 
$\{ D^2, S^1\times I, \Mo \}$,
and $D(F_i)$ is homeomorphic to one of 
$\{ D^2, \Mo \}$.
However, if $F_i$ is homeomorphic to either 
$S^1\times I$ or $\Mo$, then 
$D(F_i)$ is homeomorphic to neither $D^2$ nor
$\Mo$.
Thus we have $F_i\approx D^2$ as required. 
Similarly we have $F_i'\approx D^2$. 
 
 To show $D_i\approx D^3$, we use the same argument 
 as in \cite[Assertions 5.5, 5.6]{MY}. This is done by the flow curves  (see also Theorem \ref
{thm:smooth approximation}),
of a smooth vector field on a smooth approximation. Since the idea is the same, we omit the detail. This completes the proof. 
\end{proof}

From \eqref{eq:B-FcupF'} and Lemma \ref{lem:D_iFi=D3D2}, we have
\beq \label{cor:pa B is D2}
\pa B_i\approx D^2.
\eeq 

\begin{lem} \label{lem:top B^M}
$B_i \approx D^3$.
\end{lem}
\begin{proof}
Under the homeomorphism 
$\psi_i:\pa B_i \to D^2$, obviously we have 
$\psi_i^{-1}(\pa D^2)\subset \ba M_i$
and  $\psi_i^{-1}(\mathring{D}^2)\subset 
{\rm int}^{\bf A} M_i$.

Thus we have $\pa B_i\cap\ba M_i\approx S^1$.
Let $B^D_i:=D(B_i)$.
It follows that 
\begin{align} \label{paBD=S2}
\pa B_i^D=\pa B_i^1\cup_{S^1} \pa B_i^2\approx S^2,
\end{align}
where $B_i^1$ and $B_i^2$ denotes two copies of 
$B_i$.

To show that $B_i\approx D^3$, we consider 
the convergence $B_i^D\to B$.
Suppose that $B_i$ is not a three-disk. 
Then neither is $B_i^D$.
Since $B_i^D$
is not homeomorphic to $K_1(P^2)$ 
from \eqref{paBD=S2}, 
we can rescale $B_i^D$
to have the convergence
\begin{equation*} \label{eq:conv01}
(Y^D, Y, y_0) := \lim_{i \to \infty} \biggl(\frac{1}{\delta_i} B_i^D, \frac{1}{\delta_i} B_i, \hat p_i \biggr)
\end{equation*}
together with 
$B_i^D\approx B(Y^D)$ and $\dim Y^D(\infty)\ge 1$,
where $B(Y^D)$ is a large closed ball around the soul
of the three-dimensional Alexandrov space $Y^D$ with
nonnegative curvature.
It follows from \eqref{paBD=S2} and the list
in Corollary \ref{cor:pa B is D2}
that $B(Y^D)\approx D^3$.
Thus we have $B_i\approx B(Y)\approx D^3$,
yielding a contradiction.
\end{proof}

\begin{lem}\label{lem:projectionf''D0}
There is a map 
\begin{equation} %\label{eq:f_721}
f_{D,i}'' : D_i \to D
\end{equation}
which is isomorphic to the trivial $I$-bundle over $D$
and is compatible with $f_{i,0}$ on $\pa D\cap X_0$.
\end{lem}
\begin{proof}
Let $\pi:D\times I\to D$ be the projection.
In what follows, we may assume that $\pa D\cap X_0$ is an arc.
Set $E_{i,0}:=f_{i,0}^{-1}(\pa D\cap X_0)$.
From the construction of $D_i$, we have $E_{i,0}\subset\pa D_i$.
Choose a homeomorphism
$$
h_{i,0}:E_{i,0}\to (\pa D\cap X_0)\times I\subset\pa(D\times I)
$$ 
such that $\pi\circ h_{i,0}=f_{i,0}$ on $E_{i,0}$.
Noting $E_{i,0}\approx D^2$ and $\pa D_i\cup\ba D_i\approx S^2$, we extend $h_{i,0}$ to a homeomorphism
$h_i':\pa D_i\cup\ba D_i\to
\pa(D\times I)$. Finally,
extend $h_i'$ to a homeomorphism 
$h_i: D_i\to D\times I$ by making use of topological 
cone structures of $D_i$ and $D\times I$.
The required map is defined as 
$f_{D,i}'':=\pi\circ h_i$. 
\end{proof}

In view of Lemma \ref{lem:top B^M}, we can also prove the following lemma in the same
way as the proof of Lemma \ref{lem:projectionf''D0},
and hence omit the proof. 

\begin{lem}\label{lem:projectionf''B0}
There is a map 
\begin{equation} %\label{eq:f_721}
f_{B,i}'' : B_i \to B
\end{equation}
which is isomorphic to the trivial $I$-bundle over $B$
and compatible with $f_{i,0}$, $f_{D,i}''$ and $f_{D',i}''$ on $\pa B_i$.
\end{lem}

Combining Lemmas \ref{lem:projectionf''D0} and
\ref{lem:projectionf''B0}, 
we can define an $I$-bundle 
\[
 f_i'' : M_i'' \to X''. 
\]
which is compatible with the generalized $I$-bundle $f_i': \mathring{M}_i' \to \mathring{X}'$ in \eqref{eq:fiprime}.
Combining $f_i'$ and $f_i''$, we have a 
generalized $I$-bundle $f_i:M_i\to X$.
This completes the proof of Theorem \ref{thm:dim2wbdy}(1).

\pmed

\subsection{Case (2) -- Non-inradius collapse}

In this subsection, we prove Theorem \ref{thm:dim2wbdy}(2).
Let $M_i$ non-inradius collapse to $X$, and let $X^D=X^1\cup X^2$ be as in \eqref{eq:XD=X1+X2},
and let $Z\subset X^D$ be the limit of $\ba M_i$
under the convergence $D(M_i)\to X^D$. 
Let $\pi:X^D\to X$ be the projection.

We begin with the following. 

\begin{lem} \label{lem:X1ZX2}
$X^D$ is the gluing of $X_1$ and $X_2$ along with $Z$.
More precisely we have the following.
\begin{enumerate}
\item $Z=\pa X^1=\pa X^2$, where $\pa X^\alpha$ denotes the topological boundary of $X^\alpha$ in $X^D$\,
$(\alpha=1,2)\,;$
\item $\pi(Z)$ is an extremal subset of $X$ contained in 
$\ba X$ whose component
is either an edge or an circle.
\end{enumerate}
\end{lem}
\begin{proof} (1)
It is easy to see that $Z=X^1\cap X^2$. Therefore, 
to verify (1), it suffices to show $Z\subset \pa X^1\cap\pa X^2$.
Note that $Z$ is the fixed point set in $X^D$ of the isometric
involution $\phi$, which is non-trivial from our assumption.
Thus we have $\dim Z<\dim X$, where $\dim Z$ is the dimension of $\pi(Z)$ as an extremal subset of $X$. 
This implies that 
$\Sigma_z(Z)\subsetneqq\Sigma_z(X^\alpha)$
for any $z\in Z$.
Therefore if we take a curve $\gamma(t)$ in $X^1$ with
$\dot\gamma(0)=\xi$ for any 
$\xi\in\Sigma_z(X^1)\setminus\Sigma_z(Z)$, we see that 
$\gamma(\e)\in\mathring{X}^1$ for small enough $\e>0$, and therefore $z\in\pa X^1$.
Since $z\in\pa X^2$ similarly, we have 
$Z=\pa X^1\cap\pa X^2$.
This yields that 
$$
  \Sigma_z(Z)=\Sigma_z(\pa X^1)=\Sigma_z(\pa X^2),
  \quad
 \Sigma_z(X^D)=\Sigma_z(X^1)\cup_{\Sigma_z(Z)}\Sigma_z(X_2).
$$

(2)\, 
 Suppose that a component 
of $Z$ is a point, say $\{ z\}$.
If $U$ is a small neighborhood of $z$ in $X^D$,
then $U\setminus \{ z\}$ must be disconnected.
This contradicts the fact that $X^D$ is a two-dimensional
Alexandrov space.
\end{proof} 

The argument in the proof of 
Lemma \ref{lem:X1ZX2} can be generalized as follows.

\begin{prop} \label{prop:criterion=inradius}
Let a sequence $M_i$ of $n$-dimensional Alexandrov spaces 
with boundary in $\ca M(n,D)$ converges to 
an Alexandrov space $X$ without boundary.
Then $M_i$ inradius collapses to $X$.
\end{prop}

See also \cite{YZ2} for a related result.

\begin{proof}[Proof of Proposition \ref{prop:criterion=inradius}]
Let $X^D$ and $Z\subset X^D$ be defined as in Lemma \ref{lem:X1ZX2}.
It suffices to show that $\phi$ is the identity on $X^D$.
If this does not hold, we have 
$\dim Z\le \dim X-1$.
It follows from the above argument that
\beq \label{eq:ZX12general}
X^\alpha\subsetneqq X^D, \quad Z=\pa X^1=\pa X^2.
\eeq
Observe that $\dim Z=n-1$, which follows from
the curvature condition of $X^D$ and the fact that any geodesic joining points of 
$x_1\in X^1\setminus \pa X^1$ and
$x_2\in X^2\setminus \pa X^2$ must 
intersects $Z$.
Since $X_1$ is closed, it is possible to choose a point $z\in Z\cap \ca R_\delta(X^1)$
(see Theorem \ref{thm:regular-measure}).
Obviously, a neighborhood $U$ of $z$ in $X^1$ homeomorphic to $\R^n$. 
Let $y\in U\setminus Z$ be a point nearby $z$,
and let $w$ be a nearest point of $Z$ from $y$.
Then $\phi(y)\in X^2\setminus X^1$, and 
the union $yw\cup w\phi(y)$ forms a 
geodesic in $X^D$.
However the topology of $X^1=X$ coincides with 
the relative topology of $X^1\subset X^D$.
Therefore $yw\cup w\phi(y)$ must be contained 
in $U\subset X^1$ if $y$ is sufficiently close to 
$z$. This is a contradiction.
\end{proof}

\begin{rem} \label{rem:Omega} \upshape
\cite[Lemma 5.4]{MY} shows that 
 each point of $\Omega$ is also a corner point of $\ba X$.
\end{rem}

We take a closed domain $X_{1}$ of $X$ as 
\[
 X_{1}:= X_0\bigcup\biggl(\bigcup_{\alpha\in A_Z} 
 D_\alpha\biggr),
\]
where $X_0$ is as in \eqref{eq:choice-X0} and 
$A_Z$ denotes the set of all $\alpha$ such that 
$D_\alpha$ meet $Z$.

Let $X_1^D:=X_1\cup\phi(X_1)$.
Since $X_1^D\subset R_\delta(X^D)$, we obtain 
a closed domain $D(M_i)_1$ of $D(M_i)$ and an $\Z_2$-equivariant $S^1$-bundle 
\[
\text{
$f_{i,1}^D:D(M_i)_1\to X_1^D$ }
\]
in a way similar to \eqref{eq:equiv fib}, which 
induces a map
\beq\label{eq:noninr-fi1}
f_{i,1} : M_{i,1} \to X_1,
\eeq
where $M_{i,1}:=D(M_i)_1/\phi_i$.
By Lemma \ref{lem:X1ZX2}, both $X_1^D$ and $D(M_i)_1$ are connected.

\begin{lem} \label{lem:S1-bundle/Z}
$f_{i,1}$ is an $S^1$-bundle over a neighborhood of $Z\cap X_1$.
\end{lem}
\begin{proof}
It suffices to show that $\phi_i$ leaves $f_{i,1}^D$-fibers
over $Z\cap X_1^D$ fixed.
Otherwise, $f_{i,1}^{-1}(z)$ is an arc for a point
$z\in Z\cap X_1^D$, 
 and therefore $f_{i,1}^{-1}(z)$ has a neighborhood of the form $L^2\times [0,1)$,
which implies that
$f_{i,1}^{-1}(z)\subset {\rm int}^{\bf A} M_i$.
This is a contradiction since $Z=\lim_{i\to\infty}\ba M_i$.
\end{proof}

Let $X=X'\cup X''$ be the decomposition of $X$ as in 
Subsection \ref{ssec:inradius-collapse}.
In view of Lemma \ref{lem:X1ZX2}, we can apply 
 \cite[Theorem 1.3]{MY} to obtain 
 a generalized Seifert bundle 
\[
f_i' : \mathring{M}_i' \to \mathring{X}'
\]
for some closed domain $M_i'\subset \mathrm{int}^{\bf A} M_i$
converging to $X'$ such that $f_i'$ is compatible with $f_{i,1}$.
Set 
\[
M_i'' := M_i \setminus \mathring{M}_i',
\]
which converges to $X''$. 

Let $\{ B_\alpha, D_\alpha\}$ be as in 
the beginning of Section \ref{sec:2-dim with bdry}.
The basic strategy for the proof of Theorem \ref{thm:dim2wbdy}(2) is the same as 
that in the argument in Case (1) of the previous subsection.
Namely we shall construct closed domains $B_{\alpha,i}$, $D_{\alpha,i}$ of $M_i$ dividing $M_i''$, and 
(singular) $S^1$ bundles $f_{B_\alpha,i}'':B_{\alpha,i}\to B_\alpha$ and $f_{D_\alpha,i}'':D_{\alpha,i}\to D_\alpha$
compatible with each other and with $f_i'$.

In what follows, we fix an $\alpha$ and hence omit the subscript $\alpha$, and set 
$B:=B_\alpha$, $D:= D_{\alpha}$ and 
$D':=D_{\alpha+1}$, etc., as in \eqref{eq:omit-alpha}.
In some cases, the constructions of $D_i$ and 
$f_{D,i}''$ are almost immediate.
In fact, Lemma \ref{lem:S1-bundle/Z} shows the following.

\begin{cor} \label{cor:S1onZ}
If $D\cap Z$ is nonempty, we set
\beqq \label{cor:def=Di}
D_i:=f_{i,1}^{-1}(D).
\eeqq
Then $f_{D,i}'':=f_{i,1}|_D:D_{i}\to D$ is an $S^1$-bundle.
\end{cor}

Next suppose $D\cap Z$ is empty, and let us define a closed domain slightly larger than $D$
as
\[
  D_0:=B(\gamma,2\mu)\setminus (U(x,r/2)\cup U(x'',r/2)).
\]
In this case, a neighborhood of $M_i$ converging to $D_0$
has no Alexandrov boundary points, and therefore we can 
apply the results in \cite[Section 5]{MY}
to obtain a closed domain $W_i$ converging to $D_0$
and homeomorphic to $D^3$ or 
$\Mo\times I$.

\begin{lem} \label{lem:S1onZ}
Suppose that $D\cap Z$ is empty and 
and $W_i\approx \Mo\times I$.
Then there are a closed domain $D_i$ converging to $D$
and a map 
$$
  f_{D,i}'':D_i \to D,
$$
which is 
isomorphic to $\pi_{\Mo}\times {\rm id}_I$
and compatible with $f_{i,1}$ over $D\cap X_1$.
\end{lem}
\begin{proof}
Let $\tilde W_i$ be the orientable double cover of $W_i$ with the deck transformation $\psi_i$. We may assume $(\tilde W_i,\psi_i)$ converges to
$(\tilde D_0,\psi)$ with $\tilde D_0/\psi=D_0$.
If $\psi$ is the identity, then we would have $\tilde D_0=D_0$.
However, by \cite[Section 5]{MY}, $S^1\times (-1,1)\times (0,1)$ never 
collapses to $[0,1)\times (0,1)$ under a lower Alexandrov curvature bound. This is a contradiction.
Thus $\psi$ is not the identity.

Next remark that $\psi$ has fixed points. In fact, if 
$\psi$ has no fixed points, then $\tilde D_0$ is a nontrivial covering
of $D_0$. This is a contradiction.
Let $Z$ be the fixed point set of $\psi$. Let $\pi:\tilde D_0\to D_0$
be the projection.
Since $\pi(Z)$ is extremal in $D_0$, we see that 
$\pi(Z)=\ba D_0$ and therefore we have 
$\tilde D_0=D_0\cup_Z D_0$, which is the gluing of two copies of $D_0$
along $\ba D_0$, and $\psi$ is the reflection about $\ba D_0$.

 Therefore letting $\tilde D:=D\cup_{\ba D} D$, 
we have a $\Z_2$-equivariant $S^1$-bundle 
$\tilde g_i:\tilde D_i\to \tilde D$ for a closed domain
$\tilde D_i$ of $\tilde W_i$, which induces 
a singular $S^1$-bundle $f_{D,i}'':D_i\to D$ satisfying the required properties. 
\end{proof}

\pmed
\n
{\bf Construction of $B_i$}.\, 
In what follows, we construct a closed domain $B_i$
corresponding to $B$. 
The way of construction depends on several cases.

For $D$, let us take the above $W_i$ converging to $D_0$.
Note that $W_i\approx G_i\times I$ for some
surface $G_i$ homeomorphic to $D^2$ or $\Mo$.
Let $W_i'\approx G_i'\times I$ be defined for $D'$ similarly.
Letting $J$ be the closure of $\pa B\setminus (F\cup F')$,
we consider the gluing
\[
  \pa B= F\cup J\cup F'.
\]
We first define $\pa B_i$ as the union
\beq \label{eq:gluing=paBi}
 \pa B_i=F_i\cup f_{i,1}^{-1}(J)\cup F_i'
\eeq
of three surfaces
$f_{i,1}^{-1}(J)$, $F_i$ and $F_i'$, where $F_i$ and $F_i'$ are
the surfaces corresponding to $F$ and $F'$ respectively,
defined below. We focus on the construction of $F_i$, since 
$F_i'$ is defined in the same way.

\pmed
\n
{\bf Case A)} \, $B\cap Z=\emptyset$.
\pmed\n
In this case, $G_i$ is homeomorphic to $D^2$ or $\Mo$
(\cite[Section 5.2]{MY}).
We have essentially the following three cases:
\[
 (G_i,G_i')\approx (D^2,D^2), \quad (\Mo,\Mo), \quad (D^2,\Mo).
\]
If $G_i\approx D^2$, we can construct $F_i$ satisfying the following
in the same way as Claim \ref{claim:flowB}:
\begin{itemize}
\item $F_i$ converges to $F\,;$
\item $F_i \cap \{ d_{\gamma_i}\ge \mu/2\}\subset f_{i,1}^{-1}(F\cap X_0)\,;$
\item $F_i \cap \{d_{\gamma_i}\le \mu/3\}\subset\pa B(p_i,r)\,;$
\item $d_{p_i}$-flow curves are transversal to $F_i$.
\end{itemize}
\pmed
If $W_i\approx \Mo$, we define $F_i$ as 
\beq \label{eq:def=Fi}
  F_i := (f_{D,i}'')^{-1}(F),
\eeq
where $f_{D,i}''$ is the projection constructed in 
Lemma \ref{lem:S1onZ}.

\pmed
\n
{\bf Case B)} \, $B\cap Z\neq \emptyset$
\pmed\n
In this case, at least one of $D$ and $D'$
meets $Z$. In view of Corollary \ref{cor:S1onZ},
we have 
\beqq
G_i\approx
 \begin{cases} 
 \text{$D^2$ or $\Mo$} & \text{if $D\cap Z=\emptyset$} 
 \\ 
 S^1\times I & \text{if $D\cap Z\neq \emptyset$}.
\end{cases}
\eeqq
In what follows, we may assume $D'\cap Z\neq \emptyset$ without loss of generality.
Then we have 
\[
 (G_i,G_i')\approx (D^2,S^1\times I), \quad (\Mo,S^1\times I), \quad (S^1\times I, S^1\times I).
\]
If $G_i$ is homeomorphic to $D^2$ or $\Mo$, we define $F_i$ 
in the same way as in Case A).
If $G_i\approx S^1\times I$, we define $F_i$ as
\[
  F_i := f_{i,1}^{-1}(F).
\]
Since $F_i'$ is defined in the same way,
$\pa B_i$ is just defined by \eqref{eq:gluing=paBi} in all the cases.
We let $B_i$ be the closed domain bounded by $\pa B_i$ and
containing $p_i$.
Note that $G_i\approx F_i$.

\pmed
\n
{\bf Construction of $D_i$}. \, 
We define a closed domain $D_i$
corresponding to $D$. 
However, the domain $D_i$ is already defined in the case when
$F_i$ is homeomorphic to $S^1\times I$ or 
$\Mo\times I$ (see Corollary \ref{cor:S1onZ} and Lemma \ref{lem:S1onZ}).
In the case when $F_i\approx D^2$, we define $D_i$ as in \eqref{eq:Di}. From construction, we always have
\[
  D_i\approx F_i\times I.
\]

\pmed\n
{\bf Topology of $B_i$}.

\pmed\n
In Case A), we have $\ba B_i = \emptyset$.
Therefore, we can apply \cite[Lemmas 5.3 and 5.4]{MY} to the convergence $B_i\to B$ to 
see that $F_i$ and $F_i'$ are homeomorphic to $D^2$ or $\Mo$.
In the present case, we can apply \cite[Lemma 5.4]{MY} to determine the topology of $B_i$ via those of 
$F_i$ and $F_i'$ as follows.

\begin{lem}$($\cite{MY}$)$\label{lem:topBiCaseA}
In Case A$)$, we have 
\beq
(B_i,\pa B_i)\approx 
 \begin{cases}
(D^2 \times I,S^2) & \text{if $(F_i,F_i')\approx (D^2,D^2)$} \\
(\Mo \times I,K^2) & \text{if $(F_i,F_i')\approx (\Mo,\Mo)$} \\
(K_1(P^2),P^2) &\text{if $(F_i,F_i')\approx (D^2,\Mo)$}. 
 \end{cases}
\eeq
\end{lem}
\pmed

%%%%%%%%%%%%

Next suppose Case B).
Then the center $x$ of $B$ is contained in $Z$. Therefore
we may assume that the center $p_i$ of $B_i$ is an Alexandrov-boundary point. 

Recall that $D'\cap Z \neq \emptyset$, which implies 
$F_i'\approx S^1\times I$, and that 
we have the following three possibilities:
\begin{equation*} \label{eq:3 cases}
F_i \approx D^2, S^1 \times I \text{ or } \Mo,
\quad
\pa B_i \approx D^2, S^1 \times I \text{ or } \Mo. 
\end{equation*}
 
\pmed
\begin{lem} \label{lem:7.10} 
In Case B$)$, we have 
\beq \label{eq:B_i}
(B_i, \pa B_i)\approx 
 \begin{cases}
 (D^3, D^2) & \text{if $(F_i,F_i')\approx (D^2,S^1\times I)$} \\
(\Mo \times [0,1], \Mo \times \{0\}) &\text{if $(F_i,F_i')\approx (\Mo,S^1 \times I)$} \\
(S^1 \times D^2, S^1 \times I) & \text{if $(F_i,F_i')\approx (S^1 \times I, S^1 \times I)$}. 
 \end{cases}
\eeq

\end{lem}
\begin{proof}
If $B_i$ can not be rescaled, then $\pa B_i \approx \Sigma_{p_i}M_i \approx D^2$ and $B_i \approx D^3_+$,
which is the first case of \eqref{eq:B_i}. 
Suppose that $B_i$ can be rescaled, and let $Y$ denote the rescaling limit of $B_i$, which is a three-dimensional noncompact Alexandrov space of nonnegative curvature with boundary satisfying 
\beq \label{eq:Y(infty)-S}
\begin{aligned} \begin{cases}
&\text{$\dim Y(\infty) \ge 1$ and $\dim S\le 1$, where 
$S$ is a soul of $Y\,;$} \\
& \text{$B_i$ is homeomorphic to a large metric ball $B^Y$ around $S$}.
\end{cases}
\end{aligned}
\eeq
If $\dim S = 1$, then $Y$ is an $N$-bundle over $S^1$, 
where $N$ is a noncompact nonnegatively curved Alexandrov surface with boundary satisfying 
$\dim N(\infty) \ge 1$.
Hence $N \approx \mathbb R_+^2$, 
and $Y$ is homeomorphic to $S^1 \times \mathbb R_+^2$ or the twisted product $S^1 \tilde \times \mathbb R_+^2$ defined as $S^1 \tilde \times \mathbb R_+^2 = S^1 \times \mathbb R \times \mathbb R_+ / (z,x,y) \sim (-z,-x,y)$. 
This shows that 
$$
\text{
$(B_i, \pa B_i) \approx (S^1 \times D_+^2, S^1 \times I)$ \,or\, $(S^1 \tilde \times D_+^2, \Mo)$,}
$$
which are the third and the second cases of \eqref{eq:B_i} respectively.

We assume $\dim S = 0$. \eqref{eq:3 cases}
implies that $Y$ has exactly one end.
It follows from Theorem \ref{soul theorem} and 
\eqref{eq:Y(infty)-S} that $Y$ is homeomorphic to one of 
\beq \label{eq:S0Y}
\mathbb R_+^3, \quad \mathbb R \times D^2 / \mathbb Z_2, \quad \mathbb R^2 \times [-1,1] / \mathbb Z_2,
\eeq
where the $\mathbb Z_2$-actions are the restrictions of the involution $v \mapsto -v$ of $\mathbb R^3$.
The case $Y \approx \mathbb R^3_+$
corresponds to the first case in \eqref{eq:B_i}. 

We are going to show that the other cases in 
\eqref{eq:S0Y} never occur.
In the case $Y\approx \mathbb R \times D^2 / \mathbb Z_2$, $C$ is a geodesic ray in Theorem \ref{soul theorem}, and therefore $Y(\infty)$ is a point.
This is a contradiction. 

Suppose that $Y \approx \mathbb R^2 \times [-1,1] / \mathbb Z_2$.
Then, $B^Y \approx D^2 \times [-1,1] / \mathbb Z_2 \approx K_1(P^2)$ and 
\beq \label{eq:paBi=Mo}
\pa B_i\approx\pa B^Y \approx S^1 \times [-1,1] / \mathbb Z_2 \approx \Mo.
\eeq
Consider the double 
$B_i^D=B_i\cup \phi_i(B_i)\subset D(M_i)$ of $B_i$. 
Let $(B^D, E)$ be the limit of $(B_i^D, \ba B_i)$.
As in Lemma \ref{lem:X1ZX2}, we have 
$$
 B^D=B\cup \phi(B)=B \cup_{E} B\subset X^D,
$$
where $E$ is an arc in $\ba X \cap B$ starting from the corner point $x$, or it coincides with $\ba B$.
If $E=\ba B$, then Corollary \ref{cor:S1onZ} implies that 
$\pa B_i\approx S^1\times I$, contradicting \eqref{eq:paBi=Mo}.

Suppose the other case when 
$E$ is the segment starting from $x$ through 
$\ba D'$.
It follows from the discussion in \cite[Lemma 5.4]{MY}
applied to the convergence $B_i^D\to B^D$
that $B_i^D$ is homeomorphic to $\Mo \times I$.
However $B_i^D$ must be homeomorphic to 
a large metric ball $B^{D(Y)}$ in $D(Y)$,
which contains topological singular points.
This is a contradiction.
\end{proof}
%%%%%%%%%%%%%%%%%%

Here we investigate the components of the space $Z$ in 
Lemma \ref{lem:X1ZX2}.

\begin{lem}\label{lem:Z-(4)}
Let $\{ E_\alpha\}$ be the set of components of $Z$. Then there is a one to one correspondence between 
$\{ E_\alpha\}_\alpha$ and 
the set of components of $\ba M_i$ for large $i$.
\end{lem}

 For the proof, we have to verify the following.
%%%%%%%%
\begin{lem}\label{lem:bdy-dist}
For positive numbers $a$ and $r$, there exists $d=d(D,a,r)>0$ and $\e=\e(D,a,r)>0$ satisfying the following $:$
Let $M$ in $\ca M(3,D)$ and 
an Alexandrov surface $X$ with boundary be such that 
\begin{enumerate}
\item ${\rm inrad}(M)\ge r,\quad$ ${\rm area}(X)\ge a$,
%\item 
\item $d_{GH}(M,X)<\e$.
\end{enumerate}
Then the distance between any distinct boundary components of $\ba M$ is greater than $d$.
\end{lem}
\begin{proof}
If the lemma does not hold, we have sequences
$M_i$ and $X_i$ satisfying the above (1) such that 
$d_{GH}(M_i,X_i)\to 0$ as $i\to\infty$ 
and the distance between some distinct boundary components
$\ba_\alpha M_i$ and $\ba_\beta M_i$ tends to $0$.
Passing to a subsequence, we may assume that 
$M_i$ converges to an Alexandrov surface $X$
with boundary and ${\rm area}(X)\ge a$.
%From the discussion above, 
We may also assume that $\ba_\alpha M_i$ and 
$\ba_\beta M_i$ converge to closed subsets
$E_\alpha$ and $E_\beta$ of $\ba X$, which are
extremal subsets of $X$ ([PP]). From the assumption,
$E_\alpha$ meets $E_\beta$ at the endpoint, say $x$.
It follows that the third case occurs in \eqref{eq:B_i}
at $x$. Namely a closed domain $B_i$ converging 
to a closed neighborhood $B$ of $x$ is homeomorphic
to $S^1\times D^2_+$, where 
$\ba B_i$ is homeomorphic to $S^1\times I$, and this meets both $\ba_\alpha M_i$
and $\ba_\beta M_i$ for large $i$.
This is a contradiction.
\end{proof}

\begin{proof}[Proof of Lemma \ref{lem:Z-(4)}]
As indicated in the above proof, we may assume that 
any component $\ba_\alpha M_i$ of
$\ba M_i$ converges to an extremal subset
$E_\alpha$ contained in $\ba X$.
Lemma \ref{lem:bdy-dist} shows that
$E_\alpha$ does not meet $E_\beta$ for distinct 
$\ba_\alpha M_i$ and $\ba_\beta M_i$.
Note that since each point of $\pa Z$ is a corner point of $\ba X$, the number of $\{ \ba_\alpha M_i\}$ 
is uniformly bounded above by a constant $N_D$.
(see \cite{Fj} for general results in this direction).
\end{proof}

%%%%%%%%%
\pmed\n
{\bf Projection maps.}\,
To finish the proof of Theorem \ref{thm:dim2wbdy}(2),
we have to define projection maps $B_i\to B$
and $D_i\to D$ of singular $S^1$-bundles compatible with each other and with
$f_{i,1}$.
\pmed
\n
{\bf Construction of the projections $B_i\to B$}.

We employ mainly two methods for the construction 
of the projections $B_i\to B$.
The one is to make use of a topological cone
structure of $B_i$ if any, and the other one is to use a collar neighborhood structure of $\ba B$ in
$B$.

\pmed
We begin with Case A).

\begin{lem} \label{lem:projection} In Case A$)$, 
there is a map
\begin{equation*} %\label{eq:f_721}
f_{B,i}'' : B_i \to B
\end{equation*}
which is isomorphic to one of \eqref{eq:K(P2)},
\eqref{eq:MoxId} and \eqref{eq:L2xId} in 
Example \ref{ex:KP2}
and is compatible with $f_{i,1}$ over $\pa B\cap X_1$.
\end{lem}
\begin{proof}
Let $Q:=[0,1]\times [0,1]\subset \R_+\times \R_+$.
First we consider 

\pmed\n
{\bf Case (a)}\,\,$(B_i,\pa B_i)\approx 
(K_1(P^2),P^2)$, $(F_i,F_i')\approx (D^2,\Mo)$.

\pmed
We make use of the cone structure of $B_i$
in this case. 
The case $B_i\approx D^2\times I$ with
$(F_i,F_i')\approx (D^2,D^2)$ is similarly discussed, and hence omitted.
Let
\[
 \pi_{K_1(P^2)}:K_1(P^2)\to Q
\]
be the restriction of the projection $\pi_{K(P^2)}$, which is defined in \eqref{eq:K(P2)}, to $Q$.
For a small $s>0$, set
\begin{align*}
&P:= \pi_{K_1(P^2)}^{-1}(\{ 1\}\times [0,s])\approx D^2, \quad \\
&P':= \pi_{K_1(P^2)}^{-1}([0,s]\times \{ 1\})\approx\Mo.
\end{align*}
Let $h:(B,\pa B, x)\to (Q, \pa Q, (0,0))$ be a homeomorphism such that 
$h(F)=\pi_{K_1(P^2)}(P)$ and 
$h(F')=\pi_{K_1(P^2)}(P')$.
We can choose a homeomorphism $\hat h_i:\pa B_i \to \pa K_1(P_2)$ satisfying the following:
\begin{enumerate}
\item $\hat h_i(F_i)=P$ and $\hat h_i(F_i')=P'\,;$
\item $\pi_{K_1(P^2)}\circ \hat h_i=h\circ f_{i,1}$ on 
 $(f_{i,1})^{-1}(\pa B\cap X_1)$.
\end{enumerate}
To see this, construct $\hat h_i$ 
on $(f_{i,1})^{-1}(\pa B\cap X_1)$ so as to satisfy
(2). Then it suffices to extend it to homeomorphisms $F_i\approx P$,
$F_i' \approx P'$.
We extend $\hat h_i$ to a homeomorphism
$h_i:B_i \to K_1(P^2)$ using the cone structures
of $B_i$ and $K_1(P^2)$.
Finally we define $f_{B,i}'':B_i\to B$ by
$f_{B,i}'':=h^{-1}\circ\pi_{K_1(P^2)}\circ h_i$,
which satisfies the required properties.
%%%%%
\[
\begin{CD}
 & B_i & @ > h_i >> \, & K_1(P^2) & \\
&@V f_{B,i}'' V V& & @ V V \pi_{K_1(P^2)} V & \\ 
 & B & \,\,\, @> h >> & Q &
\end{CD}
\]
%%%%%
\psmall
Next we consider 
\pmed\n
{\bf Case (b)}\,\, $(B_i,\pa B_i)\approx 
(\Mo\times I, K^2)$, $(F_i,F_i')\approx(\Mo, \Mo)$.

\pmed
In place of the cone structure used in the above
Case (a), we make use of a collar neighborhood
of $B$ as used in \cite[Section 5.5]{MY}.
%In what follows, we identify $\ba B$ with $[0,1]$.
%%%%%
Let $h:B\to Q$ be a collar neighborhood
of $\ba B$ such that 
$h(\ba B)=[0,1]\times \{ 0\}$, 
$h(F)=\{ 0\}\times [0,1]$ and 
$h(F')=\{ 1\}\times [0,1]$.
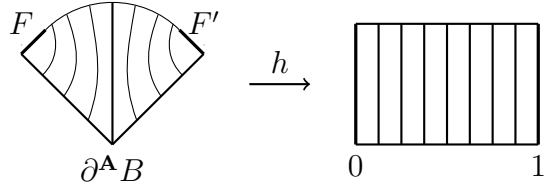
\pmed
%%%%%%%%%%%%%%%%%%%%
\begin{center}
\begin{tikzpicture}
[scale = 0.4]
\fill (0,0) coordinate (A) circle (0pt) node [below] {$\ba B$};
%%%%
\draw[thick](0,0) -- (-3,3);
\draw[thick](0,0) -- (3,3);
\draw[very thick](-3,3)--(-2.2,3.8);
\draw[very thick](3,3)--(2.2,3.8);
\draw[thin](-2.2,3.8) to [out=45,in=135] (2.2,3.8);
\fill (-3, 3.3) circle (0pt) node [above] {$F$};
\fill (3, 3.3) circle (0pt) node [above] {$F'$};
\draw[thick](0,0) -- (0,4.7);
\draw[thin](-2.3,2.3) to [out=45,in=-70] (-2,4);
\draw[thin](2.3,2.3) to [out=135,in=250] (2,4);
\draw[thin](-1.6,1.6) to [out=55,in=-70] (-1.5,4.3);
\draw[thin](1.6,1.6) to [out=125,in=250] (1.5,4.3);
\draw[thin](-0.8, 0.8) to [out=70,in=280] (-0.75,4.6);
\draw[thin](0.8, 0.8) to [out=110,in=260] (0.75,4.6);
\draw[thick](8,0) -- (14,0);
\fill (8,0) circle (0pt) node [below] {$0$};
\fill (14,0) circle (0pt) node [below] {$1$};
\draw[very thick](8,0) -- (8,4);
\draw[very thick](14,0) -- (14,4);
\draw[thick](8,4) -- (14,4);
\draw[thick](11,0) -- (11,4);
\draw[thick](8.75,0) -- (8.75,4);
\draw[thick](13.25,0) -- (13.25,4);
\draw[thick](9.5,0) -- (9.5,4);
\draw[thick](12.5,0) -- (12.5,4);
\draw[thick](10.25,0) -- (10.25,4);
\draw[thick](11.75,0) -- (11.75,4);
\draw[thick, ->] (4.5,2)--(6.5,2);
\fill (5.5,2) circle (0pt) node [above] {$h$};
\end{tikzpicture}
\end{center}
\pmed
%%%%
Choose a homeomorphism $h_i:B_i\to \Mo\times [0,1]$.
Deforming $h_i$,
we may assume that 
$h_i(F_i)=\Mo\times \{ 0\}$ and 
$h_i(F_i')=\Mo\times \{ 1\}$.
For $\alpha=0,1$, let 
$$
 \theta_\alpha:=
 h\circ\tilde f_{i,1}^\alpha\circ
 h_i^{-1}|_{\Mo\times \{\alpha\}}:
  \Mo\to [0,1],
$$
where $\tilde f_{i,1}^0:F_i\to F$
and $\tilde f_{i,1}^1:F_i'\to F'$
are singular $S^1$-bundles
isomorphic to $\pi_{\Mo}$ in Example \ref{ex:toy02} extending $f_{i,1}|_{F\cap X_{1}}$.
Note that $\theta_\alpha$ are also singular $S^1$-bundle
isomorphic to $\pi_{\Mo}$.
Take homeomorphisms $k_\alpha$ of $\Mo$
such that $\theta_\alpha\circ k_\alpha=\pi_{\Mo}$.
Since the mapping class group of $\Mo$ is $\Z_2$ 
whose nontrivial element is represented by a homeomorphism 
reversing the orientation of $\pa \Mo$, we may assume
that $k_1$ is isotopic to $k_0$ by replacing
$k_0$ if necessary.
Let $k:\Mo\times [0,1]\to\Mo$ be an isotopy from 
$k_0$ to $k_1$ with $k_t:=k(\cdot, t)$.
Letting 
\[
 \theta_t:=\pi_{\Mo}\circ k_t^{-1}:\Mo\to [0,1],
\]
define $\theta:\Mo\times [0,1]\to [0,1]$
by $\theta(\cdot,t)=\theta_t$.
Letting $\pi:\Mo\times [0,1]\to [0,1]$ be the projection, we define $f_{B,i}'':B_i\to B$
by $f_{B,i}'':=h^{-1}\circ (\theta,\pi)\circ h_i$,
which satisfies the required properties.
%%%%%
\[
\begin{CD}
 & B_i & @ > h_i >> \, & \Mo\times [0,1] & \\
&@V f_{B,i}'' V V& & @ V V (\theta,\pi) V & \\ 
 & B & \,\,\, @> h >> & [0,1]\times [0,1] &
\end{CD}
\]
%%%%%

\end{proof}

\pmed

Next we consider Case B).

\begin{lem} \label{lem:projection2}
In Case B$)$, there is a singular $S^1$-bundle 
\begin{equation*} %\label{eq:f_721}
f_{B,i}'' : B_i \to B
\end{equation*}
satisfying 
\begin{enumerate}
\item it is isomorphic to one of 
$\{ \pi_{L^2} \times \mathrm{id}_{[0,1]}, 
\pi_{\Mo}\times \mathrm{id}_{[0,1]},
\pi_{S^1\times [0,1]} \times \mathrm{id}_{[0,1]} \}\,;$
\item it is compatible with $f_{i,1}$ over 
$\pa B\cap X_1$.
\end{enumerate}
\end{lem}
\begin{proof}
First we consider

\pmed\n
{\bf Case (c)}\,\,$(B_i,\pa B_i)\approx 
(D^3,D^2)$, $(F_i,F_i')\approx (D^2,S^1\times I)$.
\pmed
This case is similarly discussed as Case (a)
via the cone structure of $D^3$, and hence omitted.

\pmed\n
{\bf Case (d)}\,\,$(B_i,\pa B_i)\approx
(S^1\times D^2,S^1\times I)$, $(F_i,F_i')\approx (S^1\times I, S^1\times I)$.
\pmed

As we see in what follows, 
the argument is similar to Case (b) in the proof of Lemma \ref{lem:projection}.
We make use of a collar neighborhood of 
$B$ again.
 
Let $h:B\to Q$ be a collar neighborhood
of $\ba B$, where we identify $\ba B$ with $[0,1]\times \{0\}$.
%%%%
Set $A:=S^1\times I$ for simplicity, and choose a homeomorphism $h_i:B_i\to A\times [0,1]$.
Deforming $h_i$ by isotopies, we may assume that 
$h_i(F_i)=A\times \{ 0\}$ and $h_i(F_i')=A\times \{ 1\}$.
For $\alpha=0,1$, let 
$$
  \theta_\alpha:=h\circ f_{i,1}\circ h_i^{-1}|_{A\times \{\alpha\}}: A\to [0,1].
$$
Note that $\theta_\alpha$ are trivial $S^1$-bundles.
Take homeomorphisms $k_\alpha$ of $A$
such that $\theta_\alpha\circ k_\alpha=\pi_{A}$.

\begin{claim} \label{claim:isotopy=theta}
There is an isotopy $\tilde k_t$\,$(t\in I)$
of homeomorphisms of $A$ such that 
$\tilde k_0=k_0$ and 
\[
 \theta_t:=\pi_{S^1\times I}\circ \tilde k_t^{-1}:A
    \to [0,1]
\]
defines a one-parameter family of 
$S^1$-bundles 
joining $\theta_0$ and $\theta_1$.
\end{claim}
\begin{proof}
We may assume that $k_\alpha$ preserve the orientation of $\pa A$ and they fix $\pa A$.
Consider the mapping class group
$MCG(A,\pa A)$ of $A$ consisting isotopy classes of $A$ fixing
$\pa A$.
%%%%
Identify 
$A$ with $\{{1\leq |z|\leq 2\}} \subset \mathbb {C}$
and define a homeomorphism $\lambda$
by
$$
\lambda(z)=e^{2i\pi |z|}z,
$$
which is the identity on $\pa A$.

Then $MCG(A,\pa A)$ is generated by the class of 
$\lambda$ (cf.\cite{Iv}).
Therefore $k_0$ is isotopic to $k_1\circ \lambda^m$ 
for some integer $m$.
Let $\tilde k:A\times [0,1]\to A$ be an isotopy from 
$k_0$ to $k_1\circ \lambda^m$ with $\tilde k_t:=\tilde k(\cdot, t)$. Then 
\[
 \theta_t:=\pi_{A}\circ \tilde k_t^{-1}: A\to [0,1]
\]
satisfies the conclusion.
\end{proof}

We now define $\theta: A\to [0,1]$
by $\theta(\cdot,t)=\theta_t$.
Letting $\pi:A\times [0,1]\to [0,1]$ be the projection, we define $f_{B,i}'':B_i\to B$
by $f_{B,i}'':=h^{-1}\circ (\theta,\pi)\circ h_i$,
which satisfies the required properties.
\pmed

Next we consider
\pmed\n
{\bf Case (e)}\,\, $(B_i,\pa B_i)\approx 
(\Mo\times [0,1],\Mo\times \{ 0\})$, 
$(F_i,F_i')\approx (\Mo,S^1\times I)$.
\pmed

%%%%%%%%
For $0<s_0<1/2$, 
take a homeomorphism $h:(B,x)\to (Q,(0,0))$ 
such that 
$h(F)=\{ 1\} \times [0,s_0]$ and 
$h(F')=[0,s_0]\times \{1\}$.
Set $Q_+:=[0,1]\times [1/2, 1]$.
Note that 
$f_{i,1}$ is defined over $B\setminus h^{-1}([0,s]\times [0,s])$ for a fixed $0<s<s_0$.
Subdivide $B_i$ as $B_i=H_i\cup K_i$, where
$$
K_i:=(h\circ f_{i,1})^{-1}(Q_+), \quad 
H_i:=B_i\setminus\mathring{K}_i.
$$ 
Let $\mathrm{\hat M\ddot{o}}:=\pi_{\Mo}^{-1}([0,1/2])
\approx \Mo$.
Let $\pi:S^1\times Q\to Q$ be the projection.
Since $h\circ f_{i,1}:K_i\to Q_+$ is 
a trivial $S^1$-bundle, we have a homeomorphism
$
g_i:K_i\to S^1\times Q_+
$
satisfying 
$$
\pi\circ g_i=h\circ f_{i,1}|_{K_i}.
$$
Take a homeomorphism
$h_i:H_i\to S^1\times Q_-$
%
%\mathrm{\hat M\ddot{o}}\times [0,1]
%$
so as to satisfy
$$ \text{$\pi\circ h_i
=h\circ f_{i,1}$ on $F_i$.}
$$
%%%%%
%%%%%%%%%%%%%%%%%%%%
%\vspace{0.05cm}
\pmed
\begin{center}
\begin{tikzpicture}
[scale = 0.35]
\fill (0,0) coordinate (A) circle (0pt);

\fill (0.15,0.15) circle (2.5pt);
\fill (0.28,0.28) circle (0pt) node [below left] {{\small $p_i$}};
%%%%
\draw[thick](0.5,0) -- (3,0);
\draw[thick](0,0.5) -- (0,3);
\draw[thick](0.5,0) to [out=180,in=270] (0,0.5);
\draw[thin](0,1.5) to [out=0,in=180] (1,1.6);
\draw[thin](1,1.6) to [out=0,in=180] (2,1.4);
\draw[thin](2,1.4) to [out=0,in=180] (2.85,1.5);
\fill (1.3,1.5) circle (0pt) node [above] {{\tiny $K_i$}};
\fill (1.3,1.5) circle (0pt) node [below] {{\tiny $H_i$}};
\draw[very thick](0,3) to [out=0,in=180] (0.7,3);
\draw[very thick](3,0) to [out=90,in=270] (3, 0.7);
\fill (0.35,3) circle (0pt) node [above] {{\tiny $F_i'$}};
\fill (3,0.35) circle (0pt) node [right] {{\tiny $F_i$}};
\draw[thin](0.7,3) to [out=0,in=90] (3, 0.7);

\draw[thick, ->] (5.5,2)--(7.5,2);
\fill (6.5,1.8) circle (0pt) node [above] {{\small $g_i$}};
\draw[thick, ->] (5.5,0.7)--(7.5,0.7);
\fill (6.5,0.6) circle (0pt) node [below] {{\small $h_i$}};
%%%%%%%%%%%%%%%%%%%%%%%

\draw (9.5,0) rectangle(13,2.7);
\draw (9.5,1.35)--(13,1.35); 
\fill (11.25,1.2) circle (0pt) node [above] {{\tiny $S^1\times Q_+$}};
\fill (11.25,1.5) circle (0pt) node [below] {{\tiny $S^1\times Q_-$}};
%%%%%%%%%%%%%%%%%%

\draw[thick, ->] (1.5,-1)--(1.5,-2.3);
\fill (1.7,-1.6) circle (0pt) node [right] {{\tiny $f_{i,1}$}};
%%%%%%%

\fill[shift={(0,-7)}] (0,0) circle (3pt);
\fill [shift={(0,-7)}](0.1,0.1) circle (0pt) node [below left] {{\small $x$}};
%%%%
\draw[shift={(0,-7)},thick](0,0) -- (3,0);
\draw[shift={(0,-7)},thick](0,0) -- (0,3);

\draw[shift={(0,-7)},thin](0,1.5) to [out=0,in=180] (1,1.6);
\draw[shift={(0,-7)},thin](1,1.6) to [out=0,in=180] (2,1.4);
\draw[shift={(0,-7)},thin](2,1.4) to [out=0,in=180] (2.85,1.5);
\fill [shift={(0,-7)}] (1.3,1.5) circle (0pt) node [above] {{\tiny $K$}};
\fill[shift={(0,-7)}] (1.3,1.5) circle (0pt) node [below] {{\tiny $H$}};
\draw[shift={(0,-7)},very thick](0,3) to [out=0,in=180] (0.7,3);
\draw[shift={(0,-7)},very thick](3,0) to [out=90,in=270] (3, 0.7);
\fill [shift={(0,-7)}] (0.35,3) circle (0pt) node [above] {{\tiny $F'$}};
\fill [shift={(0,-7)}] (3,0.35) circle (0pt) node [right] {{\tiny $F$}};
\draw[shift={(0,-7)},thin](0.7,3) to [out=0,in=90] (3, 0.7);
\draw[shift={(0,-7)},thin](0,0.5) to [out=0,in=90] (0.5, 0);

\draw[shift={(0,-7)},thick, ->] (5.5,1.5)--(7.5,1.5);
\fill[shift={(0,-7)}] (6.5,1.5) circle (0pt) node [above] {{\small $h$}};

\draw[shift={(0,-7)}] (9.5,0) rectangle(13,2.7);
\draw[shift={(0,-7)}] (9.5,1.35)--(13,1.35); 
\fill[shift={(0,-7)}] (11.25,1.35) circle (0pt) node [above] {{\tiny $Q_+$}};
\fill[shift={(0,-7)}] (11.25,1.35) circle (0pt) node [below] {{\tiny $Q_-$}};

\draw[shift={(0,-7)},very thick](9.5,2.7) to (10.3, 2.7);
\fill[shift={(0,-7)}] (9.9,2.7) circle (0pt) node [above] {{\tiny $h(F')$}};
\draw[shift={(0,-7)},very thick] (13, 0) to (13,0.8);
\fill[shift={(0,-7)}] (13,0.4) circle (0pt) node [right] {{\tiny $h(F)$}};
\draw[shift={(0,-7)}] (9.5,0) rectangle(10,0.5);
\fill[shift={(0,-7)}] (10.5,0.1) circle (0pt) node [below left] {{\tiny $(0,0)$}};
\draw[thick, ->] (11.25,-1)--(11.25,-2.3);
\fill (11.4,-1.6) circle (0pt) node [right] {{\tiny $\pi$}};
\end{tikzpicture}
\end{center}
\psmall

Set $A:=S^1\times [0,1]$ for simplicity.
To investigate the difference between 
$h_i$ and $g_i$ on the intersection, we consider the homeomorphism $k_{1/2}$ of $S^1\times [0,1]$ defined by
$$
k_{1/2}:= g_i\circ h_i^{-1}|_{A\times \{1/2\}} :A\times \{1/2\} \to A\times \{1/2\}.
$$
Up to isotopies, we may assume that $k_{1/2}$ fixes $\pa A$.
As before, $k_{1/2}$ is isotopic to $\lambda^m$ relative $\pa A$ for some
integer $m$.
Let $\tilde k_t$\,$(t\in[1/2,1])$ be an isotopy 
 such that $\tilde k_{1/2}=k_{1/2}$ and 
$\tilde k_1=\lambda^m$.
Let $\tilde k$ be a homeomorphism of $S^1\times Q_+$ defined by 
$\tilde k(\cdot, t)=(\tilde k_t, t)$.
 Finally define a homeomorphism 
$\tilde g_i:K_i\to S^1\times Q_+$ by
$\tilde g_i:= \tilde k^{-1}\circ g_i$.
Combining $h_i$ and $\tilde g_i$, we have a 
map $\hat h_i: B_i\to S^1\times Q$
such that $\pi\circ \hat h_i$ extends $h\circ f_{i,1}|_{F_i\cup F_i'}$.
The required map $f_{B,i}'':B_i\to B$ is defined as 
$f_{B_i}'':=h^{-1}\circ \pi\circ \hat h_i$.
This completes the proof.
\end{proof}
%%%%%%

\pmed\n
{\bf Construction of the projections $D_i\to D$}.
\pmed
\begin{lem} \label{lem:projection2D}
There is a map 
\begin{equation} \label{eq:f_D2}
f_{D,i}'' : D_i \to D
\end{equation}
satisfying 
\begin{enumerate}
\item it is isomorphic to one of
$\{ \pi_{L^2}\times\mathrm{id}_{[0,1]}, \pi_{\Mo} \times \mathrm{id}_{[0,1]}, \pi_{S^1\times [0,1]}\}\,;$ 
\item it is compatible
with $f_{i,1}$ on $\pa D_i\setminus (F_i\cup F_i')$,
and with $f_{B,i}''$ and $f_{B',i}''$ on $F_i$ and $F_i'$
respectively.
\end{enumerate}
\end{lem}
\begin{proof}
When $F_i$ is homeomorphic to $S^1\times I$ or $\Mo$,
the construction of $f_{D,i}''$ is already covered in 
Corollary \ref{cor:S1onZ} and Lemma \ref{lem:S1onZ}.

Suppose $F_i\approx D^2$.
As the boundary condition, we can define 
the map $\pa f_i:\pa D_i\to \pa D$ by combining 
$f_{i,1}|_{\pa D\setminus(F\cup F'')}$, $f_{B,i}''|_F$ and $f_{B'',i}''|_{F''}$.
It is easy to construct a homeomorphism
$\pa h_i:\pa D_i\to \pa (D^2\times I)$ satisfying 
$h\circ\pa f_i=(\pi_{D^2}\times {\rm id}_I)\circ \pa h_i$,
where $h:D\to I^2$ is a homeomorphism 
sending $F''$ and $F$ to $\{ 0\}\times I$ and $\{ 1\}\times I$ respectively.
Since $D_i\approx D^3$, we can extend 
$\pa h_i$ to a homeomorphism 
$h_i:D_i\to D^2\times I$.
Then the required map $f_{D,i}'':D_i\to D$ is defined as 
$f_{D,i}''=h^{-1}\circ\pi\circ h_i$.
\end{proof}

Combining $f_{i,1}$, $f_{B,i}''$ and $f_{D,i}''$
for all $B$ and $D$, 
we obtain a generalized Seifert fibration 
$f_i:M_i\to X$.
This completes the proof of Theorem \ref{thm:dim2wbdy}(2).
\pmed

\section{The case that $X$ is a circle} \label{sec:circle}
In this section, we assume that the limit space $X$ of $M_i$ is a circle. 
By Proposition \ref{prop:criterion=inradius},
$D(M_i)$ also collapses to $X$. 
By \cite[Theorem 1.7]{MY}, there is a fiber bundle 
\[
f_i^D : D(M_i) \to X
\]
whose fiber is one of $S^2$, $P^2$, $T^2$ and $K^2$.
In particular, $D(M_i)$ is a topological three-manifold, and hence $M_i$ is a topological three-manifold with boundary.

\begin{proof}[Proof of Theorem \ref{thm:main-circle}]
Let $(M_i, p_i)$ converges to $(X,p)$, where we may assume that $X$ is a circle
of length $2\pi$. Fix any $\pi/2<r<2\pi/3$, and consider the ball $B(p_i, r)$.
Note that the topological boundary $\pa B(p_i, r)$ is disconnected.
This implies that $B(p_i, r)$ can be rescaled. 

\begin{claim}\label{claim:circle}
 $B(p_i,r)$ is homeomorphic to one of 
$D^2\times I$, $\Mo\times I$ and $(S^1\times I)\times I$.
\end{claim}
\begin{proof}
Let $(Y,y_0)$ be a rescaling limit of $(\frac{1}{\delta_i}B(p_i,r), \hat p_i)$ with a reference point 
$\hat p_i$ and a rescaling constant $\delta_i$ as in 
Theorem \ref{thm:rescaling}.
Obviously $Y$ splits isometrically as a product $Y=Y_0\times \mathbb R$ with $\dim Y_0\in \{ 1, 2\}$.
Since $\pa B(p_i, r)$ is disconnected, 
$Y_0$ must be compact.
First suppose $\dim Y_0=2$. 
Then
$Y_0$ is homeomorphic to one of $D^2$, $S^1 \times I$ and $\Mo$.
Since
$B(p_i,r)$ is homeomorphic to $Y_0\times I$ by Theorem \ref{thm:stability respectful}, 
we obtain the conclusion. 

Next suppose $\dim Y_0=1$. Then $Y_0$ is a circle or an arc.
If $Y_0$ is a circle, then $Y=Y_0\times\mathbb R$ has no singular points and no boundary. Therefore we can apply Theorem \ref{thm:dim2nobdry} to conclude that 
$B(p_i, r)$ is an $I$-bundle over $S^1\times I$.

Now suppose that $Y_0$ is an arc, and let $L_1$ and $L_2$ be the components of 
$\ba B(y_0, R)$, which are line segments. 

We have the following four possibilities.
\pmed
Case (1)\, $\ba B(p_i, r)$ converges to only one of $L_1$ and $L_2$.

In this case, it follows from \cite[Section 5]{MY} and
Lemma \ref{lem:S1-bundle/Z} that 
the components of $\pa B(p_i, r)$ are homeomorphic to 
$D^2$ or $\Mo$, and therefore $B(p_i, r)$ is homeomorphic to 
$D^2\times I$ or $\Mo\times I$.

\pmed
Case (2)\, $\ba B(p_i, r)$ converges to the union $L_1\cup L_2$.

In this case, we see that the components of $\pa B(p_i, r)$ are homeomorphic to 
$S^1\times I$, and therefore $B(p_i, r)$ is homeomorphic to 
$(S^1\times I)\times I$.

\pmed
Case (3)\, $\ba B(p_i, r)$ converges to $B(y_0, R)$.

In this case, Theorem \ref{thm:dim2nobdry} shows that $B(p_i, r)$ is a trivial $I$-bundle over $B(y_0, R)$,
and therefore it is homeomorphic to 
$D^2\times I$.

\pmed
Case (4). $\ba B(p_i, r)$ is empty.

In this case, from \cite[Section 5]{MY} we see that the components of $\pa B(p_i, r)$ is homeomorphic to 
one of $S^2$, $P^2$ and $K^2$, and therefore $B(p_i, r)$ is homeomorphic to 
one of $S^2\times I$, $P^2\times I$ and $K^2\times I$.
 Repeating the argument finitely many times together with a continuation procedure, we easily have that 
$M_i$ has empty boundary, which is a contradiction.

Thus we obtain that $B(p_i,r)$ is homeomorphic to one of 
$D^2\times I$, $\Mo\times I$ and $(S^1\times I)\times I$ in either case.
\end{proof}

Next take $q_i\in M_i$ with $|p_i, q_i|\to \pi$.
Then in a way similar to \cite[lemma 6.3]{MY}, we conclude 
that $B(p_i,r)\setminus \mathring{B}(q_i,r)$, 
$B(q_i,r)\setminus \mathring{B}(p_i,r)$
and each component of $B(p_i,r)\cap B(q_i,r)$
are all homeomorphic to 
$D^2\times I$, $\Mo\times I$ or $(S^1\times I)\times I$.
Thus $M_i$ is homeomorphic to 
the total space of a fiber bundle over $S^1$ with fiber $\in \{ D^2, \Mo, S^1\times I\}$.
This completes the proof of Theorem \ref{thm:main-circle}.
\end{proof}

\section{The case that $X$ is an interval} \label{sec:interval} 

Let us consider a sequence $M_i$ in $\ca M(3,D)$ 
collapsing to an interval $X$. 
By rescaling, we may assume that $X = [0,1]$. 
Let $E\subset X$ denote the limit of $\ba M_i$ under this convergence. Since $E$ is extremal in $X$, it is one of $\{0,1\}$, $\{0\}$, $\{1\}$ and $X$.

In this section, we show the following.
 
\begin{thm} \label{thm:interval}
Under the above situation, the topology of $M_i$ is determined as follows: 
\begin{enumerate}
\item If $E=\{0,1\}$, then $M_i$ is homeomorphic to $F \times [0,1]$, where $F$ is homeomorphic to one of $S^2, P^2, T^2$ or $K^2\,;$
\item If $E= \{1\}$ $($The case of $E=\{0\}$ is essentially the same as this case$)$, then 
$M_i$ is homeomorphic to a convex set with connected boundary
in a three-dimensional noncompact Alexandrov space of nonnegative curvature $\,;$
\item If $E=X$, then $M_i$ is homeomorphic to 
a three-dimensional compact Alexandrov space of nonnegative curvature.
\end{enumerate}
As a conclusion, in any case, $M_i$ is homeomorphic to a three-dimensional nonnegatively curved Alexandrov space with boundary. 
\end{thm}

For the details in the above Cases (2),(3), see
Tables \ref{table:FB} and \ref{table:gamma-beta}.

Now, we consider the double $D(M_i)$, which collapses to
$X^D$ (see \eqref{eq:def=XD}).
By Lemma \ref{lem:X1ZX2}, we have 
\[
  X^D=X^1\cup_Z X^2,
\]
where $\pi(Z)=E$.
We have 
the following three possibilities depending on 
the above cases:

\begin{enumerate}
\item $X^D$ is a circle, i.e., $X^D = D([0,1])$. \label{case:S1}
\item $X^D = [0,2]$ that is a partial double of $X$. \label{case:partial}
\item $X^D = [0,1]$ that is $X$ itself. 
This case is equivalent to $\mathrm{inrad}(M_i) \to 0$.
\label{case:interval}
\end{enumerate}
Then, by \cite[Theorem 1.8]{MY}, there is a continuous surjection 
\[
{f_i^D} : D(M_i) \to X^D
\]
such that the restriction of ${f_i^D}$ to $X^D \setminus U(\ba X^D )$ is a $\Z_2$-equivariant fiber bundle with fiber homeomorphic to one of $S^2$, $P^2$, $T^2$ and $K^2$.
Here, $U(\ba X^D )$ denotes a small neighborhood of $\ba X^D $ in $X^D$ which is invariant under the involution $\phi$ on $X^D$defined as the limit
of the involution $\phi_i$ on $D(M_i)$.
Note that 
if $\ba X^D \neq \emptyset$, then 
${f_i^D} |_{X^D \setminus U(\ba X^D)}$ is a trivial bundle.

For later use, we divide $M_i$ into three closed domains $B_i$, $C_i$ and $B_i'$:
\[
 M_i=B_i\cup C_i\cup B_i',
\]
where $B_i$ and $B_i'$ are metric balls in $M_i$ centered at $p_i$ and $p_i'$ of radius $< 1/3$ such that $p_i$ and $p_i'$ converge to the boundary points $0$ and $1$ of $[0,1]$ respectively, 
and $C_i:=M_i \setminus (\mathring{B}_i \cup \mathring{B}_i')$. 
We set 
$$
F_i := B_i \cap C_i = \pa B_i, \quad
F_i' := B_i' \cap C_i = \pa B_i'.
$$
In view of Theorems \ref{thm:stability respectful},
\ref{thm:dim2nobdry} and \ref{thm:dim2wbdy},
in a way similar to \cite[Lemma 6.2]{MY}, we can prove that 
$$
C_i\approx F_i\times I.
$$
The limits of $B_i, C_i, B_i', F_i$ and $F_i'$ are denoted by $B, C, B', F$ and $F'$ respectively.
Note that 
\begin{equation} \label{eq:conn}
\pa B_i \text{ is connected}, 
\end{equation}
because so is $\pa B$.

\subsection{Case (\ref{case:S1})}
We consider the case (\ref{case:S1}). 
Namely, $E=\ba X$ and $X^D$ is a circle.
In this case, ${f_i^D} : D(M_i) \to X^D$ is globally a
$\Z_2$-equivariant fiber bundle with fiber 
$F_i\in \{ S^2,P^2,T^2,K^2\}$.
Note that $\phi_i$ fixes each component of $(f_i^D)^{-1}(\ba X)$
since $M_i$ has nonempty boundary.
Therefore ${f_i^D}$ induces a trivial bundle
$f_i : M_i \to X$.
Thus $M_i$ is homeomorphic to $F_i \times I$.
Note that all of $B_i, C_i$ and $B_i$ are homeomorphic to $F_i\times I$.

\subsection{Case (\ref{case:partial})}
Suppose $E=\{ 1\} \in \ba X$.
Then we can apply Case (\ref{case:S1}) to
$C\cup B'$ to verify that 
$$
(B_i',\pa B_i') \approx (C_i \cup B_i', \pa (C_i \cup B_i')) \approx (F \times [0,1], F \times \{0\})
$$ 
for $F \approx S^2, P^2, T^2$ or $K^2$, which yields
\[
(M_i, \ba M_i) = (B_i \cup C_i \cup B_i', \ba (B_i \cup C_i \cup B_i')) \approx (B_i, \pa B_i)\approx (B_i,F_i).
\]
Since $\ba B_i = \emptyset$, $(B_i,\pa B_i)$ is one of the spaces listed in \cite[Theorem 1.8]{MY}.
Thus we conclude that the topological type of $(M_i,\ba M_i) \approx (B_i,F_i)$ is determined as follows:

\renewcommand{\arraystretch}{1.2}
\begin{table}[htb] 
	\begin{tabular}{|c|c|} \hline 
		$\ba M_i$ & $M_i$ \\ \hline 
		$S^2$ & $D^3$, $P^2\tilde\times I$, $B(S^2;2)$ \\ \hline 
		$P^2$ & $K_1(P^2)$ \\ \hline 
		$T^2$ & $S^1 \times D^2$, $S^1 \times \Mo$, $K^2 \tilde \times I$, $B(S^2;4)$ \\ \hline 
		$K^2$ & $S^1 \tilde \times D^2$, $K^2 \hat \times I$, $B_\mathrm{pt}$, $B(P^2;2)$ \\ \hline 
	\end{tabular}
	\pmed
	\caption{Topology of $(M_i,\ba M_i)$}
	\label{table:FB}
\end{table}
\renewcommand{\arraystretch}{1}

Note that the model spaces listed above are also given by 
compact sublevel sets of Busemann functions 
as in Section \ref{ssec:classif=nonnegative}
in complete noncompact three-dimensional Alexandrov spaces without boundary. 
This completes the proof of Theorem \ref{thm:interval}(2).

\subsection{Case (\ref{case:interval}) -- Inradius collapse}
Let us consider the case (\ref{case:interval}), that is, $X^D=X=[0,1]$. 
We are going to determine the topology of $B_i$. 
The topology of $B_i'$ is also determined by the same way.
In this case, obviously we have 
\begin{equation} \label{eq:baB}
\ba B_i \ne \emptyset, \quad \ba C_i \ne \emptyset. 
\end{equation}
If $B_i$ can not be rescaled, then $\pa B_i \approx D^2$ and $B_i \approx D^3$.
We assume that $B_i$ can be rescaled and consider a rescaling limit $Y$ of $B_i$. 
Then $Y$ is a noncompact Alexandrov space of nonnegative curvature and of dimension $\ge 2$.

First suppose $\dim Y = 3$.
Then $Y$ has boundary and it is classified as in Soul Theorem \ref{soul theorem}.
Let $S$ denote a soul of $Y$ and let $B^Y$ be a large ball around $S$ in $Y$.
By Stability Theorem \ref{thm:stability respectful}, we have 
$B_i \approx B^Y$.
By \eqref{eq:conn}, the possible topological 
types of $Y$ and $(B_i, \pa B_i)$ are described as follows.

If $\ba Y$ is disconnected, then $Y$ is 
isometric $Y_0 \times I$, where $Y_0$ is a
noncompact Alexandrov surface of nonnegative
curvature without boundary, and hence 
homeomorphic to $\mathbb R^2$ or isometric to one of $S^1 \times \mathbb R$ and $S^1 \tilde \times \mathbb R$. 
By \eqref{eq:conn}, 
the case $Y_0 \equiv S^1 \times \mathbb R$ never occurs. 
Hence, we obtain 
\begin{equation} \label{eq:list-1}
(B_i,\pa B_i) \approx (I \times D^2, I \times \pa D^2) \text{ or } (I \times \Mo, I \times \pa \Mo).
\end{equation}

From now on, we assume that $\ba Y$ is connected.
If $S$ is a surface, then $Y \equiv S \times \mathbb R_+$, where $S$ is one of $S^2$, $P^2$, $T^2$ and $K^2$.
Thus we have 
\begin{equation} \label{eq:list-2}
(B_i,\pa B_i) \approx (S \times [0,1], S \times \{1\}),
\end{equation}

If $S$ is a circle, then $Y$ is an $N$-bundle over a circle, where $N$ is a nonnegatively curved noncompact contractible Alexandrov surface with boundary. 
Hence $N \approx \mathbb R_+^2$ or $N \equiv I \times \mathbb R$. 
Let $I = [-1,1] \subset \mathbb R$. For $k-1,2$,
define $S^1 \tilde \times_k I^2$ as
\[
 S^1 \tilde \times_k I^2:=S^1 \times I^2 / (z,x,y) \sim (-z,\tau_k(x,y)),
\] 
where $\tau_k$ denote the involutions on $\mathbb R^2$ defined by 
\[
\tau_1(x,y) = (-x,-y) ,\quad \tau_2(x,y) = (x,-y). 
\]
By \eqref{eq:conn}, 
$(B_i,\pa B_i)$ must be homeomorphic to one of the following:
\begin{equation} \label{eq:list-3}
\begin{aligned}
&(S^1 \times D_+^2, S^1 \times \pa D_+^2), (S^1 \tilde \times D_+^2, S^1 \tilde \times \pa D_+^2), \\
&(S^1 \tilde \times_1 I^2, S^1 \tilde \times_1 (I \times \pa I)), (S^1 \tilde \times_2 I^2, S^1 \tilde \times_2 (I \times \pa I)).
\end{aligned}
\end{equation}
Note that 
\begin{gather*}
S^1 \tilde \times_1 (I \times \partial I)\approx S^1\times I,
\quad S^1 \tilde \times_2 (I \times \partial I)\approx S^1\times I, \\
(S^1 \times D_+^2, S^1 \times \pa D_+^2)\approx
(S^1 \tilde \times_1 I^2, S^1 \tilde \times_1 (I \times \pa I)).
\end{gather*}

Suppose that $S$ is a point.
By \eqref{eq:conn}, $Y$ has exactly one end.
By Corollary \ref{cor:soul thm ball} together with 
Table \ref{table:BpaB2},
 $(B_i,\pa B_i)$ is homeomorphic to one of 
\begin{equation} \label{eq:list-4}
(D_+^3, \pa D_+^3), (K_1(P^2), D^2), (K_1(P^2), \Mo),
 (B_{\rm pt}, S^1\times I).
\end{equation}

Next suppose $\dim Y = 2$.
By \eqref{eq:conn}, $S^1 \times \mathbb R$ and $I \times \mathbb R$ are excluded from the possibilities of $Y$.
Therefore we have the following list for possible $Y$:
\begin{enumerate}
\item[(\ref{subsub:931})] $Y \equiv S^1 \times \mathbb R_+$; 
\item[(\ref{subsub:932})] $Y \equiv S^1 \tilde \times \mathbb R$; 
\item[(\ref{subsub:933})] $Y \equiv I \times \mathbb R_+$;
\item[(\ref{subsub:934})] $Y \approx \mathbb R^2_+$;
\item[(\ref{subsub:935})] $Y \approx \mathbb R^2$.
\end{enumerate}
As before, $B^Y$ denotes a large metric ball around a soul of $Y$. 

\subsubsection{} \label{subsub:931}
Suppose that $Y \equiv S^1 \times \mathbb R_+$. 
Then we obtain 
\[
(B^Y, \ba B^Y, \partial B^Y) \equiv (S^1 \times [0,\ell], S^1 \times \{0\}, S^1 \times \{\ell\}),
\] 
for some $\ell>0$.
From Theorem \ref{thm:dim2wbdy}, $B_i$ is either an $I$-bundle over $B^Y$ or an $S^1$-bundle over $B^Y$.
If $B_i$ is an $S^1$-bundle over $B^Y$, then 
$\pa B_i$ is homeomorphic to $T^2$ or $K^2$.
It turns out that $C_i$ has 
no Alexandrov-boundary points, which contradicts 
 \eqref{eq:baB}. 
Therefore, 
$B_i$ is an $I$-bundle over $S^1\times I$, and we have
\begin{equation} \label{eq:list-5}
(B_i,\pa B_i)\approx (S^1\times D^2, S^1\times I)
\,\,\text{or}\,\, (\Mo\times I, \Mo)
\end{equation}

\psmall
\subsubsection{} \label{subsub:932}
Let us assume that $Y \equiv S^1 \tilde \times \mathbb R$. 
By Theorem \ref{thm:dim2wbdy}, we obtain 
\begin{equation} \label{eq:list-6}
(B_i, \pa B_i) \approx (\Mo \times I, \pa \Mo \times I)\, \text{ or } \,(\Mo \tilde \times I, \pa \Mo \tilde \times I). 
\end{equation}

\subsubsection{} \label{subsub:933}
Let us consider the case that $Y \equiv I \times \mathbb R_+$. 
Then, $B^Y$ is isometric to an rectangle.
Hence, by Theorem \ref{thm:dim2wbdy}, we obtain a surjective continuous map 
\[
f_i : B_i \to B^Y. 
\]
If the general fiber of $f_i$ is $I$, then $f_i$ is a trivial $I$-bundle. 
Namely, we have in this case
\begin{equation} \label{eq:list-7}
(B_i, \pa B_i) \approx (B^Y \times I, \pa B^Y \times I). 
\end{equation}

Assume that the general fiber of $f_i$ is $S^1$. 
Let us denote by $a, b$ and $c$ edges of the rectangle $B^Z$ that is distinct from $\pa B^Z$ such that $c =I\times \{ 0\}$. 
Let $\alpha, \beta$ and $\gamma$ denote the fibers of 
$f_i$ over interior points of edges $a$, $b$ and $c$ respectively. 
Since $\pa B_i$ is a surface with boundary, at least one of $\alpha$ and $\beta$ must be a circle.
Therefore we may assume that $\alpha$ is a circle. 
The topological types of $(B_i, \pa B_i)$ are 
classified in terms of $(\beta, \gamma)$ by the following table
(see Lemmas \ref{lem:topBiCaseA} and \ref{lem:7.10}) :

\renewcommand{\arraystretch}{1.2}
\begin{table}[htb] 
	\begin{tabular}{|c|c||c|c|} \hline
		$\gamma$ & $\beta$ & $B_i$ & $\partial B_i$ \\ \hline \hline 
		& $S^1$ & $S^1 \times I \times [0,1]$ & $S^1 \times I \times \{1\}$ \\ \cline{2-4}
		$S^1$ & $I$ & $D^2 \times [0,1]$ & $D^2 \times \{1\}$\\ \cline{2-4}
		& $S^1/2$ & $\Mo \times [0,1]$ & $\Mo \times \{1\}$ \\ \hline
		& $S^1$ & $I \times D^2$ & $I \times \partial D^2$ \\ \cline{2-4}
		$I$ & $ I $ & $D^2 \times [0,1]$ & $D^2 \times \{1\}$ \\ \cline{2-4}
		& $S^1/2$ & $K_1(P^2)$ & $\Mo$ \\ \hline
		& $S^1$ & $I \times \Mo$ & $I \times \partial \Mo$ \\ \cline{2-4}
		$S^1/2$ & $I$ & $K_1(P^2)$ & $D^2$ \\ \cline{2-4}
		& $S^1/2$ & $\Mo \times [0,1]$ & $\Mo \times \{1\}$ \\ \hline
	\end{tabular}
	\pmed
	\caption{Topology of $(B_i,\pa B_i)$}
	\label{table:gamma-beta}
\end{table}
\renewcommand{\arraystretch}{1}

Let $k$ and $\ell$ be the numbers of essential singular points
of $Y$ contained in $\ba Y$ and $\mathrm{int}^{\bf A} Y$ respectively.
 
\subsubsection{} \label{subsub:934}
Let us assume that $Y\approx \mathbb R^2_+$.
Then we have only the following possibilities
(see \cite[Corollary 14.4]{SY} for instance):
\[
 (k,\ell)=(2,0), (1,0), (0,1), (0,0).
\]
The case $(k,\ell)=(2,0)$ is nothing but the case (\ref{subsub:933}).
The case $(k,\ell)=(1,0)$ is essentially done in the case (\ref{subsub:933}).
Therefore, we consider the case $(k,\ell)=(0,1)$. 
Then $Y$ is isometric to 
$D(\{ x,y\ge 0\})\cap \{y\le h\}$ for some $h>0$.
We may assume that an exceptional fiber of 
$f_i:B_i \to B^Y$ appears.
If $f_i$ is a generalized $I$-bundle, then 
\begin{equation} \label{eq:list-72}
(B_i, \pa B_i) \approx (K_1(P^2), D^2).
\end{equation}

Assume that $f_i$ is a generalized Seifert fibration.
Then the fiber of $f_i$ over any point of $\ba B^Y$ must be 
a regular circle because if it was $S^1/2$ or $I$, then
$\pa B_i$ would be homeomorphic to $S^2$ or $K^2$,
contradicting \eqref{eq:baB}. Thus we have
\begin{equation} \label{eq:list-8}
(B_i, \pa B_i) \approx 
(B_\mathrm{pt}, I \times S^1).
\end{equation}

\pmed
\subsubsection{} \label{subsub:935}
Finally, we suppose that $Y \approx \mathbb R^2$. 
Then we have $\ell\le 2$.
When $\ell =2$, 
$Y$ is isometric to $D(I\times \R_+)$.
By the same reason as in (\ref{subsub:934}), $f_i$ is a generalized $I$-bundle, which implies 

\begin{equation} \label{eq:list-9}
(B_i,\pa B_i)\approx 
(B_\mathrm{pt}^+ \cup_{I \times I} B_\mathrm{pt}^+, \Mo) \approx (B_{\rm pt}, \Mo).
\end{equation}
When $\ell \le 1$, we have 
\begin{equation} \label{eq:list-10}
(B_i, \partial B_i) \approx
 \begin{cases}
  (B_{\mathrm{pt}}^+, \Mo) &\text{ if $\ell=1$}, \\
  (D^2 \times I, \partial D^2 \times I) &\text{ if $\ell=0$}.
 \end{cases}
\end{equation}

\pmed
Let $f_i:F_i\to F_i'$ be a homeomorphism defined by
$C_i\approx F_i\times I$, and consider $M_i$ as the glued space
$$
M_i\approx B_i\cup_{f_i} B_i'.
$$
Summarizing \eqref{eq:list-1} $\sim$ \eqref{eq:list-10}
together with Table \ref{table:gamma-beta},
 we have the 
following gluing data based on the topological types of $F_i\approx F_i'$.

\renewcommand{\arraystretch}{1.2}
\begin{table}[htb] 
\begin{tabular}{|c|c|c|} \hline
$B_i$ & $F_i$ & $B_i'$ \\ \hline \hline
$D^3, K_1(P^2)$ & $D^2$ & $D^3, K_1(P^2)$ \\ \hline
$D^3$, $S^1\times D^2$, $S^1\tilde \times D^2$, $B_{\rm pt}$
& $S^1\times I$
&$D^3$, $S^1\times D^2$, $S^1\tilde \times D^2$, $B_{\rm pt}$
\\ \hline
%\psmall
$S^1\tilde \times D^2$, $K_1(P^2)$, $B_{\rm pt}$ 
& $\Mo$ 
& $S^1\tilde \times D^2$, $K_1(P^2)$, $B_{\rm pt}$ \\ \hline
\end{tabular}
\pmed
\caption{Gluing data}
\label{table:gluing-data}
\end{table}
\renewcommand{\arraystretch}{1}

For the proof of Theorem \ref{thm:interval}(3),
it suffices to show the following.

\begin{claim} \label{claim:gluing=nonnegative}
$M_i\approx B_i\cup_{f_i} B_i'$ is homeomorphic to an Alexandrov
space with nonnegative curvature.
\end{claim}
\begin{proof}
Choose $F$ from $\{ D^2, S^1\times I, \Mo\}$ with 
$F\approx F_i$, and consider a metric of $F$ with
nonnegative curvature. Let $F'$ be another copy of $F$.
Take the model spaces $P$ and $P'$ from the above list
such that $P\approx B_i$ and $P'\approx B_i'$.
Now it is easy to construct Alexandrov metrics of $P$ and $P'$
with nonnegative curvature such that 
 \begin{itemize}
\item $F\subset \ba P$ and $F'\subset\ba P'$ are
extremal subsets of $P$ and $P'$ respectively\,;
\item the identical map ${\rm id}:F\to F'$ is an isometry with respect to the intrinsic metrics induced from $P$ and $P'$ respectively.
\end{itemize}
Let us define $Q$ by the gluing 
\[
  Q:=P\cup_{\rm id} P'
\]
along $F=F'$.
According to \cite{Mit:gluing}, $Q$ is an Alexandrov 
space with nonnegative curvature.
It suffices to show that $M_i$ is homeomorphic to $Q$.
Choose a homeomorphism $\varphi_i:F=F'\to F_i$.
Using the cone structures, we extend $\varphi_i$ to a homeomorphism 
$$
 \psi_i:\ba P\bigcup_{F=F'}\ba P' \to \ba B_i\bigcup_{f_i}\ba B_i'.
$$
 Finally extend $\psi_i$ to a homeomorphism 
$\Phi_i:Q\to M_i$ using a kind of cone-extension in each case.
\end{proof}

\section{The case that $X$ is a single point set} \label{sec:point} 
In this section, we prove the following theorem as an application of the results so far.

\begin{thm} \label{thm:dim0}
Let $M_i$ be a sequence  of Alexandrov spaces with boundary in $\ca M(3,D)$ 
converging to a point. 
Then $M_i$ is homeomorphic to an Alexandrov space of nonnegative curvature. 
\end{thm}
\begin{proof}
Let $d_i := \diam\, M_i$.
Passing to a subsequence, we may assume that $d_i^{-1} M_i$
converges a compact nonnegatively curved Alexandrov space $Y$ with $1\le\dim Y\le 3$. For $\dim Y \in \{ 1,3\}$, 
Stability Theorem \ref{thm:stability}, 
Theorems \ref{thm:main-circle} and \ref{thm:interval} yield the conclusion.

From now on, we assume $\dim Y =2$. 
If $\ba Y = \emptyset$, then by Theorem \ref{thm:dim2nobdry}, $M_i$ is a generalized $I$-bundle over $Y$. 
Such a space is denoted by $B(Y; k)$, where $k$ is the number of singular $I$-fibers (\cite{SY}). 
Note that the topology of $B(Y; k)$ is completely determined 
in \cite{Y 4-dim} and \cite{MY}, which is one of 
\[
 B(S^2, 2), B(S^2;4), B(P^2,2).
\]
In particular, $B(Y;k)$ is a compact nonnegatively curved Alexandrov space with boundary.

Finally, we assume that $\dim Y = 2$ and $\ba Y \neq \emptyset$. 
Then $Y$ is either homeomorphic to $D^2$ or isometric to 
one of $S^1\times I$ and $\Mo$, which are flat.
Due to Theorem \ref{thm:dim2wbdy}, we obtain a map
 $f_i : M_i \to Y$ 
satisfying one of the following:
\begin{enumerate}
\item it is a generalized $I$ \-bundle\,$;$
\item it is a generalized Seifert fibration.
\end{enumerate}
We first consider the case $(1)$. 
If $Y$ is a M\"obius band or an annulus, then it is flat, and hence $f_i$ has is $I$-bundle.
We assume that $Y$ is a two-disk. 
Let $k$ be the number of singular $I$-fibers of $f_i$.
Since the number of essential singular points in the interior of a nonnegatively curved disk is at most $2$,
we have $k\le 2$.
By Theorem \ref{thm:dim2nobdry},
$M_i$ is homeomorphic to $D^3$, $K_1(P^2)$ or $B_\mathrm{pt}$ if $k = 0$, $1$ or $2$ respectively. 
Thus, the conclusion of Theorem \ref{thm:dim0}
certainly holds so far.

Next let us consider the case $(2)$. 
Let $E\subset \ba Y$ be the limit of $\ba M_i$.
If $Y = \Mo$, then $E=\ba Y$ and $f_i$ is an $S^1$-bundle over $\Mo$. 
If $Y = S^1 \times I$, then $E$ coincides with $\ba Y$ or a component $\ba_0 Y$ of $\ba Y$. 
If $E=\ba Y$, then $M_i$ is an
$S^1$-bundle over $S^1 \times I$, which is homeomorphic to $T^2 \times I$ or $K^2 \times I$. 
If $E=\ba_0 Y$, then $M_i$ is homeomorphic to an $F$-bundle over $S^1$, where $F$ is $D^2$ or $\Mo$. 
All those spaces have nonnegatively curved metrics.

Finally we consider the case $Y \approx D^2$. 
Let $k$ and $l$ denote the numbers of 
singular fibers of $f_i$ over points in $\mathrm{int}^{\bf A} Y$ and $\ba Y$ respectively.
Since these singular fibers appear only over essential singular points of $Y$, from Gauss-Bonnet formula for Alexandrov surfaces (see \cite{SY} for instance), we obtain 
\[
2 k + l \le 4.
\]

If $(k,\ell)=(2,0)$, then $Y$ is an envelop:
$Y=D(\R_+\times I)\cap \{ x\le h\}$ for some $h>0$,
and $E=\ba Y$. We divide $Y$ into two disks
$D_1$ and $D_2$ via a segment each of whose interiors contains 
an essential singular point of $Y$. 
Since $f_i^{-1}(D_\alpha)\approx B_{\rm pt}$ for each
$\alpha=1,2$, we obtain the gluing 
\[
 M_i\approx B_{\rm pt}\cup_{S^1\times I} B_{\rm pt}
  =B(S^2,4).
\]

Suppose $k=1$.
If $E=\ba Y$, then $\ell=0$ and 
$M_i\approx B_{\rm pt}$.
If $E\subsetneqq \ba Y$, then $\ell=2$
since the singular fibers appear at $\pa E$. 
We cut off
a neighborhood of the closure, say $G$, of $\ba Y\setminus E$ in $Y$
by a segment transversally meeting $E$.
 Then we have the gluing
\[
 M_i\approx B_{\rm pt}\cup_{S^1\times I} P,
\]
where $P$ is either $D^2\times I$ or $\Mo\times I$
depending on the fiber of $f_i$ on $G$.

Next suppose $k=0$. If $\ell=4$, then $Y$ is a rectangle.
If $E$ is connected in addition, we divide $Y$ into two rectangles by a segment from the midpoint of $E$
to have the gluing 
\[
M_i\approx K_1(P^2)\cup_{F} K_1(P^2),
\]
where $F$ is either $D^2$ or $\Mo$
depending on the fiber data of $f_i$ on the three edges of $G$.
If $E$ is disconnected, then 
$M_i\approx F\times I$, where 
$F$ is one of $S^2$, $P^2$ or $K^2$
depending on the fiber data of $f_i$ on the three edges of $G$.

If $(k,\ell)=(0,3)$, then $M_i\approx K_1(P^2)$.
If $(k,\ell)=(0,2)$, then $M_i\approx F\times I$, 
where $F$ is $D^2$ or $\Mo$.
If $(k,\ell)=(0,0)$, then $M_i\approx D^3$.

Note that all those spaces appear in 
the previous arguments, and certainly admit 
metrics of nonnegative curvature.
This completes the proof of Theorem \ref{thm:dim0}. 
\end{proof}

\appendix

\section{Equivariant flow argument} \label{sec:flow} 
This section is devoted to generalize our result \cite[Theorem 3.2]{MY} to an equivariant version. 
This is technically crucial in the present paper.

Let $M$ be an Alexandrov space. 
A map $\Phi$ is called a {\it local flow} on $M$ if the domain of $\Phi$ is of the form 
\[
\mathrm{dom}(\Phi) = \bigcup_{x \in M} U_x \times (-\epsilon_x, \epsilon_x), 
\]
where $U_x$ is an open neighborhood of $x$ in $M$ and $\epsilon_x > 0$, such that 
$\Phi$ satisfies 
\[
\Phi(\Phi(x,t), s) = \Phi(x,t+s)
\]
for each $x \in M$ and $s, t \in \mathbb R$ whenever the both sides of the above formula are well-defined.
If $\epsilon_x$ can be taken to be infinity for every $x \in M$, the local flow is called a ({\it global}) {\it flow}.
Let $U$ be an open set of $M$ and $f : U \to \mathbb R$ a Lipschitz function.
We say that a local flow $\Phi$ on $M$ is {\it gradient-like} for $f$ if it is locally Lipschitz and there is a constant $C > 0$ such that 
\[
\liminf_{t \to 0} \frac{f(\Phi(x,t))- f(x)}{t} \ge C
\]
for every $x \in U$. 

Suppose that $M$ has nonempty boundary in addition. Recall that 
$D(K)=K\cup\phi(K)$ for a closed subset $K$ of $M$, where $\phi$ is the reflection on the double $D(M)$ along the boundary $\ba M$. 

The following is an equivariant version of \cite[Theorem 3.2]{MY}. Although the proof is similar to that of \cite[Theorem 3.2]{MY},
we give an outline of the proof for reader's convenience, focusing on the equivariant aspect.

A bijective map $\varphi:X\to Y$ between metric spaces are 
called {\it $\e$-almost isometric} if it is bi-Lipschitz
with bi-Lipschitz constant $\le 1+\e$. 

\begin{thm} \label{thm:smooth approximation}
For $n \in \mathbb N$, there is $\epsilon_n > 0$ satisfying the following. 
Let $M$ be an $n$-dimensional Alexandrov space of curvature $\ge -1$ with boundary. 
Let $K\subset M$ be compact such that 
\beq \label{eq:D(K)-strained}
\text{$D(K)$ is $\epsilon$-strained }
\eeq
for $\epsilon \le \epsilon_n$.
If $\ba M \cap K \neq \emptyset$, then there exist a smooth (incomplete) Riemannian manifold $N$ with an isometric involution $\tau$ and a $DC^1$-homeomorphism $\varphi : U(D(K)) \to N
$ defined on a $\phi$-invariant neighborhood
$U(D(K))$ of $D(K)$ in $D(M)$
such that
\begin{itemize}
\item $\varphi$ is $\theta(\epsilon)$-almost isometric$\,;$
\item $\varphi$ is $\mathbb Z_2$-equivariant in the sense that $\varphi \circ \phi = \tau \circ \varphi $\,;$
$\item if $S$ is a closed set in $M$ such that
\beq\label{eq:dDSregular}
\text{ $d_{D(S)}$ is $(1-\delta)$-regular on $D(K)$}, \eeq
then there is a flow $\Phi : D(M) \times \mathbb R \to D(M)$ which is gradient-like for $d_{D(S)}$ on $U(D(K))$ with 
\[ 
\left. \frac{d}{d t} \right|_{t = 0+} d_{D(S)} \circ \Phi(x,t) > 1 - 5 \sqrt \delta - \theta(\epsilon)
\]
for every $x \in U(D(K))\,;$
\item $\Phi$ is $\mathbb Z_2$-equivariant, i.e., 
$\Phi(\phi x,t) = \tau \Phi(x,t)$
for every $(x,t) \in D(M) \times \mathbb R$.
\end{itemize}
\end{thm}

Before starting the proof of Theorem \ref{thm:smooth approximation}, we prepare several lemmas. 
Let us fix an $n$-dimensional Alexandrov space $M$ of curvature $\ge -1$ with boundary.

The following lemma can be proved via the strainer map given in \cite{BGP}: 
\begin{lem}
Let $x \in D(M)$ be an $(n,\epsilon)$-strained point with strainer of length $\ell$, where $\epsilon$ is small with respect to $n$.
Then, there exists $t>0$ depending on $\ell$ such that for any $\xi \in \Sigma_x D(M)$, there exists an $(n,2\epsilon)$-strainer $\{(a^\alpha, b^\alpha)\}_{\alpha=1,\dots,n}$ at $x$ of length $t$ such that 
\[
\angle (\uparrow_x^{a^1}, \xi )< \theta_n(\epsilon). 
\]
\end{lem}

\begin{lem} \label{lem:t}
Let $M, K, S$ be as in Theorem \ref{thm:smooth approximation}, and let $\epsilon, \delta$ be as in 
\eqref{eq:D(K)-strained} and \eqref{eq:dDSregular}.
Then there exist positive numbers $t \ll \ell$ depending on $K$ and $S$ such that 
for each $x\in B^{D(M)}(D(K),10t)$, there exists
a point $q(x) \in D(M)$ such that 
\begin{itemize}
\item $|x, q(x)| = \ell$; 
\item for every $q \in xq(x)$, we have 
\[
d_{D(S)}(q)- d_{D(S)}(x) \ge (1-2\delta) |x, q|; 
\]
\item there exists an $(n,\epsilon)$-strainer $\{(a^\alpha, b^\alpha)\}_{\alpha=1,\dots,n}$ at $x$ in $D(M)$ of length $\ell$ with 
$b^n = q(x)\,;$
\item $q(x)$ is equivariant:
$\phi q(x) = q(\phi x)$.
\end{itemize}
\end{lem}

Before proving Lemma \ref{lem:t}, we recall the notion of the {\it gradient} of distance functions
(see \cite{PP QG}, \cite{Pet Semi}).
Let $X$ be an Alexandrov space and $A$ a closed subset of $X$. 
For $x \in X \setminus A$, the gradient of $d_A = |A,\,\cdot\,|$ at $x$ is defined as the 
 element $\nabla d_A (x)$ of $T_x X$: 
\[
\nabla d_A (x) := \left\{ 
\begin{aligned}
&o_x && \text{if $d_A$ is critical at $x$} \\
&d_A'(\xi_{\max}) \xi_{\max} &&\text{otherwise,}
\end{aligned}
\right.
\]
where $\xi_{\max} \in \Sigma_x$ is the unique element which maximizes the derivative $d_A'$ of $d_A$ on $\Sigma_x$. 

\begin{proof}[Proof of Lemma \ref{lem:t}]
Since the norm of the gradient $|\nabla d_{D(S)}|$ is lower semicontinuous, there exists $t > 0$ such that 
\begin{equation} \label{eq:ab-grad}
|\nabla d_{D(S)}| > 1- \delta \text{ on } B(D(K), 10t). 
\end{equation}
Moreover, we may assume that each point of $B(D(K),10t)$ is $\epsilon$-strained with 
a strainer of length $10\ell > 0$ for a small enough $\ell$.
Set $K':=B(K,10t)$.
For each $x \in D(K')= B(D(K),10t)$, by \eqref{eq:ab-grad}, 
setting 
\[
v(x) := \frac{\nabla d_{D(S)}}{|\nabla d_{D(S)}|} \in \Sigma_x D(M), 
\]
we have 
\[
(d_{D(S)})'(v(x)) > 1- \delta.
\]
Because $d_{D(S)}$ is $\phi$-equivariant, so is $v(x)$. 
Let us take a point $q(x) \in D(M)$ such that 
\[
|x, q(x)| = \ell 
\]
and
\begin{equation*} 
%	\label{eq:almost gradient}
\angle (\uparrow_x^{q(x)}, v(x)) < \delta^2. 
\end{equation*}
Then, we have 
\begin{align*}
(d_{D(S)})'(\uparrow_x^{q(x)}) &= - \cos \angle (\Uparrow_x^{D(S)}, v(x)) \\ 
&\ge - \cos \left(\angle (\Uparrow_x^{D(S)}, \uparrow_x^{q(x)}) - \angle (\uparrow_x^{q(x)}, v(x)) \right) \\
&> (1-\delta)(1- \delta^4) - \delta^2 \\ 
&> 1 - 2 \delta. 
\end{align*}
Note that there exists $s(x) \in S$ such that 
\[
\angle (v(x), \uparrow_x^{s(x)}) - \widetilde \angle q(x) x s(x) < \theta(\epsilon)
\]
Hence, we have 
\[
- \cos \widetilde \angle q(x)x s(x) > 1 - 2\delta - \theta(\epsilon).
\]
Furthermore, since $\ell$ is small, $|x, s(x)| > 10\ell$.

Retaking $\ell$ to be small, 
we may assume that 
\[
d_{D(S)}(q)-d_{D(S)}(x) \ge (1-2\delta) |x,q|
\]
for any $q$ in $x q(x)$.
Let $\{(a^\alpha, b^\alpha)\}_{\alpha = 1,2,\dots,n}$ be an $\epsilon$-strainer at $x$ of length $\ell$ such that $b^n=q(x)$.
It is clear that $q(x)$ can be taken to satisfy $\phi(q(x)) = q(\phi(x))$. 
This completes the proof. 
\end{proof}

\begin{proof}[Outline of a proof of Theorem \ref{thm:smooth approximation}]
We need several steps for proving Theorem \ref{thm:smooth approximation}. 
 
\pmed
\noindent
{\bf Step 0}\,(Preliminaries). 

\psmall
Let $t$ be as in Lemma \ref{lem:t}.
For $0<s\ll t$, take a maximal $(0.2)s$-discrete set $\{x_j\}_{j=1}^N$ of $K$. 
If $|x_j, \ba M \cap K| \le (0.05)s$, then we take $\tilde x_j \in \ba M$ which is a midpoint of a minimal geodesic $x_j \phi(x_j)$.
If $|x_j, \ba M \cap K| > (0.05)s$, we set $\tilde x_j = x_j$. 
Then, $\{\tilde x_j\}_{j=1}^N$ is a $(0.4)s$-net of $K$. 
Note that $\tilde x_j$ may not be a point of $K$. 
Recall that 
$d_{D(S)}$ is $(1-\delta)$-regular on $B_{10 t} (D(K))$. 
We permute the indices of $\tilde x_j$ in such a way that 
\begin{itemize}
\item $\tilde x_j \in \ba M$ if $j \le N'$, for some $N' \le N$; 
\item $\tilde x_j \not \in \ba M$ if $j > N'$. 
\end{itemize}
Note that the 
$\phi$-invariant set 
\[
\{\tilde x_j\}_{j \le N'} \cup \{\tilde x_j, \phi \tilde x_j \}_{j >N'}
\]
is a $(0.4)s$-net of $D(K)$ and is $(0.1)s$-discrete. 

For each $j \in \{1,\dots, N\}$, we choose an $\epsilon$-strainer 
$\{(a_j^\alpha,b_j^\alpha)\}_{\alpha=1}^n\subset D(M)$ at $\tilde x_j$ of length $\ell$ such that 
\begin{itemize}
\item $b_j^1 = q(\tilde x_j)$ for every $j \le N$, where $q(\tilde x_j)$ is as in 
Lemma \ref{lem:t} $\,;$
\item for $j \le N'$, we have 
\begin{align*}
&a_j^\alpha, b_j^\alpha \in \ba M \hspace{1em} (1 \le \alpha \le n-1), \\
&\phi(a_j^n) = b_j^n.
\end{align*}
\end{itemize}
Note that for every $j \le N$, $\{(\phi(a_j^\alpha), \phi(b_j^\alpha))\}_{\alpha=1}^n$ is an $\epsilon$-strainer at $\phi(\tilde x_j)$. 

Since $t \ll \ell$, $\{(a_j^\alpha,b_j^\alpha)\}_\alpha$ is $\theta(\epsilon)$-strainer at $x$
 for any $x \in B^{D(M)}(\tilde x_j, 10 t)$.
Due to \cite[Lemma 1.9]{Y conv} and $s \ll t$, for $c\in \{ a_j^\alpha, b_j^\alpha\}$, we have 
\[
|\tilde \angle c xy - \angle cxy| < \theta(\epsilon)
\]
for all $x \in B^{D(M)}(\tilde x_j, s)$ and $y \in B^{D(M)}(x, s)$. 

\pmed
We are going to sketch the construction of the manifold $N$ desired in the theorem by 
dividing into
several steps.
For the detail of the discussion, we refer to \cite{O}
(cf.\cite{KMS}).

Let $E_j$ denote the standard Euclidean $n$-space for $1\le j \le N$. 
We define $f_j = (f_j^\alpha)_{\alpha=1}^n : B^{D(M)}(\tilde x_j, 10t) \to E_j$ by 
\[
f_j^\alpha(y):= \frac{1}{\mathcal H^n (B^{D(M)}(a_j^\alpha,\epsilon'))} \int_{B^{D(M)}(a_j^\alpha,\epsilon')} |y,z|-|\tilde x_j,z|\, d \mathcal H^n(z)
\]
where $\epsilon' \ll \epsilon$. 
This is a $\theta(\epsilon)$-almost isometric $DC^1$-homeomorphism, 
and the family $\{f_j, f_j \circ \phi\}_{j \le N}$ gives a $DC^1$-coordinate system of the $9t$-neighborhood of $D(K)$.
Here, for the notion of $DC^1$-regularity, we refer to \cite{Per DC} and \cite{KMS}. 
From now on, we denote $f_j'$ and $E_j'$ by
\[
f_j' := f_j \circ \phi : B^{D(M)}(\phi \tilde x_j, 10t) \to E_j' := E_j. 
\]

\vspace{1em}
\noindent{\bf Step 1}\,(Construction of $\{ F_j^k \}_{j,k}$). 
This part corresponds to \cite[Lemma 5]{O} (cf. \cite[Lemma 3.3]{MY}). 
For the detail of the following construction, see the proof therein.
First we construct an isometry 
\[
F_j^k : E_k \to E_j
\]
satisfying 
\begin{align}
|F_j^k \circ f_k (y) - f_j(y)| < \theta(\epsilon) s; \label{eq:F1} \\
|dF_j^k \circ df_k(\xi)- df_j(\xi)| < \theta(\epsilon) \label{eq:F2}
\end{align}
for all $j \ne k$, $y \in B(\tilde x_j,s) \cap B(\tilde x_k,s)$ and $\xi \in \Sigma_y (D(M))$.

For $k \in \{1,\dots, N\}$ and $\alpha \in \{1,\dots,n\}$, we take $y_k^\alpha \in \tilde x_k b_k^\alpha$ and $y_k^{-\alpha} \in \tilde x_k a_k^\alpha$ with 
\[
|\tilde x_k, y_k^{\pm \alpha}| =s.
\]
Then, we have 
\[
\left< f_k(y_k^\alpha), f_k(y_k^\beta) \right> = s^2 \delta_{\alpha \beta} + \theta(\epsilon, s/\ell).
\]
Let $\{e_k^\alpha \}_{\alpha}$ be the orthonormal basis of $E_k$ obtained by Schmidt's orthogonalization of $\{f_k(y_k^\alpha)/ s\}_\alpha$. 
Let 
\[
v_\alpha := \frac{1}{2s} (f_j(y_k^\alpha)- f_j(y_k^{-\alpha})). 
\]
Then, we have 
\[
\left< v_\alpha, v_\beta \right> = \delta_{\alpha \beta} + \theta(\epsilon). 
\]
By Schmidt's orthogonalization of $\{v_\alpha\}_\alpha$, we obtain an orthonormal basis $\{ \tilde e^\alpha_j \}_\alpha$ of $E_j$ such that 
\[
|\tilde e^\alpha_j - v_\alpha| < \theta(\epsilon). 
\]
Then the desired isometry is defined by 
\[
F_j^k (v) := f_j(\tilde x_k) + \sum_\alpha \left< v, e_k^\alpha \right> \tilde e^\alpha_j. 
\]

\vspace{1em}
\noindent{\bf Step 2}\,(Construction of transition maps $\{\tilde F_j^k, \tilde G_j^k, \tilde H_j^k, \tilde A_j^k \}_{j,k}$). 
This step corresponds to \cite[Lemma 6]{O} (cf. \cite[Lemma 3.4]{MY}). 
Let us set 
\[
V_j := U(0, (0.4)s) \subset E_j
\]
and $V_j' := V_j \subset E_j'$. 

For $j,k \le N$ with $|\tilde x_j, \tilde x_k| < (0.9)s$, 
we construct
a $\theta(\epsilon)$-almost isometric $C^\infty$-map 
\[
\tilde F_j^k : E_k \to E_j
\]
satisfying 
\begin{align}
&\tilde F_j^j = \mathrm{id}; \label{eq:F3}\\
&\tilde F_j^i(v) = \tilde F_j^k \circ \tilde F_k^i(v) \label{eq:F4}
\end{align}
for all $v \in V_i \cap \tilde F_i^k(V_k) \cap \tilde F_i^j(V_j)$. 
Moreover, $\{\tilde F_j^k\}_{j,k}$ satisfies the same property as \eqref{eq:F1} and \eqref{eq:F2}.

Let $\lambda : [0, \infty) \to [0, \infty)$ be a $C^\infty$-function such that 
\beq \label{eq:labambda}
\begin{cases}
\begin{aligned}
&\lambda = 1 \hspace{1em}\text{ on } [0,1/2] \\
&\lambda=0 \hspace{1em}\text{ on } [1, \infty) \\
&-4 \le \lambda' \le 0. 
\end{aligned}
\end{cases}
\eeq
Define $\psi_j : E_j \to [0,1]$ by 
\[
\psi_j(v) := \lambda(|v| / (0.8)s). 
\]
First we set 
\[
\tilde F_j^1 = F_j^1, \hspace{1em} \tilde F_1^j = (\tilde F_j^1)^{-1} \hspace{1em} (j \ge 1),
\]
and define $\tilde F_j^2$ with $j \ge 2$ by 
\[
\tilde F_j^2 (v) := (\psi_1 \circ \tilde F_1^2(v)) \cdot \tilde F_j^1 \circ \tilde F_1^2(v) + (1-\psi_1 \circ \tilde F_1^2(v)) \cdot F_j^2(v)
\]
for $v \in E_2$. 
 In view of \eqref{eq:labambda}, \eqref{eq:F4} holds
for $\{ i,k\}=\{ 1,2\}$ and any $j$. Moreover, 
$\{\tilde F_j^2\}$ satisfies the same property as \eqref{eq:F1} and \eqref{eq:F2}.

Similarly, as described in \cite[pp.1263-1264]{O}, we obtain $\tilde F_j^k$ for arbitrary $k,j$
satisfying the desired properties. 

For $1\le k \le N'$, we define the reflection $R_k : E_k \to E_k'$ by 
\[
R_k\biggl( \sum_{\alpha=1}^n u^\alpha e_k^\alpha\biggr) := \sum_{\alpha<n} u^\alpha e_k^\alpha - u^n e_k^n
\]
for $(u^\alpha) \in \mathbb R^n$, where $\{e_k^\alpha\}$ is as in Step 1.
That is, $R_k$ is the reflection along the hyperplane with normal vector $e_k^n$.
For $k > N'$, we define $R_k : E_k \to E_k'$ by $R_k = \mathrm{id}$.

Finally, we set 
\begin{align*}
\tilde G_j^k &= R_j \circ \tilde F_j^k : E_k \to E_j'; \\
\tilde H_j^k &= \tilde F_j^k \circ R_k : E_k' \to E_j; \\
\tilde A_j^k &= R_j \circ \tilde F_j^k \circ R_k: E_k' \to E_j'.
\end{align*}
Then, the family $\{\tilde F_j^k, \tilde G_j^k, \tilde H_j^k, \tilde A_j^k \}_{j,k}$ satisfies the cocycle condition similar to \eqref{eq:F3} and \eqref{eq:F4}. 
Indeed, for instance, for $v \in V_k \cap \tilde F_k^i (V_i) \cap (\tilde G_j^k)^{-1}(V_j')$, we have 
\[
\tilde H_i^j \circ \tilde G_j^k(v) 
= \tilde F_i^j \circ R_j \circ R_j \circ \tilde F_j^k (v) 
= \tilde F_i^j \circ \tilde F_j^k(v) 
= \tilde F_i^k(v).
\]

\vspace{1em}
\noindent{\bf Step 3}\,(Construction of a manifold $N$ and a homeomorphism $f : D(U(K)) \to N$). 
The cocycle condition of $\{\tilde F_j^k, \tilde G_j^k, \tilde H_j^k, \tilde A_j^k \}_{j,k}$ induces a natural equivalent relation on the disjoint set 
\[
\bigsqcup_{j \le N} V_j \sqcup V_j'.
\]
That is, for $v \in V_j$, $w \in V_k$, $v' \in V_j'$ and $w' \in V_k'$, 
\begin{align*}
v \sim w &\iff \tilde F_k^j (v) = w, \\
v \sim w' &\iff \tilde G_k^j (v) = w', \\
v' \sim w &\iff \tilde H_k^j (v') = w, \\
v' \sim w' &\iff \tilde A_k^j (v') = w'.
\end{align*}
Then, the quotient space 
\[
N := \left(\, \bigsqcup_{j \le N} V_j \sqcup V_j' \right)/\!\sim
\]
becomes a smooth manifold.
An involution $\tau$ on $N$ is naturally induced from the maps 
\[
V_j \ni v \mapsto R_j(v) \in V_j'.
\]
The natural projection is denoted by $\pi_N : \bigsqcup_{j} V_j \sqcup V_j' \to N$. 
From the construction, $\pi_N|_{V_j}$ and $\pi_N|_{V_j'}$ are homeomorphisms. 
Let $N_j = \pi_N(V_j)$ and $N_j' = \pi_N(V_j')$. 
We denote by $\varphi_j$ and $\varphi_j'$ the inverses of $\pi_N|_{V_j}$ and $\pi_N|_{V_j'}$, respectively. 
Then, the system $\{(N_j, \varphi_j), (N_j', \varphi_j')\}_{j \le N}$ is a smooth atlas on $N$. 

Now we construct a homeomorphism $f : U(D(K)) \to N$.
Following \cite{O}, we define maps $f^{(j)} : B(\tilde x_j, s) \to E_j$ \,$(1\le j\le N)$ satisfying 
\beq \label{eq:cycleF}
  f^{(j)}=F_j^k\circ f^{(k)} \quad \text{on $\hat V_j\cap \hat V_k$},
\eeq
where $\hat V_j:=(f^{(j)})^{-1}(V_j)$.
This is achieved via : 
\begin{align*}
f^{(1)} &:= f_1 \\ 
f^{(2)}(x) &:= \psi_1 \circ f^{(1)}(x) \cdot \tilde F_2^1 \circ f^{(1)}(x) + (1-\psi_1 \circ f^{(1)}(x)) f_2(x).
\end{align*}
For the construction of $f^{(j)}$ for $j\ge 3$,
see \cite[p.1272]{O}.

Next we set 
\begin{align*}
f^{(j)}{}' &:= R_j \circ f^{(j)} \circ \phi : B(\phi \tilde x_j, s) \to E_j',\\
f^{(j)}{}'' &:= R_j \circ f^{(j)} : B(\tilde x_j, s) \to E_j',\\
f^{(j)}{}''' &:= f^{(j)} \circ \phi : B(\phi \tilde x_j, s) \to E_j,\\
\hat V_j' &:= \phi(\hat V_j).
\end{align*}
Then we have 
\begin{align*} \label{eq:cycleA}
&f^{(j)''}=\tilde G_j^k\circ f^{(k)}, \quad 
 f^{(j)}=\tilde H_j^k\circ f^{(k)} \quad 
 f^{(j)''}=\tilde A_j^k\circ f^{(k)''} \quad \text{on $\hat V_j\cap \hat V_k$}, \\
%%%%%%%%%%%%%%%%%%%%%%%%%%%%%%%%%%%
 & f^{(j)'''}=\tilde F_j^k\circ f^{(k)'''},\quad 
f^{(j)'}=\tilde G_j^k\circ f^{(k)'''} \quad\text{on $\hat V_j'\cap \hat V_k'$}, \\
& f^{(j)'''}=\tilde H_j^k\circ f^{(k)'}, \quad 
  f^{(j)'}=\tilde A_j^k\circ f^{(k)'} \quad \text{on $\hat V_j'\cap \hat V_k'$}. \\
 \end{align*}
It follows that 
the family $\{f^{(j)}, f^{(j)}{}', f^{(j)}{}'', f^{(j)}{}''' \}_{j=1}^N$ induces a natural map 
\[
\varphi : U(D(K)) \to N, 
\]
where $U(D(K))$ is the union of $\{\hat V_j, \hat V_j'\}_{j=1}^N$. 
From the construction, $\varphi$ is a $\theta(\epsilon)$-almost isometric $\mathbb Z_2$-equivariant $DC^1$-homeomorphism.

\vspace{1em}
\noindent
{\bf Step 4}\, (Construction of a flow). 
By an argument similar to the proof of \cite[Theorem 3.2]{MY}, we can construct a desired flow for $d_{D(S)}$ on $D(M)$ which is gradient-like on $U(D(K))$ as follows. 
First, we prepare a family of vector fields $W_j$ on $V_j$ which are almost gradient vectors for $d_{D(S)} \circ \varphi^{-1} \circ \pi_N$.
Note that such a family is taken to be $\tau$-invariant, that is, it is compatible with $\{ R_j \}$.
Next, we define a $\tau$-invariant vector field on $N$ which is close to $W_j$ on each $N_j$ and consider its integral curves. 
Finally, the pull-back of the integral curves gives rise to the desired flow $\Phi$ on $D(M)$.
This completes the proof of Theorem \ref{thm:smooth approximation}. 
\end{proof}

%%%%%%%%%%%%%%

%\section{Discussions}
%
%\begin{prop}
%Let $M_i$ collapse to $X$, where $\dim X = 2$. 
%Suppose $p_i \in M_i$ converges to an interior point of $X$. 
%\end{prop}


\begin{thebibliography}{99999}
\bibitem[BBI]{BBI} D.~Burago, Yu.~Burago, and S.~Ivanov.
A course in metric geometry.
Graduate Studies in Mathematics, 33. 
American Mathematical Society, Providence, RI, 2001. xiv+415 pp.
ISBN: 0-8218-2129-6

\bibitem[BGP]{BGP} Yu.~Burago, M.~Gromov, and G.~Perel'man.
A. D. Aleksandrov spaces with curvatures bounded below, 
Uspekhi Mat. Nauk 47 (1992), no. 2(284), 3--51, 222, 
translation in Russian Math. Surveys 47 (1992), no. 2, 1--58


%\bibitem[BH]{BH}M.~Bridson and A.~Haefliger, Metric spaces of non-positive curvature, vol. 319 of Grundlehren der Mathematischen Wissenschaften, Springer-Verlag, 1999.

%\bibitem[BLP]{BLP} M.~Boileau, B.~Leeb, and J.~Porti.
%Geometrization of 3-dimensional orbifolds, Ann. of Math., 162 (2005), 195--290

%\bibitem[CaGe]{CaGe} J.~Cao and J.~Ge.
%A simple proof of Perelman's collapsing theorem for 3-manifolds.
%J. Geom. Anal. Volume 21, Number 4, 807--869, DOI: 10.1007/s12220-010-9169-5

\bibitem[CG]{CG} J.~Cheeger and D.~Gromoll. 
On the structure of complete manifolds of nonnegative
curvature, Ann. of Math. 96 (1972) 413--443.

%\bibitem[E]{Epstein} D.~B.~A.~Epstein. Curves on 2-manifolds and isotopies. Acta Math. 115 (1966) 83--107.

\bibitem[Fj]{Fj} T.~Fujioka.
Uniform boundedness on extremal subsets in Alexandrov spaces
arXiv:1809.00603

\bibitem[F]{F} K.~Fukaya, Theory of convergence for Riemannian orbifolds.
Japan.~J.~Math. Vol 12, No 1, 1986, 121--160.

\bibitem[FY]{FY} K.~Fukaya and T.~Yamaguchi. The fundamental groups of almost nonnegatively curved manifolds, Ann. of Math. 136 (1992) 253--333.

%\bibitem[FY2]{FY isom} Fukaya and Yamaguchi, Isometry group....

%\bibitem[GGG]{GGG} F.~Garaz-Garcia and L.~Guijarro. On three-dimensional Alexandrov spaces...

%\bibitem[GP]{GP} K.~Grove and P.~Petersen. A radius sphere theorem, Invent. Math. 112 (1993), no. 3, 577--83. 

%\bibitem[Kap1]{Kap:restr} V.~Kapovitch. Restrictions on collapsing with a lower sectional curvature bound. Math. Z. 249 (2005), no. 3, 519--539.


\bibitem[GGZ]{GGZ} F. Galaz-Garcia, L. Guijarro, J. Núñez-Zimbrón, Sufficiently collapsed irreducible Alexandrov 3-spaces are geometric,
Indiana Univ. Math. J. 69 (2020), (3), 977–1005

\bibitem[GP]{GP} K. Grove and P. Petersen.
A radius sphere theorem, 
 Inventiones mathematicae 112.3 (1993), 577-584.

\bibitem[Iv]{Iv} N. V. Ivanov, Mapping class groups, Handbook of geometric topology,
North Holland, 
Amsterdam, 2002, pp. 523–633. 



\bibitem[K]{Kap:stab} V.~Kapovitch. Perelman's stability theorem,
Surveys of Differential Geometry, Metric and Comparison Geometry, vol. XI, 
International press, (2007), 103--136.

\bibitem[KPT]{KPT} V.~Kapovitch, A.~Petrunin, W.~Tuschmann.
Nilpotency, almost nonnegative curvature, and the gradient flow on Alexandrov spaces,
Ann. of Math. 171 (2010), 343-373. 

\bibitem[KW]{KW}
V.~Kapovitch, B.~Wilking.
Structure of fundamental groups of manifolds with Ricci curvature bounded below, arXiv:1105.5955


%\bibitem[KL]{KL} B.~Kleiner and J.~Lott. 
%Locally collapsed 3-manifolds.....
%
%\bibitem[KL orb]{KL orb} B.~Kleiner and J.~Lott. 
%Geometrization of three-dimensional orbifolds via Ricci flow....

\bibitem[KMS]{KMS} K.~Kuwae, Y.~Machigashira and T.~Shioya.
Sobolev spaces, Laplacian, and heat kernel on Alexandrov spaces, 
Math. Z. 238 (2001), no. 2, 269--316.

%\bibitem[L]{Lyt}A.~Lytchak. 
%Differentiation in metric spaces. 
%Algebra i Analiz 16 (2004), no. 6, 128--161;
%translation in St. Petersburg Math. J. 16 (2005), no. 6, 1017--1041 

%\bibitem[Milka]{Milka} A.~D.~Milka. 
%Metric structure of a certain class of spaces that contain straight lines. (Russian) 
%Ukrain. Geometr. Sb. Vyp. 4 1967 43--48. 

%\bibitem[M]{Mitsuishi} A.~Mitsuishi. 
%A splitting theorem for infinite dimensional Alexandrov spaces with nonnegative curvature and its applications. 
%Geom. Dedicata 144 (2010), 101--114.

\bibitem[M]{Mit:gluing} A.~Mitsuishi.
Self and partial gluing theorems for Alexandrov spaces with a lower curvature bound, arXiv:1606.02578


\bibitem[MY]{MY} A.~Mitsuishi and T.~Yamaguchi. Collapsing three-dimensional closed Alexandrov spaces with a lower curvature bound, 
Trans.~Amer.~Math.~Soc. 367 (2015), 2339--2410. 


\bibitem[O]{O} Y.~Otsu. On manifolds of small excess, 
Amer. J. Math. 115 (1993), no. 6, 1229--1280.

\bibitem[OS]{OS} Y.~Otsu and T.~Shioya.
The Riemannian structure of Alexandrov spaces, 
J. Differential Geom. 39 (1994), no. 3, 629--658

\bibitem[P1]{PerAlex2} G.~Perelman. 
A. D. Alexandrov's spaces with curvatures bounded from below. II, Preprint.

\bibitem[P2]{Per:elem} G.~Perelman. 
Elements of Morse theory on Alexandrov spaces, 
St.~Petersburg Math. J. 5 (1994) 207--214.



\bibitem[P3]{Per DC} G.~Perelman.
DC Structure on Alexandrov Space. Preprint.


\bibitem[PP1]{PP ext} G.~Perelman and A.~Petrunin.
Extremal subsets in Aleksandrov spaces and the generalized Liberman theorem. 
(Russian) Algebra i Analiz 5 (1993), no. 1, 242--256; 
translation in St. Petersburg Math. J. 5 (1994), no. 1, 215--227

\bibitem[PP2]{PP QG} G.~Perelman and A.~Petrunin.
Quasigeodesics and gradient curves in Alexandrov spaces, preprint.
%
%\bibitem[Pet QG]{Pet QG} A.~Petrunin. 
%Quasigeodesics in multidimensional Alexandrov spaces.
%Thesis (Ph.D.)-University of Illinois at Urbana-Champaign. 1995. 92 pp.

%\bibitem[Pet Appl]{Pet Appl} A.~Petrunin.
%Applications of quasigeodesics and gradient curves. 
%(English summary) Comparison geometry (Berkeley, CA, 1993--94), 203--219,
%Math. Sci. Res. Inst. Publ., 30, Cambridge Univ. Press, Cambridge, 1997. 

%\bibitem[Pet Para]{Pet parallel} A.~Petrunin. 
%Parallel transportation for Alexandrov space with curvature bounded below. 
%Geom. Funct. Anal. 8 (1998), no. 1, 123--148.

\bibitem[Pt]{Pet Semi} A.~Petrunin. 
Semiconcave functions in Alexandrov's geometry. 
Surv. Differ. Geom., Vol. XI, 137--201, Int. Press, Somerville, MA, 2007.


%\bibitem[Pl]{Plaut} C.~Plaut. 
%Spaces of Wald-Berstovskii curvature bounded below, J. Geom. Anal. 6, 113--134 (1996)

\bibitem[SY]{SY} T.~Shioya and T.~Yamaguchi. 
Collapsing three-manifolds under a lower curvature bound. 
J. Differential Geom. 56, 1--66 (2000)

\bibitem[XY]{XY} S.~Xu, X.~Yao.
Margulis lemma and Hurewicz fibration Theorem on Alexandrov spaces,
Communications in Contemporary Mathematics, 24 (2022)




%\bibitem[SY05]{SY vol collapse}
%T.~Shioya and T.~Yamaguchi. 
%Volume collapsed three-manifolds with a lower curvature bound.
%Math. Ann., 333 (1), (2005), 131--155.

%\bibitem[Sie]{Sie}L.~C.~Seibenmann. 
%Deformation of homeomorphisms on stratified sets. I, II. 
%Comment. Math. Helv. 47 (1972), 123--136; ibid. 47 (1972), 137--163.

%\bibitem[W]{Wo} J. A. Wolf, 
%Spaces of constant curvature, 
%Publish or Pelish, 1984.

\bibitem[Y1]{Yam collapsing and pinching} T.~Yamaguchi.
Collapsing and Pinching Under a Lower Curvature Bound. 
Ann. of Math. 133 (1991) 317--357.

\bibitem[Y2]{Y conv} 
T.~Yamaguchi.
A convergence theorem in the geometry of Alexandrov spaces. 
Actes de la Table Ronde de Geometrie Differentielle (Luminy, 1992), 
Semin. Congr., vol. 1, Soc. Math. France, Paris, 1996, pp. 601--642

\bibitem[Y3]{Y 4-dim} T.~Yamaguchi, 
Collapsing 4-manifolds under a lower curvature bound. 
arXiv:1205.0323v1 [math.DG]

\bibitem[Y4]{Y ess} T.~Yamaguchi. Collapsing and essential covering. 
arXiv:1205.0441v1 [math.DG]

\bibitem[YZ1]{YZ} T.~Yamaguchi and Z.~Zhang
Inradius collapsed manifolds,
Geom. Topol. 23 (2019) 2793–2860

\bibitem[YZ2]{YZ2} T.~Yamaguchi and Z.~Zhang
Limits of manifolds with boundary, in preparation.
\end{thebibliography}
\end{document}